\documentclass{amsart}
\usepackage{amsthm}

\usepackage[latin1]{inputenc}
\usepackage{subfigure}
\usepackage{color}
\usepackage{amsmath}
\usepackage{amsthm}
\usepackage{amstext}
\usepackage{amssymb}
\usepackage{amsfonts}
\usepackage{graphicx}
\usepackage{young}
\usepackage{multicol}
\usepackage{mathrsfs}
\usepackage{stmaryrd}
\usepackage{bbm}
\usepackage[all]{xy}
\usepackage{mathabx}
\usepackage{tikz}
\usepackage{mathtools}
\usepackage{tabularx}
\usepackage{array}
\usepackage{commath}

\usepackage{caption}

\theoremstyle{plain}
\theoremstyle{definition}
\newtheorem{theorem}{Theorem}[section]

\newtheorem{remark}[theorem]{Remark}
\newtheorem{lemma}[theorem]{Lemma}
\newtheorem{definition}[theorem]{Definition}

\newtheorem{question}[theorem]{Question}
\newtheorem{example}[theorem]{Example}
\newtheorem{proposition}[theorem]{Proposition}
\newtheorem{corollary}[theorem]{Corollary}

\DeclareMathAlphabet{\mathpzc}{OT1}{pzc}{m}{it}

\usepackage{amssymb,amsmath,tabularx,graphicx}
\usepackage{tikz}
\usetikzlibrary[calc,intersections,through,backgrounds,arrows,decorations.pathmorphing]

\def\la{\leftarrow}
\def\ra{\rightarrow}
\def\pr{\prime}
\def\ov{\overline}
\def\setm{\setminus}

\begin{document}

\title{Minimal Length Maximal Green Sequences}
\date{}
\author{Alexander Garver}
\address{Laboratoire de Combinatoire et d'Informatique Math\'ematique,
Universit\'e du Qu\'ebec \`a Montr\'eal}
\email{alexander.garver@lacim.ca}

\author{Thomas McConville}
\address{Department of Mathematics,
Massachusetts Institute of Technology}
\email{thomasmc@mit.edu}

\author{Khrystyna Serhiyenko}
\address{Department of Mathematics, 
University of California, Berkeley}
\email{khrystyna.serhiyenko@berkeley.edu}

\maketitle

\begin{abstract}
Maximal green sequences are important objects in representation theory, cluster algebras, and string theory. It is an open problem to determine what lengths are achieved by the maximal green sequences of a quiver. We combine the combinatorics of surface triangulations and the basics of scattering diagrams to address this problem. Our main result is a formula for the length of minimal length maximal green sequences of quivers defined by triangulations of an annulus or a punctured disk.
\end{abstract}

\tableofcontents

\section{Introduction}

A maximal green sequence is a distinguished sequence of local transformations, known as \textbf{mutations}, of a given \textbf{quiver} (i.e., directed graph). Maximal green sequences were introduced by Keller in \cite{keller2011cluster} in order to obtain combinatorial formulas for the refined Donaldson-Thomas invariants of Kontsevich and Soibelman \cite{KS}. They are also important in string theory \cite{alim2013bps}, representation theory \cite{bdp, by14}, and cluster algebras \cite{fomin2002cluster}. 

Recently, there have been many developments on the combinatorics of maximal green sequences (see \cite{mills2016maximal} and references therein). In particular, in \cite{mills2016maximal} it is shown that almost any quiver $Q$ with a finite \textbf{mutation class} (i.e., the set of quivers obtained from $Q$ by mutations) has a maximal green sequence. Our goal is to add to the known combinatorics by developing a numerical invariant of the set of maximal green sequences of $Q$: the length of minimal length maximal green sequences.

This invariant is natural from the perspective of cluster algebras. Fixing a quiver $Q$ induces an orientation of the edges of the corresponding exchange graph. It turns out that the maximal green sequences of $Q$ are in natural bijection with finite length maximal directed paths of the resulting oriented exchange graph (see \cite{bdp}). Examples of oriented exchange graphs include the Hasse diagrams of Tamari lattices and of Cambrian lattices of type $\mathbb{A}$, $\mathbb{D}$, and $\mathbb{E}$ \cite{ReadingCamb}. In these examples, the minimal length maximal green sequences always have length equal to the number of vertices of $Q$, but, in general, the minimal length of a maximal green sequence may be larger than the number of vertices of $Q$. Thus understanding the length of minimal length maximal green sequences provides new information about oriented exchange graphs.

Our main results (see Theorems~\ref{Thm:typeD} and \ref{Thm_affinequivers_min_length}) are formulas for this minimal length when $Q$ is of \textbf{mutation type} $\mathbb{D}_n$ or \textbf{mutation type} $\widetilde{\mathbb{A}}_{n}$ (i.e., $Q$ is in the mutation class of a type $\mathbb{D}_n$ or of an affine type $\mathbb{A}_{n}$ quiver). This number was calculated in mutation type $\mathbb{A}$ in \cite{cormier2015minimal}. We also obtain explicit constructions of minimal length maximal green sequences of $Q$, in the process of proving Theorems~\ref{Thm:typeD} and \ref{Thm_affinequivers_min_length}.


We also suspect that the length of minimal length maximal green sequences may be significant from the perspective of representation theory of algebras. More specifically, given two quivers that define derived equivalent cluster-tilted algebras, it appears that they will have minimal length maximal green sequences of the same length (see Section~\ref{Sec_conc} for more details). 

The paper is organized as follows. In Section~\ref{Sec_MGSs_prelim}, we review the basics of quiver mutation, maximal green sequences, and oriented exchange graphs. We also recall the definitions of $\textbf{c}$- and $\textbf{g}$-vectors, and the result that maximal green sequences are in bijection with certain sequences of $\textbf{c}$-vectors. In Section~\ref{Sec_scatter_diag}, we present the basics of scattering diagrams. We use these to reformulate a result of Muller (see Theorem~\ref{thmMullerSubquiver}) that shows that a maximal green sequence of a quiver determines a maximal green sequence of any full subquiver by removing from the corresponding sequence of $\textbf{c}$-vectors all those that are not supported on the full subquiver.

In Section~\ref{sec:direct sums}, we use Theorem~\ref{thmMullerSubquiver} to show that the minimal length of maximal green sequences of the direct sum of two quivers $Q^{1}$ and $Q^{2}$ is the minimal length of maximal green sequences of $Q^{1}$ plus the minimal length of maximal green sequences of $Q^{2}$. In Section~\ref{Sec_branches}, we use Theorem~\ref{thmMullerSubquiver} to show that if a quiver $\widetilde{Q}$ is obtained from another quiver $Q$ by ``attaching'' several mutation type $\mathbb{A}$ quivers to $Q$, then the problem of finding the length of minimal length maximal green sequences of $\widetilde{Q}$ reduces to finding the length of minimal length maximal green sequences of $Q$ (see Corollary~\ref{Cor:min_length_Qtilde}).

In Sections~\ref{Sec_type_D} and \ref{Sec_type_IV}, we apply our results to quivers of mutation type $\mathbb{D}_n$ and calculate the minimal length of maximal green sequences in Theorem~\ref{Thm:typeD}.    Here, we also use combinatorics of triangulated surfaces in the sense of \cite{fomin2008cluster} and \cite{fomin2012cluster} which we review in Section~\ref{Sec:QT}.   In particular, we recall how maximal green sequences can be interpreted geometrically using shear coordinates.   In Section~\ref{Sec_affine_A_quivers}, we do likewise for quivers of mutation type $\widetilde{\mathbb{A}}_{n}$ (see Theorem~\ref{Thm_affinequivers_min_length}). We remark that Corollary~\ref{Cor:min_length_Qtilde} plays in important role in these sections.


{\bf Acknowledgements.~} Alexander Garver thanks Greg Muller and Rebecca Patrias for useful conversations. At various stages of this project, Alexander Garver received support from an RTG grant DMS-1148634, NSERC, and the Canada Research Chairs program. Khrystyna Serhiyenko was supported by the NSF Postdoctoral Fellowship MSPRF-1502881.

\section{Maximal green sequences}\label{Sec_MGSs_prelim}

\subsection{Quivers and quivers mutation}

A \textbf{quiver} $Q$ is a directed graph. In other words, $Q$ is a 4-tuple $(Q_0,Q_1,s,t)$, where $Q_0 = [m] := \{1,2, \ldots, m\}$ is a set of \textbf{vertices}, $Q_1$ is a set of \textbf{arrows}, and two functions $s, t:Q_1 \to Q_0$ defined so that for every $\alpha \in Q_1$, we have $s(\alpha) \xrightarrow{\alpha} t(\alpha)$. An \textbf{ice quiver} is a pair $(Q,F)$ with $Q$ a quiver and $F \subset Q_0$ a set of \textbf{frozen vertices} with the restriction that any $i,j \in F$ have no arrows of $Q$ connecting them. By convention, we assume $Q_0\backslash F = [n]$ and $F = [n+1,m] := \{n+1, n+2, \ldots, m\}.$ We refer to elements of $Q_0\backslash F$ as \textbf{mutable vertices}. Any quiver $Q$ is regarded as an ice quiver by setting $Q = (Q, \emptyset)$.

If a given ice quiver $(Q,F)$ is \textbf{2-acyclic} (i.e., $Q$ has no loops or 2-cycles), we can define a local transformation of $(Q,F)$ called \textbf{mutation}. The {\bf mutation} of an ice quiver $(Q,F)$ at a mutable vertex $k$, denoted $\mu_k$, produces a new ice quiver $(\mu_kQ,F)$ by the three step process:

(1) For every $2$-path $i \to k \to j$ in $Q$, adjoin a new arrow $i \to j$.

(2) Reverse the direction of all arrows incident to $k$ in $Q$.

(3) Remove any $2$-cycles created, and remove any arrows created that connect two frozen vertices.

\noindent From now on, we will only work with 2-acyclic quivers. We show an example of mutation below with the mutable (resp., frozen) vertices in black (resp., blue).

\vspace{-.1in}
\[
\begin{array}{c c c c c c c c c}
\raisebox{-.4in}{$(Q,F)$} & \raisebox{-.4in}{=} & {\begin{xy} 0;<1pt,0pt>:<0pt,-1pt>:: 
(0,30) *+{1} ="0",
(40,0) *+{2} ="1",
(80,30) *+{3} ="2",
(40,60) *+{\textcolor{blue}{4}} ="3",
"0", {\ar@<-.5ex>"1"},
"0", {\ar@<.5ex>"1"},
"1", {\ar"2"},
"3", {\ar"2"},
\end{xy}} & \raisebox{-.4in}{$\stackrel{\mu_2}{\longmapsto}$} & {\begin{xy} 0;<1pt,0pt>:<0pt,-1pt>:: 
(0,30) *+{1} ="0",
(40,0) *+{2} ="1",
(80,30) *+{3} ="2",
(40,60) *+{\textcolor{blue}{4}} ="3",
"2", {\ar"1"},
"0", {\ar@<-.5ex>"2"},
"0", {\ar@<.5ex>"2"},
"1", {\ar@<-.5ex>"0"},
"1", {\ar@<.5ex>"0"},
"3", {\ar"2"},
\end{xy}} & \raisebox{-.4in}{=} & \raisebox{-.4in}{$(\mu_2Q,F)$}
\end{array}
\]

The information of an ice quiver can be equivalently described by its (skew-symmetric) \textbf{exchange matrix}. Given $(Q,F),$ we define $B = B_{(Q,F)} = (b_{ij}) \in \mathbb{Z}^{n\times m} := \{n \times m \text{ integer matrices}\}$ by $b_{ij} := |\{i \stackrel{\alpha}{\to} j \in Q_1\}| - |\{j \stackrel{\alpha}{\to} i \in Q_1\}|.$ Furthermore, ice quiver mutation can equivalently be defined  as \textbf{matrix mutation} of the corresponding exchange matrix. Given an exchange matrix $B \in \mathbb{Z}^{n\times m}$, the \textbf{mutation} of $B$ at $k \in [n]$, also denoted $\mu_k$, produces a new exchange matrix $\mu_k(B) = (b^\prime_{ij})$ with entries
\[
b^\prime_{ij} := \left\{\begin{array}{ccl}
-b_{ij} & : & \text{if $i=k$ or $j=k$} \\
b_{ij} + \frac{|b_{ik}|b_{kj}+ b_{ik}|b_{kj}|}{2} & : & \text{otherwise.}
\end{array}\right.
\]
For example, the mutation of the ice quiver above (here $m=4$ and $n=3$) translates into the following matrix mutation. Note that mutation of matrices and of ice quivers is an involution (i.e., $\mu_k\mu_k(B) = B$). 
\[
\begin{array}{c c c c c c c c c c}
B_{(Q,F)} & = & \left[\begin{array}{c c c | r}
0 & 2 & 0 & 0 \\
-2 & 0 & 1 & 0\\
0 & -1 & 0 & -1\\
\end{array}\right]
& \stackrel{\mu_2}{\longmapsto} &
\left[\begin{array}{c c c | r}
0 & -2 & 2 & 0 \\
2 & 0 & -1 & 0\\
-2 & 1 & 0 & -1\\
\end{array}\right] 
& = & B_{(\mu_2Q,F)}.
\end{array}
\]

Let Mut($(Q,F)$) denote the collection of ice quivers obtainable from $(Q,F)$ by finitely many mutations where such ice quivers are considered up to an isomorphism of quivers that fixes the frozen vertices. We will refer to Mut($(Q,F)$) as the \textbf{mutation class} of $(Q,F)$. Such an isomorphism is equivalent to a simultaneous permutation of the rows and first $n$ columns of the corresponding exchange matrices.

For our purposes, it will be useful to recall the following classification of \textbf{mutation type $\mathbb{A}_n$ quivers} (i.e., quivers $R \in \text{Mut}(1 \leftarrow 2 \leftarrow \cdots \leftarrow n)$) due to Buan and Vatne.

\begin{lemma}\cite[Prop. 2.4]{bv08}\label{BV}
A connected quiver $Q$ with $n$ vertices is of mutation type $\mathbb{A}_n$ if and only if $Q$ satisfies the following:
\begin{itemize}
      \item[i)] All non-trivial cycles in the underlying graph of $Q$ are oriented and of length 3.
      \item[ii)] Any vertex has degree at most 4.
      \item[iii)] If a vertex has degree 4, then two of its adjacent arrows belong to one 3-cycle, and the other two belong to another 3-cycle.
      \item[iv)] If a vertex has degree 3, then two of its adjacent arrows belong to a 3-cycle, and the third arrow does not belong to any 3-cycle.
     \end{itemize} 
\end{lemma}

\noindent At times we will also say that a quiver $Q$ is of mutation type $\mathbb{A}$, meaning that $Q$ is of mutation type $\mathbb{A}_{|Q_0|}$.

\subsection{Maximal green sequences and oriented exchange graphs}


In this section, we review the notions of maximal green sequences and oriented exchange graphs. 

Given a quiver $Q$, we define its \textbf{framed} (resp., \textbf{coframed}) quiver to be the ice quiver $\widehat{Q}$ (resp., $\widecheck{Q}$)  where $\widehat{Q}_0 := Q_0 \sqcup [n+1, 2n]$ (resp., $\widecheck{Q}_0 := Q_0 \sqcup [n+1, 2n]$), $F = [n+1, 2n]$, and $\widehat{Q}_1 := Q_1 \sqcup \{i \to n+i: i \in [n]\}$ (resp., $\widecheck{Q}_1 := Q_1 \sqcup \{n+i \to i: i \in [n]\}$). We denote elements of $\text{Mut}(\widehat{Q})$ by $\overline{Q}$, and we will henceforth write $i^\prime = n+i$ where $i \in [n]$. We say that a mutable vertex $i$ of $\overline{Q}$ is \textbf{green} (resp., \textbf{red}) if there are no arrows in $\overline{Q}$ of the form $i \leftarrow j^\prime$ (resp., $i \rightarrow j^\prime$) for some $j \in [n]$. The celebrated theorem of Derksen, Weyman, and Zelevinsky \cite[Theorem 1.7]{dwz10}, known as sign-coherence of $\textbf{c}$-vectors and $\textbf{g}$-vectors, implies that given $\overline{Q} \in \text{Mut}(\widehat{Q})$ any mutable vertex of $\overline{Q}$ is either green or red.

\begin{definition}[\cite{keller2011cluster}]\label{mgsdef}A \textbf{maximal green sequence} of $Q$ is a sequence $\textbf{i} = (i_1, \ldots, i_k)$ of mutable vertices of $\widehat{Q}$ where 

$\begin{array}{rl}
i) & \text{for all $j \in [k]$ vertex $i_j \in [n]$ is green in $\mu_{i_{j-1}}\circ \cdots\circ \mu_{i_1}(\widehat{Q})$, and}\\ 
ii) & \text{each vertex $i \in [n]$ of $\mu_{i_{k}}\circ \cdots\circ \mu_{i_1}(\widehat{Q})$ is red}.\end{array}$ 

\noindent We let $\text{MGS}(Q)$ denote the set of maximal green sequences of $Q$. {At times, we will use the abuse of language where we refer to $\underline{\mu}_\textbf{i}:= \mu_{i_k}\mu_{i_{k-1}}\cdots\mu_{i_1}$ also as a maximal green sequence.} Additionally, we define $\ell(\textbf{i}) := k$ and $\ell(\underline{\mu}_{\textbf{i}}) := k$ to be the \textbf{length} of the maximal green sequence $\textbf{i}$. \end{definition}

For a given quiver $Q$, it is an open question to determine what positive integers can be realized as lengths of maximal green sequences of $Q$. In \cite[Lemma 2.20]{bdp}, it is shown that if $Q$ is acyclic, then $Q$ has a maximal green sequence of length $|Q_0|$ and this is the shortest any maximal green sequence of a quiver can be. The following theorem is the first to address this question for an infinite family of quivers in which oriented cycles may appear.

\begin{theorem}\label{thmtypeA} \cite[Theorem 6.5, Theorem 7.2]{cormier2015minimal}
The length of a minimal length maximal green sequence of a mutation type $\mathbb{A}_n$ quiver $Q$ is $|Q_0| + |\{\text{3-cycles of $Q$}\}|.$
\end{theorem}

The maximal green sequences of $Q$ are in natural bijection with the finite length maximal directed paths in the \textbf{oriented exchange graph} of $Q$ \cite[Proposition 2.13]{bdp}, which we now define. The \textbf{exchange graph} of $\widehat{Q}$, denoted $EG(\widehat{Q})$, is the (a priori infinite) graph whose vertices are elements of Mut$(\widehat{Q})$  and two vertices are connected by an edge if the corresponding quivers differ by a single mutation. The \textbf{oriented exchange graph} of $Q$, denoted $\overrightarrow{EG}(\widehat{Q})$, is the directed graph whose vertex set is $\text{Mut}(\widehat{Q})$ and whose edges are of the form $\overline{Q} \longrightarrow \mu_{j}\overline{Q}$ where $j \in [n]$ is green in $\overline{Q}$. Oriented exchange graphs were introduced in \cite{bdp} where they initiated the study of maximal green sequences. We show the oriented exchange graph of $Q = 1 \to 2$ in Figure~\ref{A2oreg}. 

\begin{figure}[h]
$$\begin{array}{cccccl}
\includegraphics[scale=1]{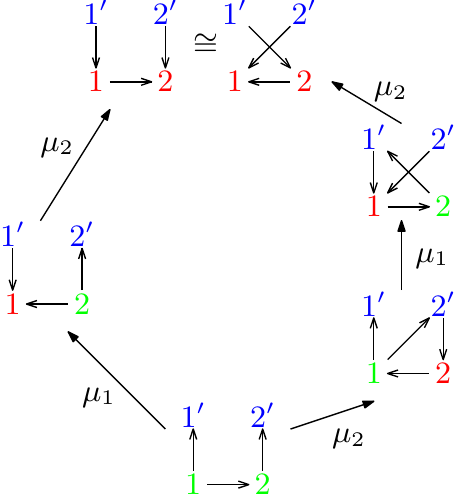} & \raisebox{.85in}{$\cong$} & \includegraphics[scale=1]{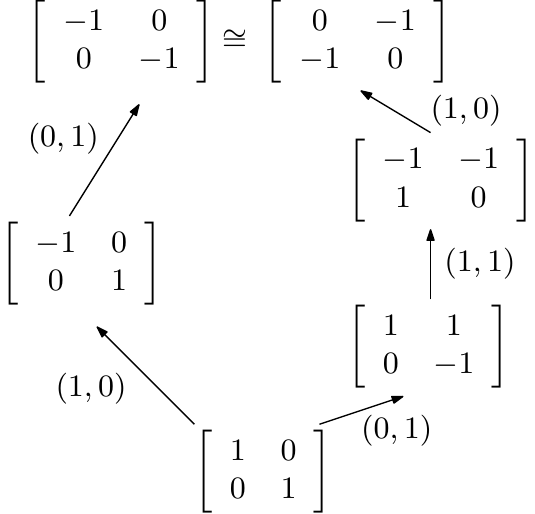} & \raisebox{.85in}{$\cong$} & \includegraphics[scale=1]{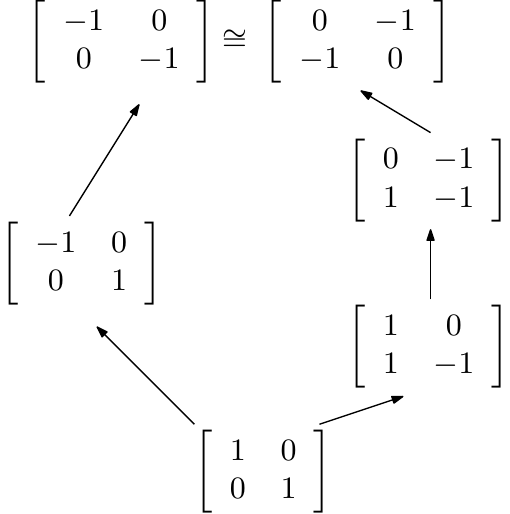}\\ \\ \overrightarrow{EG}(\widehat{Q}) & & \textbf{c}\text{-mat}(Q) & & \textbf{g}\text{-mat}(Q)
\end{array}$$
\caption{The oriented exchange graph of $Q=1\rightarrow 2$ and the directed graphs $\textbf{c}\text{-mat}(Q)$ and $\textbf{g}\text{-mat}(Q)$.}
\label{A2oreg}
\end{figure}


The oriented exchange graph of a quiver $Q$ can be equivalently described using the \textbf{c}-vectors and \textbf{c}-matrices of $Q$. We say that $C = C_{\overline{Q}} \in \mathbb{Z}^n$ is a \textbf{c}-\textbf{matrix} of $Q$ if there exists $\overline{Q} \in EG(\widehat{Q})$ such that $C$ is the $n\times n$ submatrix of $B_{\overline{Q}} = (b_{ij})_{i \in [n], j \in [2n]}$ containing its last $n$ columns. That is, $C = (b_{ij})_{i \in [n], j \in [n+1,2n]}$. We let \textbf{c}-mat($Q$) $:= \{C_{\overline{Q}}: \overline{Q} \in EG(\widehat{Q})\}$. A row vector of a \textbf{c}-matrix is known as a \textbf{c}-\textbf{vector}. Since a \textbf{c}-matrix $C$ is only defined up to a permutations of its rows, $C$ can be regarded simply as a \textit{set} of \textbf{c}-vectors.

Sign coherence of \textbf{c}-vectors, says that any \textbf{c}-vector is either a nonzero element of $\mathbb{Z}^n_{\ge 0}$ or of $\mathbb{Z}^n_{\le 0}$. In the former case, we say a \textbf{c}-vector is \textbf{positive}. In the latter case, we say a \textbf{c}-vector is \textbf{negative}. From our definition of the exchange matrix $\overline{Q}$, we have that $\textbf{c}_k \in C_{\overline{Q}}$ (the \textbf{c}-vector of vertex $k$ in $\overline{Q}$) is positive (resp., negative) if and only if $k$ is green (resp., red) in $\overline{Q}$.

The set $\textbf{c}\text{-mat}(Q)$ can be regarded as a directed graph whose vertices are \textbf{c}-matrices and whose directed edges are exactly those of the form $C_{\overline{Q}} \rightarrow C_{\mu_k\overline{Q}}$ where $\textbf{c}_k \in C_{\overline{Q}}$  is positive. Now by regarding $\textbf{c}\text{-mat}(Q)$ as a directed graph, it follows from \cite{by14} that $\overrightarrow{EG}(\widehat{Q}) \cong \textbf{c}\text{-mat}(Q)$ and maximal directed paths in each are in natural bijection.\footnote{One can define a notion of \textbf{c}-matrix mutation (for example, see \cite[page 35]{by14}). However, we omit this definition because we will not need it.} Moreover, each directed edge $C_{\overline{Q}} \rightarrow C_{\mu_k\overline{Q}}$ of $\textbf{c}\text{-mat}(Q)$ is labeled by the \textbf{c}-vector $\textbf{c}_k$. We obtain the following lemma.

\begin{lemma}\label{cvecs_of_mgs}
A maximal green sequence $\textbf{i} = (i_1, \ldots, i_k)$ of a quiver $Q$ is equivalent to a maximal directed path $\textbf{c}(\textbf{i}) = (\textbf{c}_{i_1}, \ldots, \textbf{c}_{i_k})$ in $\textbf{c}\text{-mat}(Q)$ where $\textbf{c}_{i_j} \in C_{\mu_{i_{j-1}}\circ \cdots \circ \mu_{i_1}(\widehat{Q})}$ is the \textbf{c}-vector corresponding to vertex $i_j \in (\mu_{i_{j-1}}\circ \cdots \circ \mu_{i_1}(\widehat{Q}))_0$.
\end{lemma}

We conclude this section by presenting yet another equivalent description of the oriented exchange graph of $Q$, which uses \textbf{g}-vectors and \textbf{g}-matrices (see \cite{fomin2007cluster}). The definition of a \textbf{g}-vector $\textbf{g} \in \mathbb{Z}^{n}$ (resp., \textbf{g}-matrix $G \in \mathbb{Z}^{n\times n}$) was originally made for each cluster variable (resp., seed) of a cluster algebra with principal coefficients $\mathcal{A}$ where $n$ is the rank of the cluster algebra. The \textbf{g}-vectors of $\mathcal{A}$ provide a $\mathbb{Z}^n$-grading of the cluster variables by \cite[Proposition 6.1]{fomin2007cluster}. 

In this paper, we define \textbf{g}-matrices and \textbf{g}-vectors using quiver mutation. We let the identity matrix $I_n \in \mathbb{Z}^{n\times n}$ be the \textbf{g-matrix} associated to $\widehat{Q}$. The rows vectors $\{\textbf{g}_1, \ldots, \textbf{g}_n\}$ of $I_n$ are the \textbf{g-vectors} of $I_n$ where $\textbf{g}_i$ corresponds to mutable vertex $i$ in $\widehat{Q}$. Here $\textbf{g}_i$ is the $i$th standard basis vector. Each $\overline{Q} \in \overrightarrow{EG}(\widehat{Q})$ will have an associated \textbf{g-matrix} whose row vectors are its \textbf{g-vectors}. As is the case with \textbf{c}-matrices, the \textbf{g}-matrices are only defined up to row permutations and thus can be regarded simply as sets of \textbf{g}-vectors.

We produce the \textbf{g}-matrix associated to $\overline{Q} \in \overrightarrow{EG}(\widehat{Q})$ by defining a mutation operation on \textbf{g}-matrices. Let $G = G_{\overline{Q}}$ be the \textbf{g}-matrix of $\overline{Q}$ whose \textbf{g}-vectors are $\{\textbf{g}_1,\ldots, \textbf{g}_n\}$. Then the \textbf{g}-matrix of $\mu_k\overline{Q}$, denoted $\mu_k(G)$, is defined to the be matrix whose row vectors are $\{\textbf{g}_1, \ldots, \mu_k(\textbf{g}_k), \ldots, \textbf{g}_n\}$ where $$\begin{array}{rcl} \mu_k(\textbf{g}_k) & := & \left\{\begin{array}{rcl} -\textbf{g}_k + \displaystyle\sum_{(j \stackrel{\alpha}{\rightarrow} k) \in \overline{Q}_1} \textbf{g}_j & : & \text{$k$ is green in $\overline{Q}$}\\ -\textbf{g}_k + \displaystyle\sum_{(j \stackrel{\alpha}{\leftarrow} k) \in \overline{Q}_1} \textbf{g}_j & : & \text{$k$ is red in $\overline{Q}$.} \end{array}\right. \end{array}$$ We let $\textbf{g}\text{-mat}(Q)$ denote the set of \textbf{g}-matrices of $Q$. In Figure~\ref{A2oreg}, we show the \textbf{g}-matrices of $Q = 1 \to 2$. Perhaps surprisingly, the \textbf{c}-matrices and \textbf{g}-matrices of a quiver can be obtained from each other by a simple bijection, as the following theorem shows.

\begin{theorem}(\cite[Theorem 1.2]{nakanishi2012tropical},\cite[Theorem 4.17]{by14})\label{thm_tropical_dual}
The map $$\begin{array}{rcl} \textbf{c}\text{-mat}(Q) & \longrightarrow & \textbf{g}\text{-mat}(Q)\\ C & \longmapsto & (C^{-1})^T \\ (G^{-1})^T & \longmapsfrom & G \end{array}$$ sending a \textbf{c}-matrix or \textbf{g}-matrix to its inverse transpose is a bijection, and this map commutes with mutation of \textbf{c}-matrices and \textbf{g}-matrices of $Q$.
\end{theorem}

In particular, Theorem~\ref{thm_tropical_dual} shows that one can transfer the edge orientation of $\textbf{c}\text{-mat}(Q)$ to $\textbf{g}\text{-mat}(Q)$. Thus, one can regard $\textbf{g}\text{-mat}(Q)$ as a directed graph whose vertices are \textbf{g}-matrices and whose directed edges are exactly those of the form $G \to \mu_kG$ where $k$ is green in $\overline{Q}$ and $G = G_{\overline{Q}}$. In Figure~\ref{A2oreg}, we show $\textbf{g}\text{-mat}(Q)$ with $Q = 1\to 2$.

\section{Scattering diagrams}\label{Sec_scatter_diag}

In \cite{gross2014canonical}, scattering diagrams are shown to be extremely useful in modeling cluster algebras. They are used by Gross, Hacking, Keel, and Kontsevich to prove many long-standing conjectures on cluster algebras including the Positivity Conjecture and sign coherence of \textbf{c}-vectors and \textbf{g}-vectors. The first (resp., second) conjecture was only known for skew-symmetric cluster algebras (i.e., cluster algebras defined by quivers) from \cite[Theorem 1.1]{lee2015positivity} (resp., \cite[Theorem 1.7]{dwz10}). We follow the treatment of scattering diagrams defined by quivers as presented in \cite{muller2015existence}. We recommend that the reader refer to \cite{muller2015existence} for more details. 

In this section, we work with a fixed quiver $Q$, and we will associate to $Q$ a particular scattering diagram whose existence follows from \cite[Theorem 1.13]{gross2014canonical}. This scattering diagram, which we will refer to as the \textbf{consistent scattering diagram} for $Q$ and will denote by $\mathfrak{D}(Q)$, will encode all of the maximal green sequences of $Q$. We begin by reviewing scattering diagram terminology. Our goal in this section is to obtain a restatement of Greg Muller's theorem \cite[Theorem 9]{muller2015existence} (see Theorem~\ref{thmMullerSubquiver}).


Define $R := \mathbb{Q}[x_1^{\pm1},\ldots, x_n^{\pm1}][[y_1, \ldots, y_n]]$ to be the formal powers series ring with coefficients coming from the Laurent polynomial ring in the variables $x_1, \ldots, x_n$. For $\textbf{m} = (m_1, \ldots, m_n) \in \mathbb{Z}^n$ (resp., $\textbf{m} = (m_1, \ldots, m_n) \in \mathbb{N}^n$), we write $x^\textbf{m} = x_1^{m_1}\cdots x_n^{m_n}$ (resp., $y^\textbf{m} = y_1^{m_1}\cdots y_n^{m_n}$). For each $\textbf{m} \in \mathbb{N}^n$, define a \textbf{formal elementary transformation}, denoted $E_\textbf{m}$, by $$\begin{array}{rcl} R & \stackrel{E_{\textbf{m}}}{\longrightarrow} & R\\ x^{\textbf{m}^\prime} & \longmapsto & (1+x^{B\textbf{m}}y^\textbf{m})^{\frac{\textbf{m}\cdot\textbf{m}^\prime}{\text{gcd}(\textbf{m})}}x^{\textbf{m}^\prime}\\ y^{\textbf{m}^\prime} & \longmapsto & y^{\textbf{m}^\prime} \end{array}$$ where $\textbf{m}\cdot \textbf{m}^\prime := \sum_{i = 1}^n m_im_i^\prime$, $\text{gcd}(\textbf{m})$ is the greatest common divisor of the coordinates of \textbf{m}, and $B = B_{Q}$ is the exchange matrix of $Q$. The formal elementary transformation $E_\textbf{m}$ is an automorphism of $R$ with inverse given by $$\begin{array}{rcl} R & \stackrel{E^{-1}_{\textbf{m}}}{\longrightarrow} & R\\ x^{\textbf{m}^\prime} & \longmapsto & (1+x^{B\textbf{m}}y^\textbf{m})^{-\frac{\textbf{m}\cdot\textbf{m}^\prime}{\text{gcd}(\textbf{m})}}x^{\textbf{m}^\prime}\\ y^{\textbf{m}^\prime} & \longmapsto & y^{\textbf{m}^\prime}. \end{array}$$ 

As shown in \cite{muller2015existence}, elementary transformations, in general, do not commute. For example, by \cite[Prop. 4.1.1]{muller2015existence} if $\textbf{m}\cdot B\textbf{m}^\prime = 1$ for $\textbf{m}, \textbf{m}^\prime \in \mathbb{N}^n$, then $E_{\textbf{m}}E_{\textbf{m}^\prime} = E_{\textbf{m}^\prime}E_{\textbf{m}+\textbf{m}^\prime}E_{\textbf{m}}.$

Define a \textbf{wall} to be a pair $(\textbf{m}, W)$ where \begin{itemize} \item $\textbf{m} \in \mathbb{N}^n$ is nonzero and
\item $W \subset \mathbb{R}^n$ is a convex polyhedral cone that spans $\textbf{m}^\perp := \{\textbf{m}^\prime \in \mathbb{R}^n: \ \textbf{m}\cdot \textbf{m}^\prime = 0\}.$
\end{itemize}

\noindent Given a wall $(\textbf{m}, W)$, the fact that $\textbf{m} \in \mathbb{N}^n$ implies that $W \cap (\mathbb{R}_{>0})^n = W \cap (\mathbb{R}_{<0})^n = \emptyset.$ Also, $W$ defines two half-spaces $\{\textbf{m}^\prime \in \mathbb{R}^n: \ \textbf{m}\cdot \textbf{m}^\prime > 0\}$ and $\{\textbf{m}^\prime \in \mathbb{R}^n: \ \textbf{m}\cdot \textbf{m}^\prime < 0\}$, which we will refer to as the \textbf{green} side and \textbf{red} side of $W$, respectively. We define a \textbf{scattering diagram} in $\mathbb{R}^n$, denoted $\mathfrak{D} = \{(\textbf{m}^{(i)}, W_i)\}_{i \in I}$,  to be a \textit{multiset} of walls\footnote{As is noted in \cite{muller2015existence}, we allow multiple copies of the same wall.} where each $(\textbf{m}^{(i)}, W_i) \in \mathfrak{D}$ has $\textbf{m}^{(i)} \in \mathbb{N}^n$ and $W \in \mathbb{R}^n$. We define a \textbf{chamber} of $\mathfrak{D}$ to be a path-connected component of $\mathbb{R}^n-\mathfrak{D}$. 

Scattering diagrams can be used to understand diagrams of elementary transformations, as any wall $(\textbf{m}, W)$ is naturally associated with the elementary transformation $E_{\textbf{m}}$. Given a scattering diagram $\mathfrak{D}$ in $\mathbb{R}^n$ we say that a smooth path $p:[0,1] \to \mathbb{R}^n$ is \textbf{finite transverse} if \begin{itemize}\item $p(0) \not \in W$ and $p(1) \not \in W$ for any wall $(\textbf{m}, W)$ in $\mathfrak{D}$, \item for any wall $(\textbf{m}, W)$ of $\mathfrak{D}$, $p$ may only intersect $W$ transversally, and in this case, $p$ must not intersect the boundary of $W$, \item $p$ intersects at most finitely many distinct walls of $\mathfrak{D}$, and $p$ does not intersect two walls that span distinct hyperplanes at points where the two intersect (see path $q$ in Figure~\ref{DiagEx1}). \end{itemize} 

\begin{figure}
\includegraphics[scale=2]{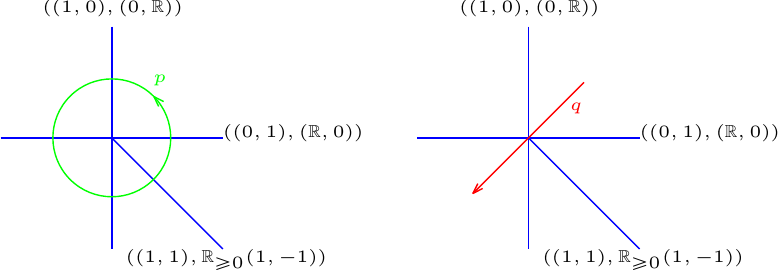}
\caption{Two copies of the consistent scattering diagram for $Q=1 \to 2$. The finite transverse loop $p$ defines the path-ordered product $E^{-1}_{(0,1)}E^{-1}_{(1,1)}E^{-1}_{(1,0)}E_{(0,1)}E_{(1,0)}$. We leave it to the reader to verify that this path-ordered product is the trivial automorphism.}
\label{DiagEx1}
\end{figure}

\noindent Let $(\textbf{m}^{(1)}, W_1), \ldots, (\textbf{m}^{(k)}, W_k)$ be the sequence of walls crossed by a finite transverse path $p$. Then $p$ determines the \textbf{path-ordered product} of elementary transformations given by $$ \displaystyle \overleftarrow{\displaystyle \prod\limits_{i \in [k]}} E^{\epsilon_i}_{\textbf{m}^{(i)}} := E^{\epsilon_k}_{\textbf{m}^{(k)}} \cdots E^{\epsilon_1}_{\textbf{m}^{(1)}}$$ where $\epsilon_{i} = 1$ (resp., $\epsilon_i = -1$) if $p$ crosses $W_i$ from its green side to its red side (resp., from its red side to its green side). 

We say that a scattering diagram $\mathfrak{D} = \{(\textbf{m}^{(i)}, W_i)\}_{i \in I}$ with finitely many walls is \textbf{consistent} if any path-ordered product defined by finite transverse loop is the trivial automorphism of $R$ (see loop $p$ in Figure~\ref{DiagEx1}). We say that two scattering diagrams $\mathfrak{D}_1$ and $\mathfrak{D}_2$ with finitely many walls are \textbf{equivalent} if each smooth path $p$ that is finite transverse in $\mathfrak{D}_1$ and $\mathfrak{D}_2$ defines the same path-ordered product in $\mathfrak{D}_1$ as it does in $\mathfrak{D}_2$. For brevity, we do not define consistency or equivalence for scattering diagrams $\mathfrak{D}$ with infinitely many walls. 


\begin{theorem}\cite[Theorems 1.13 and 1.28]{gross2014canonical}
For each quiver $Q$, there is a consistent scattering diagram, denoted $\mathfrak{D}(Q)$, unique up to equivalence, such that \begin{itemize}
\item for each $i \in \{1,\ldots, n\}$, there is a wall of the form $(\textbf{e}_i, \textbf{e}_i^\perp)$, and
\item every other wall $(\textbf{m}, W)$ in $\mathfrak{D}(Q)$ satisfies $B\textbf{m} \not \in W$.
\end{itemize}
\end{theorem}

The scattering diagram $\mathfrak{D}(Q)$ has two distinguished chambers, namely, $(\mathbb{R}_{>0})^n$ the \textbf{all-positive chamber} and $(\mathbb{R}_{<0})^n$ the \textbf{all-negative chamber}. We say that a chamber $\mathcal{C}$ in $\mathfrak{D}(Q)$ is \textbf{reachable} if there exists a finite transverse path from $(\mathbb{R}_{>0})^n$ to $\mathcal{C}$.

\begin{theorem}\cite[Lemmas 2.9 and 5.12]{gross2014canonical}\label{g_c_vec_DQ}
Every reachable chamber of $\mathfrak{D}(Q)$ is of the form $$\mathbb{R}_{>0}\textbf{g}_1 + \mathbb{R}_{>0}\textbf{g}_2+ \cdots + \mathbb{R}_{>0}\textbf{g}_n$$ for some \textbf{g}-matrix $G = \{\textbf{g}_1, \textbf{g}_2, \ldots, \textbf{g}_n\} \in \textbf{g}\text{-mat}(Q)$. The inward-pointing normal vectors of the same reachable chamber are the row vectors of the \textbf{c}-matrix $C = (G^{-1})^T \in \textbf{c}\text{-mat}(Q)$. These correspondences induce bijections between the reachable chambers of $\mathfrak{D}(Q)$ and elements of $\textbf{g}\text{-mat}(Q)$ and elements of $\textbf{c}\text{-mat}(Q)$. In particular, the positive \textbf{c}-vectors of $Q$ are exactly the vectors $\textbf{m}$ that appear in a wall $(\textbf{m}, W)$ of $\mathfrak{D}(Q)$ that is incident to a reachable chamber.
\end{theorem}


A useful consequence of Theorem~\ref{g_c_vec_DQ} is that there is a bijection between maximal green sequences of $Q$ and finite transverse paths $p:[0,1]\to \mathbb{R}^n$ in $\mathfrak{D}(Q)$ from the all-positive chamber to the all-negative chamber with the property that whenever $p$ crosses a wall of $\mathfrak{D}(Q)$ it does so from the green side to the red side. The following is essentially proven in \cite{muller2015existence}.

\begin{theorem}\label{thmMullerSubquiver}
Let $Q$ be any quiver and $Q^\dagger$ any full subquiver of $Q$. There is a map $\text{MGS}(Q) \to \text{MGS}(Q^\dagger)$ defined by sending $\textbf{i} \in \text{MGS}(Q)$ to the maximal green sequence $\textbf{i}^\dagger \in \text{MGS}(Q^\dagger)$ where $\textbf{c}(\textbf{i}^\dagger) = (\textbf{c}^{(1)},\ldots, \textbf{c}^{(\ell)})$ is the unique longest subsequence of $\textbf{c}(\textbf{i})$ where each $\textbf{c}^{(j)} = (c_{1}^{(j)},\ldots, c_{n}^{(j)})$ satisfies $c_i^{(j)} = 0$ if $i \in Q_0\backslash Q^\dagger_0.$  \end{theorem}
\begin{proof}
The scattering diagram $\mathfrak{D}(Q^\dagger) = \{(\textbf{m}^\dagger_j, W^\dagger_j)\}_{j \in J}$ can be obtained from $\mathfrak{D}(Q)$ by considering the subset of walls of $\mathfrak{D}(Q)$ given by $\pi^*\mathfrak{D}(Q^\dagger) = \{(\pi^T(\textbf{m}^\dagger_{j}), \pi^{-1}(W^\dagger_j))\}_{j \in J}$. Here $\pi:\mathbb{R}^n \to \mathbb{R}^k$ is the canonical projection and $\pi^T: \mathbb{R}^k \to \mathbb{R}^n$ is the coordinate inclusion. More precisely, by \cite[Theorem 33]{muller2015existence}, the scattering diagram $\pi^*\mathfrak{D}(Q^\dagger)$ is obtained from $\mathfrak{D}(Q)$ by removing all of its walls $(\textbf{m}, W)$ where $\textbf{m} = (m_1, \ldots, m_n)$ has any nonzero entry $m_i$ with $i \in Q_0\backslash Q^\dagger_0.$ Now by Theorem~\ref{g_c_vec_DQ}, any wall $(\textbf{m}^\dagger, W^\dagger)$ of $\mathfrak{D}(Q^\dagger)$ that is incident to a reachable chamber of $\mathfrak{D}(Q^\dagger)$ has the property that $\textbf{m}^\dagger = (m^\dagger_1,\ldots, m^\dagger_n)$ is a positive \textbf{c}-vector of $Q$ where $m^\dagger_i = 0$ if $i \in Q_0\backslash Q^\dagger_0$.

Next, let $\textbf{i} \in \text{MGS}(Q)$ and let $p:[0,1] \to \mathbb{R}^n$ be a finite transverse path that defines $\textbf{i}$. Now consider the path $\pi\circ p:[0,1] \to \mathbb{R}^k$. By the proof of \cite[Theorem 9]{muller2015existence}, we know that $\pi\circ p$ is a finite transverse path in $\mathfrak{D}(Q^\dagger)$. Thus $\pi \circ p$ defines a maximal green sequence $\textbf{i}^\dagger \in \text{MGS}(Q^\dagger)$, and $\pi\circ p$ only crosses walls $(\textbf{m}^\dagger, W^\dagger)$ of $\mathfrak{D}(Q^\dagger)$ that are incident to reachable chambers of $\mathfrak{D}(Q^\dagger)$. Now by the previous paragraph, we conclude that $\textbf{c}(\textbf{i}^\dagger) = (\textbf{c}^{(1)},\ldots, \textbf{c}^{(\ell)})$ is the unique longest subsequence of $\textbf{c}(\textbf{i})$ where each $\textbf{c}^{(j)} = (c_1^{(j)},\ldots, c^{(j)}_n)$ satisfies $c_i^{(j)} = 0$ if $i \in Q_0\backslash Q^\dagger_0.$
\end{proof}




As a corollary we obtain a result on the lengths of maximal green sequences that characterizes acyclic quivers. 

\begin{corollary}\label{Cor:acyclic_result}
A quiver $Q$ is acyclic if and only if $Q$ admits a maximal green sequence of length $\abs{Q_0}$. 
\end{corollary}

\begin{proof}
The forward direction follows from \cite[Lemma 2.20]{bdp}.  To show the backward direction we suppose that there exists a quiver $Q$ with oriented cycles that admits a maximal green sequence {\bf i} of length $\abs{Q_0}$.  Note that in this case every vertex must be mutated exactly once.  Let $Q'$ be a subquiver of $Q$ consisting of a single vertex, and observe that $Q'$ admits a unique maximal green sequence of length 1.  Theorem~\ref{thmMullerSubquiver} implies that ${\bf c}({\bf i})$ consists of $\abs{Q_0}$ vectors of the form ${\bf c}^{(j)}$, where the only nonzero entry of ${\bf c}^{(j)}$ equals 1 and occurs at position $j$.  

Since $Q$ contains oriented cycles there exists a full subquiver $Q^{\dagger}$ of $Q$ such that no vertex of $Q^{\dagger}$ is a source in $Q^{\dagger}$.  By Theorem \ref{thmMullerSubquiver} the sequence ${\bf i}^{\dagger}$ is a maximal green sequence for $Q^{\dagger}$ of length $\abs{Q^{\dagger}}$.   Suppose ${\bf i}^{\dagger}$ begins with a mutation at vertex $i$.  By construction, there exists an arrow $j \to i$ ending at $i$ in $Q^{\dagger}$.  Then in $\mu_i(\widehat{Q^{\dagger}})$ the {\bf c}-vector at vertex $j$ becomes strictly positive at position $i$.  Furthermore, since we will not mutate at $i$ again while performing ${\bf i}^{\dagger}$ it follows that prior to a mutation at $j$ the corresponding {\bf c}-vector will remain strictly positive at position $i$.  However, this contradicts the description of ${\bf c} ({\bf i}^{\dagger})$ according to Theorem \ref{thmMullerSubquiver}.  Therefore, we can conclude that if $Q$ contains oriented cycles then it cannot admit a maximal green sequence of length $\abs{Q_0}$.  
\end{proof}

\section{Direct sums of quivers}\label{sec:direct sums}

In this section, we recall the definition of a direct sum of quivers following \cite{garver2014maximal}. In \cite{garver2014maximal}, it was shown that, under certain restrictions, if a quiver $Q$ can be written as a direct sum of quivers where each summand has a maximal green sequence, then the maximal green sequences of the summands can be concatenated in some way to give a maximal green sequence for $Q$. We show that, under those same restrictions, if $Q$ is a direct sum of quivers, then the minimal length of a maximal green sequence of $Q$ is the sum of the lengths of minimal length maximal green sequences of its summands. Throughout this section, we let $(Q_1,F_1)$ and $(Q_2,F_2)$ be finite ice quivers with $n_1$ and $n_2$ nonfrozen vertices, respectively. Furthermore, we assume $(Q_1)_0\backslash F_1 = [n_1]$ and $(Q_2)_0\backslash F_2 = [n_1 + 1, n_1 + n_2].$  

\begin{definition}\label{dirsum}
Let $(a_1,\ldots,a_k)$ denote a $k$-tuple of elements from $(Q_1)_0\backslash F_1$ and $(b_1,\ldots, b_k)$ a $k$-tuple of elements from $(Q_2)_0\backslash F_2$. (By convention, we assume that the  $k$-tuple $(a_1,\ldots, a_k)$ is ordered so that $a_i \le a_j$ if $i < j$ unless stated otherwise.)  Additionally, let $({R}_1,F_1) \in \text{Mut}(({Q_1},F_1))$ and $(R_2, F_2) \in \text{Mut}(({Q_2},F_2))$. We define the \textbf{direct sum} of $({R}_1,F_1)$ and $(R_2,F_2)$, denoted $(R_1,F_1)\oplus_{(a_1,\ldots, a_k)}^{(b_1,\ldots,b_k)}(R_2,F_2)$, to be the ice quiver with vertices$$\left((R_1,F_1)\oplus_{(a_1,\ldots, a_k)}^{(b_1,\ldots,b_k)}(R_2,F_2)\right)_0 := ({R}_1)_0 \sqcup ({R}_2)_0 = ({Q}_1)_0 \sqcup ({Q}_2)_0 = [n_1+n_2] \sqcup F_1 \sqcup F_2$$ and arrows $$\left((R_1,F_1)\oplus_{(a_1,\ldots, a_k)}^{(b_1,\ldots,b_k)}(R_2,F_2)\right)_1 := (R_1,F_1)_1 \sqcup (R_2,F_2)_1 \sqcup \left\{a_i \stackrel{\alpha_i}{\to} b_i: i \in [k]\right\}.$$ Observe that we have the identification of ice quivers $$\widehat{Q_1\oplus_{(a_1,\ldots, a_k)}^{(b_1,\ldots,b_k)}Q_2} \cong \widehat{Q_1}\oplus_{(a_1,\ldots, a_k)}^{(b_1,\ldots,b_k)}\widehat{Q_2},$$ 
where $m=2(n_1+n_2)$ in both cases. 
\end{definition}

\begin{example}\label{directsum1}
Let $Q$ denote the quiver shown in Figure~\ref{dirsumex}. Define $Q_1$ to be the full subquiver of $Q$ on the vertices $1,\ldots, 4$, $Q_2$ to be the full subquiver of $Q$ on the vertices $6,\ldots, 11$, and $Q_3$ to be the full subquiver of $Q$ on the vertex 5. Note that $Q_1, Q_2,$ and $Q_3$ are each irreducible. Then $$Q = Q_1\oplus_{(1,1,1,3,4,4)}^{(5,8,11,8,9,11)}Q_{23}$$ where $Q_{23} = Q_2\oplus_{(6)}^{(5)}Q_3$. On the other hand, we could write $$Q = Q_{12} \oplus_{(1,6)}^{(5,5)}Q_3$$ where $Q_{12} = Q_1\oplus_{(1,1,3,4,4)}^{(8,11,8,9,11)}Q_2$. Additionally, note that $$Q_1\oplus_{(1,1,1,3,4,4)}^{(5,8,11,8,9,11)}Q_{23} = Q_1\oplus_{(1,1,1,3,4,4)}^{(5,8,11,8,9,11)}\left(Q_2\oplus_{(6)}^{(5)}Q_3\right) \neq \left(Q_1\oplus_{(1,1,1,3,4,4)}^{(5,8,11,8,9,11)}Q_2\right)\oplus_{(6)}^{(5)}Q_3$$
where the last equality does not hold because $Q_1\oplus_{(1,1,1,3,4,4)}^{(5,8,11,8,9,11)}Q_2$ is not defined as $5$ is not a vertex of $Q_2$. This shows that the direct sum of two quivers, in the sense of this paper, is not associative.

\begin{figure}[h]
$$\begin{xy} 0;<1pt,0pt>:<0pt,-1pt>:: 
(25,25) *+{1} ="0",
(0,75) *+{2} ="1",
(50,75) *+{3} ="2",
(25,125) *+{4} ="3",
(150,0) *+{5} ="4",
(200,25) *+{6} ="5",
(150,50) *+{7} ="6",
(200,75) *+{8} ="7",
(150,100) *+{9} ="8",
(200,125) *+{10} ="9",
(150,150) *+{11} ="10",
"1", {\ar"0"},
"0", {\ar"2"},
"0", {\ar"4"},
"0", {\ar"7"},
"0", {\ar"10"},
"3", {\ar"1"},
"2", {\ar"3"},
"2", {\ar"7"},
"3", {\ar"8"},
"3", {\ar"10"},
"5", {\ar"4"},
"6", {\ar"5"},
"5", {\ar"7"},
"7", {\ar"6"},
"7", {\ar"8"},
"8", {\ar"9"},
"10", {\ar"8"},
"9", {\ar"10"},
\end{xy}$$
\caption{}
\label{dirsumex}
\end{figure}
\end{example}

\begin{theorem}\label{tcolordirsum}\cite[Prop. 3.14]{garver2014maximal}
If $\textbf{i}_1 \in \text{MGS}\left(Q_1\right)$ and $\textbf{i}_2 \in \text{MGS}\left(Q_2\right)$ and $Q := Q_1\oplus_{(a_1,\ldots, a_t)}^{(b_1,\ldots, b_t)}Q_2$ is a direct sum of quivers where $|\{(a_i \stackrel{\alpha}{\longrightarrow} b_j) \in (Q)_1\}| \le 1$ for any $i, j \in [t]$, then $\textbf{i}_1\circ\textbf{i}_2 \in \text{MGS}\left(Q\right).$
\end{theorem}

\begin{proposition}\label{minlengthdirectsum}
Assume $\textbf{i}_1 \in \text{MGS}\left(Q_1\right)$ and $\textbf{i}_2 \in \text{MGS}\left(Q_2\right)$ and $Q := Q_1\oplus_{(a_1,\ldots, a_t)}^{(b_1,\ldots, b_t)}Q_2$ is a direct sum of quivers where $|\{(a_i \stackrel{\alpha}{\longrightarrow} b_j) \in (Q)_1\}| \le 1$ for any $i, j \in [t]$. If $\ell_j$ for $j = 1,2$ is the length of a minimal length maximal green sequence of $Q_j$, then $\ell_1 + \ell_2$ is the length of a minimal length maximal green sequence of $Q$. Furthermore, the maximal green sequence $\textbf{i}_1\circ \textbf{i}_2$ of $Q$ achieves this length.
\end{proposition}
\begin{proof}
Assume that $\textbf{i}_j \in \text{MGS}\left(Q_j\right)$ for $j = 1,2$ is of minimal length. By Theorem~\ref{tcolordirsum}, the sequence $\textbf{i}_1\circ\textbf{i}_2 \in \text{MGS}\left(Q\right)$ and has length $\ell_1 +\ell_2$. 

We show that $Q$ has no maximal green sequences of length less than $\ell_1 + \ell_2$. Let $\textbf{i} = (i_1,\ldots, i_k) \in \text{MGS}(Q).$ By Theorem~\ref{thmMullerSubquiver}, the full subquivers $Q_1$ and $Q_2$ have maximal green sequences $\textbf{j}_1 \in \text{MGS}(Q_1)$ and $\textbf{j}_2  \in \text{MGS}(Q_2)$ such that $\textbf{c}(\textbf{j}_1)$ (resp., $\textbf{c}(\textbf{j}_2)$) is the unique longest subsequence of $\textbf{c}(\textbf{i})$ where each \textbf{c}-vector $\textbf{c} = (c_{1},\ldots, c_{n})$ in $\textbf{c}(\textbf{j}_1)$ (resp., $\textbf{c}(\textbf{j}_2)$) satisfies $c_i = 0$ if $i \in Q_0\backslash (Q_1)_0$ (resp., $i \in Q_0\backslash (Q_2)_0$). This implies that $\ell(\textbf{j}_1) \ge \ell_1$ and $\ell(\textbf{j}_2) \ge \ell_2$. Now since $Q_1$ and $Q_2$ have no common vertices, we know that $\ell(\textbf{i}) \ge \ell_1 + \ell_2.$ 
\end{proof}

\section{Quivers with branches of mutation type $\mathbb{A}_n$}\label{Sec_branches}

In this section, we show that if a quiver $\widetilde{Q}$ consists of a quiver $Q$ and mutation type $\mathbb{A}$ subquivers branching out from a given set of vertices of $Q$, then we can essentially ignore these mutation type $\mathbb{A}$ subquivers while computing (minimal length) maximal green sequences for $\widetilde{Q}$.  If these branches are made up of acyclic quivers, then $\widetilde{Q}$ can be realized as a direct sum of quivers, and we already discussed in Section~\ref{sec:direct sums} how to construct maximal green sequences for such $\widetilde{Q}$.   However, if the branches contain 3-cycles then we cannot use the same approach.  Therefore, the results presented in this section enable us to decompose $\widetilde{Q}$ even further by disregarding such subquivers of mutation type $\mathbb{A}$ and focusing only on $Q$. 

We begin with the following lemma, which is a reformulation of \cite[Lemma 2.14]{cormier2015minimal}.

\begin{lemma}\label{oldlem}
Let $Q \in \textup{Mut} (\widehat{Q'})$ for some quiver $Q'$, such that the following conditions hold. 

\begin{itemize}
\item  $Q$ is composed of two full subquivers $\mathcal{D}$ and the framed quiver $\widehat{\mathcal{C}}$; 
\item $\widehat{\mathcal{C}}$ and $\mathcal{D}$ are connected by a single arrow $x\rightarrow z$, where $x\in\mathcal{C}_0$ and $z\in\mathcal{D}_0$.   
\end{itemize} 

\noindent Let $\mu_{\mathcal{C}}$ be a maximal green sequence for $\mathcal{C}$, then 

\begin{itemize}
\item $\mu_{\mathcal{C}}(Q)$ is composed of two full subquivers $\mathcal{D}$ and $\mu_{\mathcal{C}} (\widehat{\mathcal{C}})$;
 \item $\mu_{\mathcal{C}} (\widehat{\mathcal{C}})$ and $\mathcal{D}$ are connected by a single arrow $z\rightarrow v$,  where $v$ is a unique vertex of $\mathcal{C}_0$ with an arrow $x' \rightarrow v$ starting at the frozen vertex $x'$.  
\end{itemize}

$$ \begin{tikzpicture}[scale = .4][
> = stealth, 
            shorten > = 4pt, ]
\node[draw=none, fill=none] at (0.7,0.4) {$Q$:};
            
\node[draw=none, fill=none] at (7.7,0.4) {$\widehat{\mathcal{C}}$};
\node[draw, shape=circle, fill=black, scale=0.4] at (10.5,0) { };
\node[draw=none, fill=none] at (10.5,-0.8) {$x$};
\node[draw, shape=circle, fill=black, scale=0.4] at (12.5,0) { };
\node[draw=none, fill=none] at (12.5,-0.9) {$z$};
\node[draw=none, fill=none] at (15.7,0.4) {$\mathcal{D}$};

\draw [black] plot [smooth cycle] coordinates {(4.5,0) (6.5,2) (10.5,2) (10.5,0) (8.5,-2)};
\draw [black] plot [smooth cycle] coordinates {(12.5,0) (14.5,2) (18.5,2) (18.5,0) (16.5,-2)};

\draw [->] (10.6,0) -- node { } (12.2,0);

\end{tikzpicture}$$

\end{lemma}

\begin{lemma}\label{3-cycle}
Consider a quiver $Q$ composed of an arbitrary quiver $\mathcal{C}$ and a 3-cycle joined together at a vertex $x$ as shown below.  
$$\begin{tikzpicture}[scale=.4][
> = stealth, 
            shorten > = 4pt, ]
\node[draw=none, fill=none] at (0.7,0.4) {$Q$:};          
\node[draw=none, fill=none] at (7.7,0.4) {$\mathcal{C}$};
\node[draw, shape=circle, fill=black, scale=0.4] at (10.5,0) { };
\node[draw=none, fill=none] at (10.5,-0.7) {$x$};

\node[draw, shape=circle, fill=black, scale=0.4] at (13.5,2) { };
\node[draw=none, fill=none] at (13.5,2.7) {$y$};

\node[draw, shape=circle, fill=black, scale=0.4] at (13.5,-2) { };
\node[draw=none, fill=none] at (13.5,-2.6) {$z$};

\draw [black] plot [smooth cycle] coordinates {(4.5,0) (6.5,2) (10.5,2) (10.5,0) (8.5,-2)};
\draw [->] (10.6,0) -- node { } (13.35,1.85);
\draw [->] (13.5,2) -- node { } (13.5,-1.8);
\draw [->] (13.5,-2) -- node { } (10.7,-.1);
\end{tikzpicture}$$
Let $\mu_{\mathcal{C}}$ be a maximal green sequence for $\mathcal{C}$, then 

$$\mu_z\mu_y \mu_{\mathcal{C}} \mu_z$$
is a maximal green sequence for $Q$.  
\end{lemma}

\begin{proof}
We begin by mutating $\widehat{Q}$ at vertex $z$.  The resulting quiver is shown below and we note that $\widehat{\mathcal{C}}$ remains a full subquiver of $\mu_z(\widehat{Q})$ connected to the remaining vertices of $\widehat{Q}$ by a single arrow.  

$$\begin{tikzpicture}[scale=.4][
> = stealth, 
            shorten > = 4pt, ]
\node[draw=none, fill=none] at (-1.5,0.4) {$\mu_z(\widehat{Q})$:};
\node[draw=none, fill=none] at (3.7,0.4) {$\widehat{\mathcal{C}}$};
\node[draw, shape=circle, fill=black, scale=0.4] at (6.5,0) { };
\node[draw=none, fill=none] at (7.3, .6) {$x$};

\node[draw, shape=circle, fill=black, scale=0.4] at (6.5,-2) { };
\node[draw=none, fill=none] at (6.5,-2.6) {$x'$};

\node[draw, shape=circle, fill=black, scale=0.4] at (9.5,2) { };
\node[draw=none, fill=none] at (9.5,2.9) {$y$};

\node[draw, shape=circle, fill=black, scale=0.4] at (12.5,2) { };
\node[draw=none, fill=none] at (12.5,2.9) {$y'$};

\node[draw, shape=circle, fill=black, scale=0.4] at (9.5,-2) { };
\node[draw=none, fill=none] at (9.5,-2.6) {$z$};

\node[draw, shape=circle, fill=black, scale=0.4] at (12.5,-2) { };
\node[draw=none, fill=none] at (12.5,-2.6) {$z'$};

\draw [black] plot [smooth cycle] coordinates {(0.5,0) (2.5,2) (6.5,2) (6.5,0) (2.5,-2)};
\draw [<-] (9.5,1.7) -- node { } (9.5,-1.8);
\draw [<-] (9.3,-1.8) -- node { } (6.7,-.1);
\draw [<-] (6.5, -1.7) -- node {} (6.5, 0);
\draw [<-] (12.2, 2) -- node {} (9.5, 2);
\draw [<-] (9.8, -2) -- node {} (12.5, -2);
\draw [<-] (12.3, -1.8) -- node {} (9.5, 2);


\node[draw=none, fill=none] at (21,0.4) {$\mu_{\mathcal{C}}\mu_z(\widehat{Q})$:};
\node[draw=none, fill=none] at (26.7,0.4) {$\mu_{\mathcal{C}}(\widehat{\mathcal{C}})$};

\node[draw, shape=circle, fill=black, scale=0.4] at (29.5,-2) { };
\node[draw=none, fill=none] at (29.5,-2.6) {$x'$};

\node[draw, shape=circle, fill=black, scale=0.4] at (30,1) { };
\node[draw=none, fill=none] at (30.2,1.7) {$v$};

\node[draw, shape=circle, fill=black, scale=0.4] at (32.5,2) { };
\node[draw=none, fill=none] at (32.5,2.9) {$y$};

\node[draw, shape=circle, fill=black, scale=0.4] at (35.5,2) { };
\node[draw=none, fill=none] at (35.5,2.9) {$y'$};

\node[draw, shape=circle, fill=black, scale=0.4] at (32.5,-2) { };
\node[draw=none, fill=none] at (32.5,-2.6) {$z$};

\node[draw, shape=circle, fill=black, scale=0.4] at (35.5,-2) { };
\node[draw=none, fill=none] at (35.5,-2.6) {$z'$};

\draw [black] plot [smooth cycle] coordinates {(23.5,0) (25.5,2) (29.5,2) (29.5,0) (25.5,-2)};
\draw [<-] (32.5,1.7) -- node { } (32.5,-1.8);
\draw [->] (32.3,-1.8) -- node { } (30.3,.7);
\draw [->] (29.5, -1.7) -- node {} (30.1, .7);
\draw [<-] (35.2, 2) -- node {} (32.5, 2);
\draw [<-] (32.8, -2) -- node {} (35.5, -2);
\draw [<-] (35.3, -1.8) -- node {} (32.5, 2);

\end{tikzpicture}$$

We can see that $\mu_z(\widehat{Q})$ satisfies conditions of Lemma~\ref{oldlem} with $\mathcal{D}$ being the full subquiver of $\mu_z(\widehat{Q})$ on vertices $y, z, y', z'$.  Thus, by Lemma~\ref{oldlem} the quiver $\mu_{\mathcal{C}}\mu_z (\widehat{Q})$ consists of two full subquivers $\mu_{\mathcal{C}}(\widehat{Q}')$ and $\mathcal{D}$ connected by a single arrow $z\to v$ where $v$ is the unique vertex of $\mathcal{C}$ with an arrow $x'\to v$ in $\mu_{\mathcal{C}} (\widehat{\mathcal{C}})$.  Since, $\mu_{\mathcal{C}}$ is a maximal green sequence for $\mathcal{C}$ if follows that all vertices of $\mathcal{C}$ are red in $\mu_{\mathcal{C}}\mu_z (\widehat{Q})$.  The only remaining green vertex in $\mu_{\mathcal{C}}\mu_z (\widehat{Q})$ is $y$.  Now, we perform $\mu_y$ which makes the vertex $z$ green, and then we perform $\mu_z$.  The resulting quivers are depicted below, where we let $\mu_Q = \mu_z\mu_y\mu_{\mathcal{C}}\mu_z$.

$$\begin{tikzpicture}[scale=.4][
> = stealth, 
            shorten > = 4pt, ]
\node[draw=none, fill=none] at (-2.5,0.4) {$\mu_y\mu_{\mathcal{C}}\mu_z(\widehat{Q})$:};
\node[draw=none, fill=none] at (3.7,0.4) {$\mu_{\mathcal{C}}(\widehat{\mathcal{C}})$};

\node[draw, shape=circle, fill=black, scale=0.4] at (6.5,-2) { };
\node[draw=none, fill=none] at (6.5,-2.6) {$x'$};

\node[draw, shape=circle, fill=black, scale=0.4] at (7,1) { };
\node[draw=none, fill=none] at (7.2,1.7) {$v$};

\node[draw, shape=circle, fill=black, scale=0.4] at (9.5,2) { };
\node[draw=none, fill=none] at (9.5,2.9) {$y$};

\node[draw, shape=circle, fill=black, scale=0.4] at (12.5,2) { };
\node[draw=none, fill=none] at (12.5,2.9) {$y'$};

\node[draw, shape=circle, fill=black, scale=0.4] at (9.5,-2) { };
\node[draw=none, fill=none] at (9.5,-2.6) {$z$};

\node[draw, shape=circle, fill=black, scale=0.4] at (12.5,-2) { };
\node[draw=none, fill=none] at (12.5,-2.6) {$z'$};

\draw [black] plot [smooth cycle] coordinates {(0.5,0) (2.5,2) (6.5,2) (6.5,0) (2.5,-2)};

\draw [->] (9.3,-1.8) -- node { } (7.3,.7);
\draw [->] (6.5, -1.7) -- node {} (7.1, .7);
\draw [->] (9.5, -2) -- node {} (12.2, 1.8);

\draw [->] (9.5,1.7) -- node { } (9.5,-1.7);
\draw [->] (12.2, 2) -- node {} (9.8, 2);
\draw [->] (12.3, -1.8) -- node {} (9.7, 1.8);


\node[draw=none, fill=none] at (19.5,0.4) {$\mu_{Q}(\widehat{Q})$:};
\node[draw=none, fill=none] at (24.7,0.4) {$\mu_{\mathcal{C}}(\widehat{\mathcal{C}})$};

\node[draw, shape=circle, fill=black, scale=0.4] at (28,1) { };
\node[draw=none, fill=none] at (28.2,1.8) {$v$};

\node[draw, shape=circle, fill=black, scale=0.4] at (27.5,-2) { };
\node[draw=none, fill=none] at (27.5,-2.6) {$x'$};

\node[draw, shape=circle, fill=black, scale=0.4] at (30.5,2) { };
\node[draw=none, fill=none] at (30.5,2.9) {$y$};

\node[draw, shape=circle, fill=black, scale=0.4] at (33.5,2) { };
\node[draw=none, fill=none] at (33.5,2.9) {$y'$};

\node[draw, shape=circle, fill=black, scale=0.4] at (30.5,-2) { };
\node[draw=none, fill=none] at (30.5,-2.6) {$z$};

\node[draw, shape=circle, fill=black, scale=0.4] at (33.5,-2) { };
\node[draw=none, fill=none] at (33.5,-2.6) {$z'$};

\draw [black] plot [smooth cycle] coordinates {(21.5,0) (23.5,2) (27.5,2) (27.5,0) (23.5,-2)};

\draw [<-] (30.5,1.7) -- node { } (30.5,-1.8);
\draw [->] (27.5, -2) -- node {} (28.1, .7);
\draw [->] (28, 1) -- node {} (30.3,-1.8);
\draw [->] (33.3, -1.8) -- node {} (30.7, 1.8);
\draw [->] (33.5, 2) -- node {} (30.7, -1.8);
\draw [->] (30.5, 2) -- node {} (28.2, 1.2);

\end{tikzpicture}$$

We can see that in the final quiver $\mu_{Q}(\widehat{Q})$ every vertex is red.  This shows that $\mu_z\mu_y \mu_{\mathcal{C}} \mu_z$ is a maximal green sequence for $Q$.  
\end{proof}

Repeated applications of the lemma will allow us to construct minimal length maximal green sequences for quivers $\widetilde{Q}$ defined below. 

\begin{definition}\label{quiver_qtilde}
Given a quiver $Q$ let $\widetilde{Q}$ be a quiver composed of full connected subquivers $Q, Q^1, Q^2, \dots, Q^k$, such that all of the following conditions hold. 
\begin{itemize}
\item $Q_0^i \cap Q_{0} = \{x_i\}$.
\item $Q_0^i \cap Q_0^j = \Big\{ \begin{array}{ll} \{x_i\} & \text{if } x_i = x_j \\ \varnothing & \text{otherwise} \end{array}$. 
\item for every arrow in $\widetilde{Q}$, whenever one of the endpoints belongs to $Q_0^i \setminus \{x_i\}$ then the other endpoint belongs to $ Q_0^i$. 
\item for every $i$ the quiver $Q^i$ is of mutation type $\mathbb{A}$. 
\end{itemize}

$$
\begin{tikzpicture}
[scale=.4][
> = stealth, 
            shorten > = 4pt, ]
           
\node[draw=none, fill=none] at (-10,2) {$\widetilde{Q}$:};
\node[draw=none, fill=none] at (2,2) {$Q$};
\node[draw=none, fill=none] at (6.5,6) {$Q^1$};
\node[draw=none, fill=none] at (6,-2.5) {$Q^2$};
\node[draw=none, fill=none] at (-4,2) {$Q^k$};

\node[draw, shape=circle, fill=black, scale=0.4] at (-1,2) { };
\node[draw=none, fill=none] at (-1.5,3) {\small ${x_k}$};

\node[draw, shape=circle, fill=black, scale=0.4] at (3.8,4.4) { };
\node[draw=none, fill=none] at (3.3,5.3) {\small ${x_1}$};

\node[draw, shape=circle, fill=black, scale=0.4] at (3.6,-0.7) { };
\node[draw=none, fill=none] at (3,-1.6) {\small ${x_2}$};

\draw (2,2) circle (3cm);

\draw (-4,2) ellipse (3cm and 1cm);
\draw[rotate=30] (8.5,2) ellipse (3cm and 1cm);
\draw[rotate=-35] (6.3,1.5) ellipse (3cm and 1cm);
\draw[dotted][rotate=210]  (2,-1) arc (0:75:5cm);
\end{tikzpicture}$$

\end{definition}

Note that the vertices $x_i$ are not necessarily distinct.  With this notation consider the following results which show that in order to find a (minimal length) maximal green sequence for $\widetilde{Q}$ it suffices to find such sequence for its subquiver $Q$.  In other words, one can disregard the attached mutation type $\mathbb{A}$ quivers.  

\begin{theorem}\label{qtilde}
Let $\mu_{Q}$ be a maximal green sequence for $Q$. Then the quiver $\widetilde{Q}$ admits a maximal green sequence of length 

$$\ell({\mu_{Q}})+ l_{min}^1+l_{min}^2+\dots+ l_{min}^k - k$$
where $l_{min}^i$ is the minimal length of a maximal green sequence for $Q^i$.  
\end{theorem}

\begin{proof}
Starting with $Q$ we will gradually construct $\widetilde{Q}$ by adding either a 3-cycle or an arrow, that does not lies in any 3-cycle.   At each step we will produce a maximal green sequence for the resulting quiver building on the sequence obtained in the previous step.  

Let $z_1$ be a vertex of $\widetilde{Q}$ that does not belong to $Q$ such that it is connected to a vertex in $Q$ by an arrow $\alpha_1$.  By definition of $\widetilde{Q}$ there exists a unique integer $i$ such that $z_1\in Q_0^i$ and the arrow $\alpha_1$ is unique, because $Q^i$ is of mutation type $\mathbb{A}$.  Moreover, the endpoint of $\alpha_1$ that lies in $Q$ equals $x_i$, where $\{x_i\}=Q_0\cap Q^i_0$.  Now we consider four possibilities depending on $\alpha_1$.  If this arrow does not lie in a 3-cycle in $Q^i$ then we have case (i) or (ii), and if $\alpha_1$ lies in a 3-cycle then we have case (iii) or (iv) as depicted below. 
$$\xymatrix@!C=5pt@!R=5pt{&&&&&&&&&y_1\ar[dl]&&&&y_1\ar[dr]\\
x_1 \ar[rr]^{\alpha_1} && z_1 && x_1 &&\ar[ll]_{\alpha_1} z_1 && x_1\ar[rr]^{\alpha_1}&& z_1 \ar[ul] &&  x_1 \ar[ur]&& \ar[ll]_{\alpha_1} z_1\\
&\text{(i)}&&&&\text{(ii)} &&&&\text{(iii)} &&&& \text{(iv)} }$$

Now, let $Q(\alpha_1)$ be a full subquiver of $\widetilde{Q}$ on vertices $Q_0\cup \{z_1\}$ in cases (i) and (ii) or on vertices $Q_0\cup \{z_1, y_1\}$ in cases (iii) and (iv).  By Theorem~\ref{tcolordirsum} we have that $\mu_{z_1} \mu_Q$ or $\mu_Q \mu_{z_1}$ is a maximal green sequence for $Q(\alpha_1)$ in case (i) or (ii) respectively.  On the other hand, in case (iii) or (iv) the quiver $Q(\alpha_1)$ satisfies conditions of Lemma~\ref{3-cycle}, so we have that $\mu_{y_1}\mu_{z_1}\mu_Q \mu_{y_1}$ or $\mu_{z_1}\mu_{y_1}\mu_Q\mu_{z_1}$ is a maximal green sequence for $Q(\alpha_1)$ respectively.  

Next, we proceed in the same way with $Q(\alpha_1)$ as we did with $Q$.  That is, we pick a vertex $z_2 \not \in Q(\alpha_1)_0$ that is connected to $Q(\alpha_1)$ by a (unique) arrow $\alpha_2$.  Analogously, we define $Q(\alpha_1, \alpha_2)$, depending on the configuration around $\alpha_2$, and the corresponding maximal green sequence for this quiver.  Continuing in this way we obtain a maximal green sequence for $\widetilde{Q}$.  

Observe, that at each step we either add a single vertex or a 3-cycle, that is two vertices that lie in the same 3-cycle.  At the same time, we increase the length of the maximal green sequence by one or three respectively.  Thus, the length of the resulting maximal green sequence for $\widetilde{Q}$ is 

$$\ell({\mu_Q}) + (n_1 + t_1) + \dots + (n_k + t_k)$$
where $n_i$ is the number of vertices in $Q_0^i \setminus \{x_i\}$ and $t_i$ is the number of 3-cycles in $Q^i$.  By Theorem~\ref{thmtypeA} we have

$$l^i_{min}=(n_i + 1) +  t_i$$ 
which completes the proof of the theorem.    
\end{proof}

\begin{remark}
One can use the construction in the proof of Theorem~\ref{qtilde} to build a concrete maximal green sequence for $\widetilde{Q}$.  
\end{remark}

We obtain the following corollaries.  

\begin{corollary}\label{qtilde-cor1}
The quiver $\widetilde{Q}$ admits a maximal green sequence if and only if $Q$ admits a maximal green sequence. 
\end{corollary}

\begin{proof}
The forward direction follows from Theorem~\ref{thmMullerSubquiver} and the backward direction follows from Theorem~\ref{qtilde}.
\end{proof}

\begin{corollary}\label{Cor:min_length_Qtilde}
Let $\mu_{Q_{min}}$ be a minimal length maximal green sequence for $Q$.  Then the minimal length of a maximal green sequence for $\widetilde{Q}$ is 
$$\ell({\mu_{Q_{min}}})+l_{min}^1+l_{min}^2+\dots+ l_{min}^k - k$$
where $l_{min}^i$ is the minimal length of a maximal green sequence for $Q^i$.  
\end{corollary}

\begin{proof}
Let $\mu_{\widetilde{Q}_{min}}$ be a minimal length maximal green sequence for $\widetilde{Q}$.  Note, that by assumption $Q$ admits a maximal green sequence, so Coroallary~\ref{qtilde-cor1} implies that $\mu_{\widetilde{Q}_{min}}$ exists.  Since, $\mu_{\widetilde{Q}_{min}}$ is of minimal length we have 

$$\ell({\mu_{\widetilde{Q}_{min}}})\leq \ell({\mu_{Q_{min}}})+l_{min}^1+l_{min}^2+\dots+ l_{min}^k - k$$
by Theorem~\ref{qtilde}. 

To show the reverse inequality consider a maximal green sequence ${\bf i}$ for $\widetilde{Q}$.  Let $Q(1)$ be the full subquiver of $\widetilde{Q}$ on vertices $(\widetilde{Q}_0\setminus Q^1_0)\cup \{x_1\}$.   According to Theorem \ref{thmMullerSubquiver} the sequence ${\bf i}$ induces ${\bf i}(1)$ and ${\bf i}_1$ maximal green sequences for $Q(1)$ and $Q^1$ respectively.  By definition of ${\bf i}(1)$ and ${\bf i}_1$ given in the theorem we have 
$$\ell({\mu_{\bf i}})\geq \ell({\mu_{{\bf i}(1)}})+\ell({\mu_{{\bf i}_1}})-1$$
because $Q(1)$ and $Q^1$ have a single vertex in common and together they make up $\widetilde{Q}$.  

Now we apply the same argument to the quiver $Q(1)$ and its corresponding sequence ${\bf i}(1)$.  Let $Q(2)$ be the full subquiver of $Q(1)$ on vertices $(Q(1)_0\setminus Q^2_0)\cup \{x_2\}$.  Then by the same reasoning as above we deduce that 
$$\ell({\mu_{\bf i}})\geq \ell({\mu_{{\bf i}(2)}})+\ell({\mu_{{\bf i}_1}})+\ell({\mu_{{\bf i}_2}})-2$$
where ${\bf i}(2)$ and ${\bf i}_2$ are the resulting maximal green sequences for $Q(2)$ and $Q^2$ respectively.  

Continuing in this way we see that 
\begin{align*}
\ell({\mu_{\bf i}}) & \geq \ell({\mu_{{\bf i}(k)}})+\ell({\mu_{{\bf i}_1}})+\ell({\mu_{{\bf i}_2}})+ \dots + \ell({\mu_{{\bf i}_k}}) -k\\
&\geq \ell({\mu_{{Q}_{min}}})+ l^1_{min} + l^2_{min} + \dots + l^k_{min} - k
\end{align*}
where ${\bf i}_j$ is a maximal green sequence for $Q^j$ for $j\in[k]$ and $Q(k) = Q$ with the corresponding sequence ${\bf i}(k)$.  Since ${\bf i}$ was arbitrary this shows the reverse inequality and completes the proof. 
\end{proof}



\section{Quivers of mutation type $\mathbb{D}_n$}\label{Sec_type_D}

We now focus on the minimal length maximal green sequences of quivers of \textbf{mutation type $\mathbb{D}_n$}. By definition, this is the family of quivers $R$ belonging to the mutation class of an orientation of a $\mathbb{D}_n$ Dynkin diagram where $n \ge 3$. We recall the description of this mutation class, which was discovered by Vatne in \cite[Theorem 3.1]{vatne2010mutation}. There are four families of quivers that make up this mutation class. Vatne refers to these as \textbf{Type I}, \textbf{Type II}, \textbf{Type III}, and \textbf{Type IV} $\mathbb{D}_n$ quivers. Throughout this section, we will assume that $Q$ is a connected quiver. Additionally, we say a vertex $c$ of some quiver is a \textbf{connecting} vertex if it has degree at most 2, and if it has degree 2, it belongs to a 3-cycle of the same quiver.

\textbf{Type I:} A quiver $Q$ is of Type I (see Figure~\ref{type1}) if and only if $Q$ has the following properties: 
\begin{itemize}
\item it has a full subquiver of the form $a \leftrightarrow c \leftrightarrow b$ (we use the notation $a \leftrightarrow c$ to indicate that there exists an arrow in $Q$ connecting $a$ and $c$ and that arrow can be oriented in either direction) where vertices $a$ and $b$ each have degree 1 in $Q$,
\item the full subquiver $Q^\prime$ of $Q$ on the vertices $Q_0\backslash\{a,b\}$ is of mutation type $\mathbb{A}$, and
\item the vertex $c$ is a connecting vertex of $Q^\prime$.
\end{itemize}

\begin{figure}
\captionsetup{width=0.45\textwidth}
\centering
\begin{minipage}{.5\textwidth}
  \centering
  \def\svgwidth{120pt}
  \begin{tikzpicture}[scale=.4][
> = stealth, 
            shorten > = 4pt, ]
\node[draw=none, fill=none] at (-5,0) {$Q$:};
\node[draw=none, fill=none] at (5,0) {$Q^{\prime}$};
\node[draw, shape=circle, fill=black, scale=0.4] at (2,0) { };
\node[draw=none, fill=none] at (2, .6) {$c$};
\node[draw, shape=circle, fill=black, scale=0.4] at (0,1.5) { };
\node[draw=none, fill=none] at (0,2) {$a$};
\node[draw, shape=circle, fill=black, scale=0.4] at (0,-1.5) { };
\node[draw=none, fill=none] at (0,-2.1) {$b$};
\draw [<->] (0.2,1.3) -- node { } (1.8,0.2);
\draw [<->] (0.2,-1.3) -- node { } (1.8,-0.2);
\draw (5,0) ellipse (3cm and 1cm);
\end{tikzpicture}
  \caption{Type I quivers}
   \label{type1}
\end{minipage}%
\begin{minipage}{.5\textwidth}
\centering
\begin{tikzpicture}[scale=.4][
> = stealth, 
            shorten > = 4pt, ]         
\node[draw=none, fill=none] at (-11,0) {$Q$:};
\node[draw=none, fill=none] at (-5,0) {$Q^{(1)}$};
\node[draw=none, fill=none] at (5,0) {$Q^{(2)}$};
\node[draw, shape=circle, fill=black, scale=0.4] at (2,0) { };
\node[draw=none, fill=none] at (1.9, .7) {$d$};
\node[draw, shape=circle, fill=black, scale=0.4] at (-2,0) { };
\node[draw=none, fill=none] at (-2, .6) {$c$};
\node[draw, shape=circle, fill=black, scale=0.4] at (0,1.5) { };
\node[draw=none, fill=none] at (0,2) {$a$};
\node[draw, shape=circle, fill=black, scale=0.4] at (0,-1.5) { };
\node[draw=none, fill=none] at (0,-2.1) {$b$};
\draw [<-] (0.2,1.3) -- node { } (1.8,0.2);
\draw [<-] (0.2,-1.3) -- node { } (1.8,-0.2);
\draw [<-] (-1.8,0.2) -- node { } (-0.2,1.3);
\draw [<-] (-1.8,-0.2) -- node { } (-0.2,-1.3);
\draw [->] (-1.7,0) -- node { } (1.7,0);
\draw (5,0) ellipse (3cm and 1cm);
\draw (-5,0) ellipse (3cm and 1cm);
\end{tikzpicture}
\caption{Type II quivers}
\label{type2}
\end{minipage}
\end{figure}

%
%
%
%
%
%
%
%
%
%
%
%
%
%

\textbf{Type II:} A quiver $Q$ is of Type II (see Figure~\ref{type2}) if and only if $Q$ has the following properties:
\begin{itemize}
\item it has a full subquiver on the vertices $a,b,c,d$ of the form shown in Figure~\ref{type2},
\item the subquiver $Q^\prime$ obtained from $Q$ by removing vertices $a$ and $b$ and the arrow $c \to d$ consists of two connected components $Q^{(1)}$ and $Q^{(2)}$, each of which is of mutation type $\mathbb{A}$, and
\item the vertex $c$ (resp., $d$) is a connecting vertex of $Q^{(1)}$ (resp., $Q^{(2)}$).
\end{itemize}

\textbf{Type III:} A quiver $Q$ is of Type III (see Figure~\ref{type3}) if and only if $Q$ has the following properties:
\begin{itemize}
\item it has a full subquiver on the vertices $a, b, c, d$ that is an oriented 4-cycle as shown in Figure~\ref{type3},
\item the full subquiver $Q^\prime$ of $Q$ on the vertices $Q_0\backslash\{a,b\}$ consists of two connected components $Q^{(1)}$ and $Q^{(2)}$, each of which is of mutation type $\mathbb{A}$, and
\item the vertex $c$ (resp., $d$) is a connecting vertex of $Q^{(1)}$ (resp., $Q^{(2)}$).
\end{itemize}

\begin{figure}
\captionsetup{width=0.45\textwidth}
\centering
\begin{minipage}{.5\textwidth}
  \centering
  \def\svgwidth{120pt}
  \begin{tikzpicture}[scale=.4][
> = stealth, 
            shorten > = 4pt, ]
\node[draw=none, fill=none] at (-11,0) {$Q$:};
\node[draw=none, fill=none] at (-5,0) {$Q^{(1)}$};
\node[draw=none, fill=none] at (5,0) {$Q^{(2)}$};
\node[draw, shape=circle, fill=black, scale=0.4] at (2,0) { };
\node[draw=none, fill=none] at (1.9, .7) {$d$};
\node[draw, shape=circle, fill=black, scale=0.4] at (-2,0) { };
\node[draw=none, fill=none] at (-2, .6) {$c$};
\node[draw, shape=circle, fill=black, scale=0.4] at (0,1.5) { };
\node[draw=none, fill=none] at (0,2) {$a$};
\node[draw, shape=circle, fill=black, scale=0.4] at (0,-1.5) { };
\node[draw=none, fill=none] at (0,-2.1) {$b$};
\draw [<-] (0.2,1.3) -- node { } (1.8,0.2);
\draw [->] (0.2,-1.3) -- node { } (1.8,-0.2);
\draw [<-] (-1.8,0.2) -- node { } (-0.2,1.3);
\draw [->] (-1.8,-0.2) -- node { } (-0.2,-1.3);
\draw (5,0) ellipse (3cm and 1cm);
\draw (-5,0) ellipse (3cm and 1cm);
\end{tikzpicture}
  \caption{Type III quivers}
   \label{type3}
\end{minipage}%
\begin{minipage}{.5\textwidth}
\centering
\begin{tikzpicture}[scale=.4][
> = stealth, 
            shorten > = 4pt, ]
\node[draw=none, fill=none] at (-11,0) {$Q$:};
\node[draw=none, fill=none] at (0,4.5) {{\small $Q^{(1)}$}};
\node[draw=none, fill=none] at (7,2) {{\small $Q^{(2)}$}};
\node[draw=none, fill=none] at (-6.5,2) {{\small $Q^{(k)}$}};
\node[draw, shape=circle, fill=black, scale=0.4] at (2,0) { };
\node[draw=none, fill=none] at (2.3, .7) {$a_2$};
\node[draw, shape=circle, fill=black, scale=0.4] at (-2,0) { };
\node[draw=none, fill=none] at (-2.2, .6) {$a_1$};
\node[draw, shape=circle, fill=black, scale=0.4] at (0,1.5) { };
\node[draw=none, fill=none] at (1,1.5) {$b_1$};
\node[draw, shape=circle, fill=black, scale=0.4] at (4.5,-2.5) { };
\node[draw=none, fill=none] at (5,-3) {$a_3$};
\node[draw, shape=circle, fill=black, scale=0.4] at (4.5,0) { };
\node[draw=none, fill=none] at (5.2,-0.6) {$b_2$};
\node[draw, shape=circle, fill=black, scale=0.4] at (-4.5,-2.5) { };
\node[draw=none, fill=none] at (-5,-3) {$a_k$};
\node[draw, shape=circle, fill=black, scale=0.4] at (-4.5,0) { };
\node[draw=none, fill=none] at (-5,-0.6) {$b_k$};
\draw [<-] (0.2,1.3) -- node { } (1.8,0.2);
\draw [<-] (-1.8,0.2) -- node { } (-0.2,1.3);
\draw [->] (-1.7,0) -- node { } (1.7,0);
\draw [->] (4.2,0) -- node { } (2.2,0);
\draw [<-] (4.3,-2.3) -- node { } (2.2,-0.2);
\draw [<-] (4.5,-0.3) -- node { } (4.5,-2.2);
\draw [<-] (-4.3,0) -- node { } (-2.2,0);
\draw [->] (-4.3,-2.3) -- node { } (-2.2,-0.2);
\draw [->] (-4.5,-0.3) -- node { } (-4.5,-2.2);
\draw[rotate=90] (4, 0) ellipse (2.5cm and 1cm);
\draw[rotate=40] (6, -2.8) ellipse (2.5cm and 1cm);
\draw[rotate=-40] (-6, -2.8) ellipse (2.5cm and 1cm);
\draw[dotted][rotate=240]  (4.5,-2.5) arc (0:65:8.5cm);
\end{tikzpicture}
\caption{Type IV quivers}
\label{type4}
\end{minipage}
\end{figure}

\textbf{Type IV:} A quiver $Q$ is of Type IV (see Figure~\ref{type3}) if and only if $Q$ has the following properties:
\begin{itemize}
\item it has a full subquiver $R$ that is an oriented $k$-cycle where $k \ge 3$, $R_0 = \{a_1,a_2, \ldots, a_k\}$, and $R_1 = \{a_i \to a_{i+1}: i \in [k-1]\}\sqcup \{a_k\to a_1\}$,
\item for each arrow $\alpha \in R_1$, there may be a vertex $b_i \in Q_0\backslash R_0$ that is in a 3-cycle $b_i \to a_i \stackrel{\alpha}{\to} a_{i+1} \to b_i$, which is a full subquiver of $Q$, but there are no other vertices in $Q_0\backslash R_0$ that are connected to vertices of $R$,
\item the full subquiver $Q^\prime$ obtained from $Q$ by removing the vertices and arrows of the subquiver $R$ consists of the quivers $\{Q^{(i)}\}_{i \in [k]}$ some of which may be empty quivers, and where each quiver $Q^{(i)}$ is of mutation type $\mathbb{A}$ and has $b_i$ as a connecting vertex.
\end{itemize}


Next we present the main theorem of this section that describes the minimal length of maximal green sequences for quivers of type $\mathbb{D}_n$ depending on the particular type.  The proof for quivers of Type IV uses triangulations of surfaces and is presented in the following sections.

\begin{theorem}\label{Thm:typeD}
Let $Q$ be a quiver of mutation type $\mathbb{D}_n$. Let $\ell$ denote the length of a minimal length maximal green sequence of $Q$.
\begin{itemize}
\item[i)] If $Q$ is of Type I, then $\ell = n + |\{\text{3-cycles in } Q\}|$.
\item[ii)] If $Q$ is of Type II, then $\ell = n + |\{\text{3-cycles in } Q^{(1)}\}| + |\{\text{3-cycles in } Q^{(2)}\}| + 1.$
\item[iii)] If $Q$ is of Type III, then $\ell = n + |\{\text{3-cycles in } Q\}| + 2.$
\item[iv)] If $Q$ is of Type IV, then $\ell = n + \sum_{i = 1}^{k}|\{\text{3-cycles in } Q^{(i)}\}| + |\{i \in [k]: \text{deg}(a_i) = 4\}| + k - 2$.
\end{itemize}
\end{theorem}
\begin{proof}
i) Letting $Q(a)$ (resp., $Q(b)$) be the full subquiver of $Q$ on the vertex $a$ (resp., $b$), it is clear that these have minimal length maximal green sequences of length 1.

Let $Q^\prime$ be the full subquiver of $Q$ of mutation type $\mathbb{A}$ described in the definition of Type I quivers. By Theorem~\ref{thmtypeA}, there exists a minimal length maximal green sequence of $Q^\prime$ with length $|Q^\prime_0| + |\{\text{3-cycles in } Q^\prime\}|$. By the description of Type I quivers, we have that $\{\text{3-cycles in } Q\} = \{\text{3-cycle in } Q^\prime\}$. Now the result follows from Corollary~\ref{Cor:min_length_Qtilde} (or from Proposition~\ref{minlengthdirectsum}).

ii) One checks that the sequence $(d, a, b, c, d)$ is a maximal green sequence of the full subquiver of $Q$ on the vertices $a, b, c,$ and $d$. Since this full subquiver of $Q$ is not acyclic, Corollary~\ref{Cor:acyclic_result} implies that it does not have maximal green sequences of length 4. Thus $(d,a,b,c,d)$ is a minimal length maximal green sequence.

Next, since the quivers $Q^{(1)}$ and $Q^{(2)}$ are of mutation type $\mathbb{A}$, Theorem~\ref{thmtypeA} implies that there exist minimal length maximal green sequences $\textbf{i}_1 \in \text{MGS}(Q^{(1)})$ and $\textbf{i}_2 \in \text{MGS}(Q^{(2)})$ with lengths $|Q^{(1)}_0| + |\{\text{3-cycles in } Q^{(1)}\}|$ and  $|Q^{(2)}_0| + |\{\text{3-cycles in } Q^{(2)}\}|$, respectively. Now by Corollary~\ref{Cor:min_length_Qtilde}, we know that the minimal length of a maximal green sequence of $Q$ has the desired length.

iii) {One checks that $(a,c,b,d,c,a)$ is a minimal length maximal green sequence of the full subquiver of $Q$ on the vertices $a, b, c,$ and $d$.} As in Case ii), we apply Theorem~\ref{thmtypeA} to the quivers $Q^{(1)}$ and $Q^{(2)}$, and the desired result then follows from Corollary~\ref{Cor:min_length_Qtilde}.

iv)  Let $Q^{(i_1)}, \ldots, Q^{(i_t)}$ be the nonempty mutation type $\mathbb{A}$ quivers obtained from $Q$ by removing the vertices $a_1, \ldots, a_k$, and let $b_{i_1}, \ldots, b_{i_t}$ denote the corresonding connecting vertices. By Theorem~\ref{Dn_IV_lower_bound} and Theorem~\ref{Thm:Dn_short_seq}, we know that the full subquiver $Q^\prime$ of $Q$ on the vertices $a_1, \ldots, a_k, b_{i_1}, \ldots, b_{i_t}$ has a minimal length maximal green sequence of length $|Q^\prime_0| + |\{i \in [k]: \text{deg}(a_i) = 4\}| + k - 2.$ By Theorem~\ref{thmtypeA}, each quiver $Q^{(i_j)}$ has a minimal length maximal green sequence of length $|Q^{(i_j)}_0| + |\{\text{3-cycles in } Q^{(i_j)}\}|$. Now observe that $|Q^\prime_0| - t = k$. The result thus follows from Corollary~\ref{Cor:min_length_Qtilde}.
\end{proof}



\section{Quivers defined by triangulated surfaces}\label{Sec:QT}
To understand the minimal length maximal green sequences of $\mathbb{D}_n$ quivers of Type IV, it will be useful to think of such quivers as \textbf{signed adjacency quivers} of tagged triangulations of a once-punctured disk. In this section, we review the basic notions of quivers arising from triangulated surfaces, following \cite{fomin2008cluster}.

\subsection{Signed adjacency quivers} Let $\textbf{S}$ denote an oriented Riemann surface that may or may not have a boundary and let $\textbf{M} \subset \textbf{S}$ be a finite subset of $\textbf{S}$ where we require that for each component $\textbf{B}$ of $\partial \textbf{S}$ we have $\textbf{B}\cap \textbf{M} \neq \emptyset.$ We call the elements of $\textbf{M}$ \textbf{marked points}, we call the elements of $\textbf{M}\backslash(\textbf{M}\cap \partial \textbf{S})$ \textbf{punctures}, and we call the pair $(\textbf{S},\textbf{M})$ a \textbf{marked surface}. We require that $(\textbf{S},\textbf{M})$ is not one of the following \textbf{degenerate marked surfaces}: a sphere with one, two, or three punctures; a disc with one, two, or three marked points on the boundary; or a punctured disc with one marked point on the boundary. For the remainder of Section~\ref{Sec:QT}, we will work with a fixed marked surface $(\textbf{S},\textbf{M}).$

We define an $\textbf{arc}$ on \textbf{S} to be a curve $\gamma$ in \textbf{S} such that \begin{itemize} 
\item its endpoints are marked points; 
\item $\gamma$ does not intersect itself, except that its endpoints may coincide; 
\item except for the endpoints, $\gamma$ is disjoint from $\textbf{M}$ and from the boundary of $\textbf{S}$; 
\item $\gamma$ does not cut out an unpunctured monogon or and unpunctured digon. (In other words, $\gamma$ is not contractible into $\textbf{M}$ or onto the boundary of $\textbf{S}$.)\end{itemize} 
An arc $\gamma$ is considered up to isotopy relative to the endpoints of $\gamma$. 
We say two arcs $\gamma_1$ and $\gamma_2$ on \textbf{S} are \textbf{compatible} if they are isotopic relative to their endpoints to curves that are nonintersecting except possibly at their endpoints. An \textbf{ideal triangulation} of $(\textbf{S},\textbf{M})$ is defined to be a maximal collection of pairwise compatible arcs, denoted $\textbf{T}$. 

One can also move between different triangulations of a given marked surface $(\textbf{S},\textbf{M}).$ Let $\gamma$ be an arc in a triangulation $\textbf{T}$ that has no \textbf{self-folded triangles} (e.g., the region of \textbf{S} bounded by $\gamma_3$ and $\gamma_4$ in Figure~\ref{taggedarcs} is an example of a self-folded triangle).  Define the $\textbf{flip}$ of arc $\gamma$ to be the unique arc $\gamma^\prime \neq \gamma$ that produces a triangulation of $(\textbf{S},\textbf{M})$ given by $\textbf{T}^\prime = (\textbf{T}\backslash\{\gamma\}) \sqcup \{\gamma^\prime\}$ (see Figure~\ref{msflip}). 

If $(\textbf{S},\textbf{M})$ is a marked surface where $\textbf{M}$ contains punctures, there are triangulations of \textbf{S} that contain self-folded triangles. We refer to the arc $\gamma_3$ (resp., $\gamma_4$) shown in the self-folded triangle in Figure~\ref{taggedarcs} as a \textbf{loop} (resp., a \textbf{radius}). As the flip of a radius of a self-folded triangle is not defined, Fomin, Shapiro, and Thurston introduced \textbf{tagged arcs}, a generalization of arcs, in order to develop such a notion.

\begin{figure}[h]
$\begin{array}{rcl}\raisebox{-1in}{\includegraphics[scale=2]{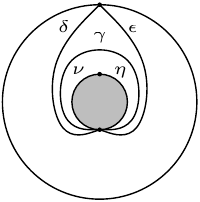}} & \raisebox{-.25in}{$\longleftrightarrow$} & \raisebox{-1in}{\includegraphics[scale=2]{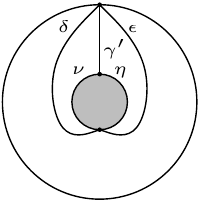}} \end{array}$
\caption{A flip connecting two triangulations of an annulus. The ideal quadrilateral of $\gamma$ is the disk with four marked points determined by the arcs $\{\gamma, \delta, \epsilon, \nu, \eta\}$ with the triangulation $\{\gamma\}$.}
\label{msflip}
\end{figure}

A \textbf{tagged arc} $\iota(\gamma)$ is obtained from an arc $\gamma$ that does not cut out a once-punctured monogon and ``tagging'' its ends either as $\textbf{plain}$ or $\textbf{notched}$ so that: \begin{itemize} \item an end of $\gamma$ lying on the boundary of $\textbf{S}$ is tagged plain; and \item both ends of a loop have the same tagging. \end{itemize} We use the symbol $\bowtie$ to indicate that an end of an arc is notched. We say two tagged arcs $\iota(\gamma_1)$ and $\iota(\gamma_2)$ are \textbf{compatible} if the following hold: \begin{itemize} 
\item Their underlying arcs $\gamma_1$ and $\gamma_2$ are the same, and the tagged arcs $\iota(\gamma_1)$ and $\iota(\gamma_2)$ have the same tagging at exactly one endpoint. 
\item Their underlying arcs $\gamma_1$ and $\gamma_2$ are distinct and compatible, and any common endpoints of $\iota(\gamma_1)$ and $\iota(\gamma_2)$ have the same tagging. \end{itemize} 
A \textbf{tagged triangulation} of $(\textbf{S},\textbf{M})$ is a maximal collection of pairwise compatible tagged arcs.  It follows from the construction that any arc in a tagged triangulation can be flipped.  In particular, the flip of $\gamma_4$ in Figure~\ref{taggedarcs} results in a new arc with one notched endpoint at the puncture and the other endpoint on the boundary where the arcs $\gamma_1$ and $\gamma_2$ meet.  This enables us to define the exchange graph $EG({\bf T})$ of a tagged triangulation {\bf T} whose vertices are tagged triangulations and two vertices are connected by an edge if the corresponding triangulations are related by a flip.   

It will be useful to keep track of the region of a triangulation where a flip of a tagged arc takes place. We refer to this region as an ideal quadrilateral. Define the \textbf{ideal quadrilateral} of a tagged arc $\gamma$ in a tagged triangulation $\textbf{T}$ to be the triangulated surface contained in $(\textbf{S},\textbf{M}$) that is
\begin{itemize}
\item an unpunctured disk with four marked points that consists of two triangles of $\textbf{T}$ that have $\gamma$ as a side if $\gamma$ is not a nontrivial tagged arc in a once-punctured digon (see Figure~\ref{msflip}) or
\item the once-punctured digon containing $\gamma$ and the triangulation of it determined by $\textbf{T}$ (see Figure~\ref{taggedarcs}).
\end{itemize}



\begin{figure}[h]
{\includegraphics[scale=1]{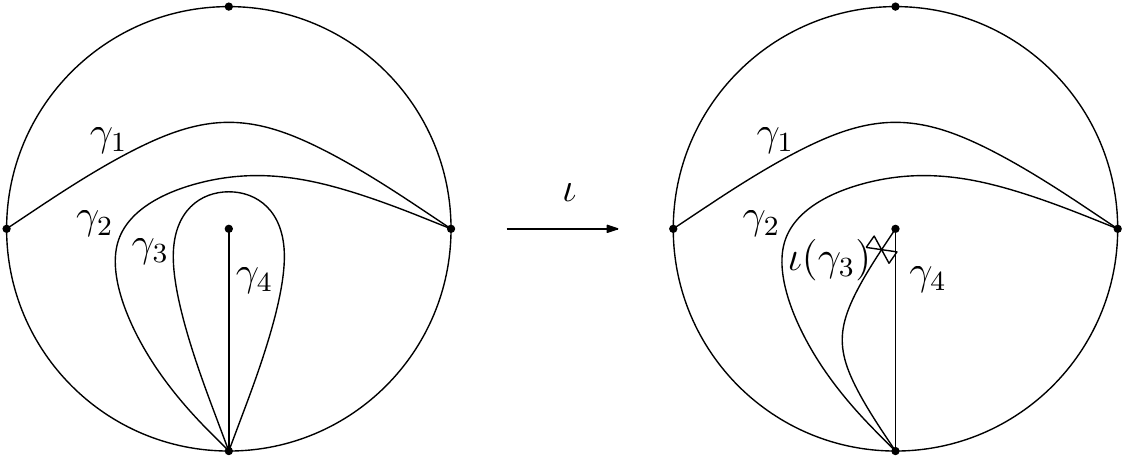}}
\caption{The map identifying a triangulation of a punctured disk as a tagged triangulation of a punctured disk. The ideal quadrilateral of $\iota(\gamma_3)$ in Figure~\ref{taggedarcs} is a once-punctured digon triangulated by $\{\iota(\gamma_3),\gamma_4\}.$}
\label{taggedarcs}
\end{figure}

Each triangulation $\textbf{T}$ of \textbf{S} defines a {\bf signed adjacency quiver} $Q_\textbf{T}$ by associating vertices to arcs and arrows based on oriented adjacencies (see Figure~\ref{msquiv}). More precisely, given an ideal triangulation $\textbf{T}$ consider a map $\pi : \textbf{T} \to \textbf{T}$ on the set of arcs defined as follows.  If $\gamma$ is a radius of a self-folded triangle then let $\pi(\gamma)$ be the corresponding loop, otherwise let $\pi(\gamma)=\gamma$.  Then the quiver $Q_\textbf{T}$ consists of vertices $i_{\gamma}$ for every $\gamma\in \textbf{T}$ and arrows $i_{\gamma}\to i_{\gamma'}$ for every non self-folded triangle with sides $\pi(\gamma)$ and $\pi(\gamma')$ such that $\pi(\gamma)$ follows $\pi(\gamma')$ in the clockwise order.  Finally, we remove a maximal collection of oriented 2-cycles from $Q_\textbf{T}$ to produce a 2-acyclic quiver.  Note that we defined $Q_\textbf{T}$ for an ideal triangulation $\textbf{T}$, however every tagged triangulation ${\bf T}$ can be transformed into an ideal triangulation ${\bf T}^{\circ}$ by forgetting the tagging of each arc and introducing self-folded triangles around puncture $p$ whenever two underlying arcs coincide and have different tagging at $p$.  More precisely, if there are two arcs attached to $p$ with different tagging, then replace the notched arc in ${\bf T}$ by a loop of a self-folded triangle enclosing $p$ in ${\bf T}^{\circ}$.  In this way, the definition of $Q_\textbf{T}$ can be extended to a tagged triangulation $\textbf{T}$ by setting $Q_\textbf{T} = Q_{\textbf{T}^{\circ}}$. 

The following theorem shows that flips are compatible with mutations of the associated quiver so that $EG({\bf T}) \cong EG(\widehat{Q}_{\bf T})$.  

\begin{theorem}
Given a tagged triangulation ${\bf T}$ and an arc $\gamma \in {\bf T}$, let ${\bf T}'$ be obtained from ${\bf T}$ after the flip of $\gamma$, then 
$$\mu_{i_\gamma}(Q_{\bf T}) = Q_{\bf{T}'}.$$
\end{theorem}


\begin{figure}[h]
$\begin{array}{rcl}\raisebox{-1in}{\includegraphics[scale=2]{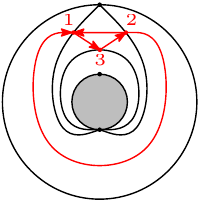}} & \raisebox{-.25in}{$\rightsquigarrow$} & \begin{xy} 0;<1pt,0pt>:<0pt,-1pt>:: 
(0,40) *+{1} ="0",
(60,40) *+{2} ="1",
(30,0) *+{3} ="2",
"0", {\ar"2"},
"2", {\ar"1"},
"1", {\ar@<-.5ex>"0"},
"1", {\ar@<.5ex>"0"},
\end{xy}\end{array}$
\caption{The quiver $Q_\textbf{T}$ defined by a triangulation \textbf{T}.}
\label{msquiv}
\end{figure}

\subsection{Shear coordinates}
Shear coordinates provide a way to orient $EG({\bf T})$ and establish a bijection with $\overrightarrow{EG}(\widehat{Q}_{\bf T})$.  This allows one to translate the notion of maximal green sequences into the geometric setting.   

A fixed orientation $\mathcal{O}$ of a surface $\textbf{S}$ induces an orientation $\mathcal{O}$ on each component of $\partial \textbf{S}$ such that the surface $\textbf{S}$ lies to the right of every component. Let $\gamma$ be a tagged arc in {\bf S}. The {\bf elementary lamination} $\ell({\gamma})$ is a curve that runs along $\gamma$ within a small neighborhood of it. If $\gamma$ has an
endpoint $M$ on the boundary $\partial {\bf S}$, then $\ell({\gamma})$ begins at a
point $M'\not\in {\bf M}$ on $\partial{\bf S}$ located near $M$ in the direction of $\mathcal{O}$, and proceeds along $\gamma$.  If $\gamma$ has an endpoint at a puncture $p$, then $\ell({\gamma})$ spirals
into $p$: clockwise if $\gamma$ is notched at $p$, and counterclockwise if it is tagged plain.  For example see Figure~\ref{fig:elem_lamination}.  A set of curves $L$ is a {\bf lamination} if it consists of a set of elementary laminations arising from some pairwise compatible tagged arcs.

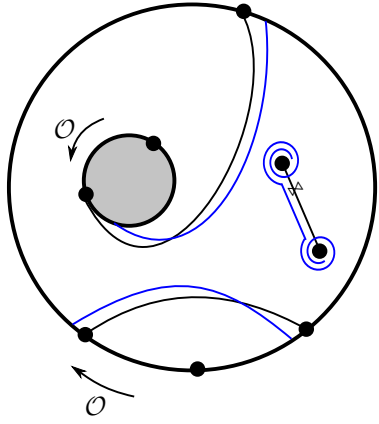
\begin{figure}[htb]

  \centering
  {\begingroup%
  \makeatletter%
  \def\svgwidth{140pt}
  \providecommand\color[2][]{%
    \errmessage{(Inkscape) Color is used for the text in Inkscape, but the package 'color.sty' is not loaded}%
    \renewcommand\color[2][]{}%
  }%
  \providecommand\transparent[1]{%
    \errmessage{(Inkscape) Transparency is used (non-zero) for the text in Inkscape, but the package 'transparent.sty' is not loaded}%
    \renewcommand\transparent[1]{}%
  }%
  \providecommand\rotatebox[2]{#2}%
  \ifx\svgwidth\undefined%
    \setlength{\unitlength}{153.60992736bp}%
    \ifx\svgscale\undefined%
      \relax%
    \else%
      \setlength{\unitlength}{\unitlength * \real{\svgscale}}%
    \fi%
  \else%
    \setlength{\unitlength}{\svgwidth}%
  \fi%
  \global\let\svgwidth\undefined%
  \global\let\svgscale\undefined%
  \makeatother%
  \begin{picture}(1,1.15121086)%
    \put(0,0){\includegraphics[width=\unitlength]{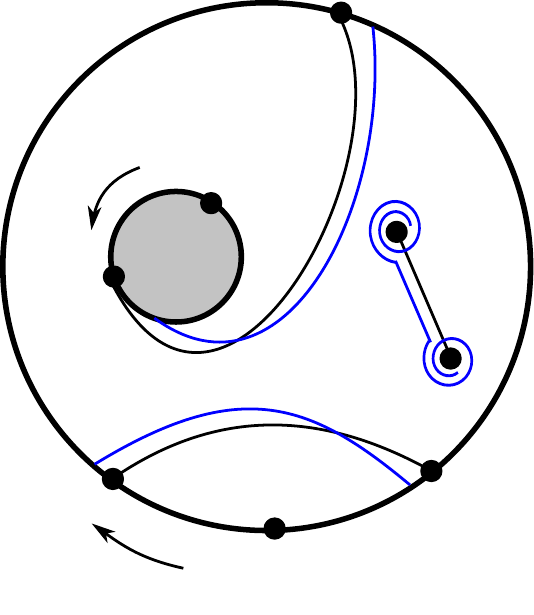}}%
    \put(0.21012815,0.08449305){\color[rgb]{0,0,0}\makebox(0,0)[lt]{\begin{minipage}{0.62867961\unitlength}\raggedright $\mathcal{O}$\end{minipage}}}%
    \put(0.12625487,0.83501751){\color[rgb]{0,0,0}\makebox(0,0)[lt]{\begin{minipage}{0.6342596\unitlength}\raggedright $\mathcal{O}$\end{minipage}}}%
    \put(0.75148429,0.65019976){\color[rgb]{0,0,0}\rotatebox{32.21378771}{\makebox(0,0)[lt]{\begin{minipage}{0.22176609\unitlength}\raggedright $\scriptstyle{\bowtie}$\end{minipage}}}}%
  \end{picture}%
\endgroup%
}
  \caption{Lamination corresponding to the given set of arcs.}
   \label{fig:elem_lamination}
\end{figure}

\begin{definition}\cite[Definition 13.3]{fomin2012cluster}\label{def:shearcoord}
Let $L$ be a lamination, and let ${\bf T}$ be an ideal
triangulation. For each arc $\gamma \in {\bf T}$, that is not a radius of a self-folded triangle, the corresponding {\bf shear coordinate} of $L$ with
respect to ${\bf T}$, denoted by $b_{\gamma}({\bf T}, L)$, is defined as a sum of contributions from all
intersections of curves in $L$ with $\gamma$. Specifically, such an intersection contributes +1
(resp., $-1$) to $b_{\gamma}({\bf T}, L)$ if the corresponding segment of a curve in $L$ cuts through the
quadrilateral surrounding $\gamma$ as shown in Figure~\ref{lamin_convention}  on the left (resp., right).  
\end{definition}

\begin{figure}[h]
$$\begin{array}{cccccccc}
\includegraphics[scale=1]{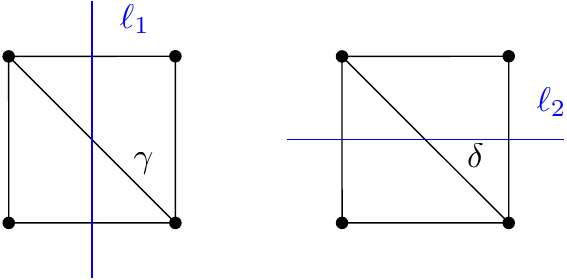}
\end{array}$$
\caption{The shear coordinate of $\gamma$ with respect to $\ell_1$ is $+1$, and the shear coordinate of $\delta$ with respect to $\ell_2$ is $-1$.}
\label{lamin_convention}
\end{figure}

Next, we describe how to extend the definition of shear coordinates to tagged triangulations.

Let ${\bf T}$ be a tagged triangulation.  The shear coordinates $b_{\gamma}({\bf T}, L)$ for $\gamma \in {\bf T}$ and some lamination $L$ are uniquely defined by
the following rules:
\begin{itemize}
\item [(i)] Suppose that tagged triangulations ${\bf T}_1$ and ${\bf T}_2$ coincide except that at a particular
puncture $p$, the tags of the arcs in ${\bf T}_1$ are all different from the tags of
their counterparts in ${\bf T}_2$. Suppose that laminations $L_1$ and $L_2$ coincide except
that each curve in $L_1$ that spirals into $p$ has been replaced in $L_2$ by a curve
that spirals in the opposite direction. Then $b_{\gamma_1}({\bf T}_1, L_1) = b_{\gamma_2}({\bf T}_2, L_2)$ for each arc $\gamma_1 \in {\bf T}_1$ and its counterpart $\gamma_2 \in {\bf T}_2$.
\item [(ii)] By performing tag-changing transformations as
above, we can convert any tagged triangulation into a tagged triangulation ${\bf T}$
that does not contain any notched arcs except possibly inside once-punctured
digons. Let ${\bf T}^{\circ}$ denote the ideal triangulation that is represented by such ${\bf T}$ as defined earlier.  Let $\gamma^{\circ}$
 be an arc in ${\bf T}^{\circ}$ that is not a radius of some self-folded triangle, and let $\gamma$ be the corresponding arc in ${\bf T}$.
Then, for a lamination $L$, define $b_{\gamma}({\bf T}, L)$ by applying Definition~\ref{def:shearcoord} to the arc $\gamma^{\circ}$ viewed inside the triangulation ${\bf T}^{\circ}$.
\end{itemize}

Note that if $\gamma^{\circ}$ is a radius of a self-folded triangle in ${\bf T}^{\circ}$
enveloping a puncture $p$,
then we can first apply the tag-changing transformation (i) to ${\bf T}$ at $p$, and then use
the rule (ii) to determine the shear coordinate in question.    

The following theorem establishes a relationship between shear coordinates and {\bf c}-vectors.  
\begin{theorem}
Consider a triangulation ${\bf T}$ with its associated quiver $Q_{\bf T}$.  Let $Q_{\bf T}' \in EG(\widehat{Q}_{\bf T})$ and let ${\bf T}' \in EG({\bf T})$ be the corresponding triangulation.  Then the {\bf c}-matrix $C_{Q_{\bf T}'}$ has entries, indexed by the arcs $\gamma \in {\bf T}$ and $\gamma' \in \bf{T}'$, as follows 
$$c_{\gamma', \gamma} = b_{\gamma'}({\bf T}', \ell({\gamma}))$$
where $\ell({\gamma})$ is the elementary lamination corresponding to $\gamma$.
\end{theorem}

With the above notation we say that an arc $\gamma' \in {\bf T}'$ is {\bf green} (resp., {\bf red}) if $c_{\gamma', \gamma}\geq 0$ (resp., $c_{\gamma', \gamma}\leq 0$) for all $\gamma \in {\bf T}$.  Note that the color of every arc in ${\bf T}'$ agrees with the color of the associated vertex in $Q_{\bf T}'$.   This enables us to define an oriented exchange graph $\overrightarrow{EG}({\bf T})$ of {\bf T} such that  $\overrightarrow{EG}({\bf T}) \cong \overrightarrow{EG}(\widehat{Q}_{\bf T})$.


Next, we define an automorphism $\varrho$ on the set of arcs in a given $(\textbf{S}, \textbf{M})$.  

\begin{definition}
Recall, that an orientation $\mathcal{O}$ of a surface $\textbf{S}$ induces an orientation $\mathcal{O}$ on each component of $\partial \textbf{S}$ such that the surface $\textbf{S}$ lies to the right of every component.  Given a marked point $M\in \partial \textbf{S}$ let $\rho(M)$ be the next marked point in $\partial \textbf{S}$ along $\mathcal{O}$, and let $\delta_M$ be the corresponding boundary segment connecting $M$ and $\rho(M)$.  If $M\in \textbf{M}$ is a puncture, then let $\rho(M)=M$ and the segment $\delta_M$ equal a single point $M$.  Given an arc $\gamma$ in $\textbf{S}$ with endpoints $M_1, M_2$ let $\varrho(\gamma)$ be the arc with endpoints $\rho(M_1), \rho(M_2)$ that follows along the segments $\delta_{M_1}$, $\gamma$ and then $\delta_{M_2}$ in order.  Furthermore, if $M_i$ is a puncture for $i=1, 2$ then the tagging of $\varrho(\gamma)$ is different from that of $\gamma$ at $M_i$.   For example see Figure~\ref{fig:rho}.
\end{definition}

\begin{figure}[htb]

  \centering
  {\begingroup%
  \makeatletter%
  \def\svgwidth{300pt}
  \providecommand\color[2][]{%
    \errmessage{(Inkscape) Color is used for the text in Inkscape, but the package 'color.sty' is not loaded}%
    \renewcommand\color[2][]{}%
  }%
  \providecommand\transparent[1]{%
    \errmessage{(Inkscape) Transparency is used (non-zero) for the text in Inkscape, but the package 'transparent.sty' is not loaded}%
    \renewcommand\transparent[1]{}%
  }%
  \providecommand\rotatebox[2]{#2}%
  \ifx\svgwidth\undefined%
    \setlength{\unitlength}{418.18134592bp}%
    \ifx\svgscale\undefined%
      \relax%
    \else%
      \setlength{\unitlength}{\unitlength * \real{\svgscale}}%
    \fi%
  \else%
    \setlength{\unitlength}{\svgwidth}%
  \fi%
  \global\let\svgwidth\undefined%
  \global\let\svgscale\undefined%
  \makeatother%
  \begin{picture}(1,0.4137449)%
    \put(0,0){\includegraphics[width=\unitlength]{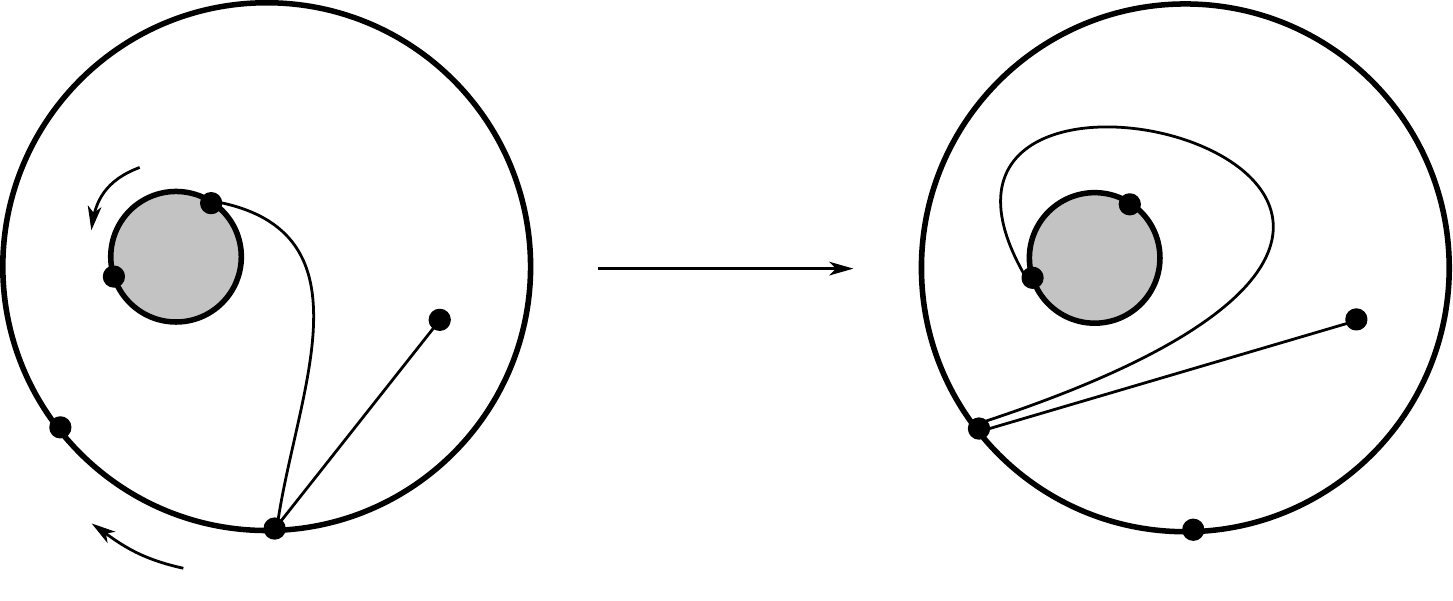}}%
    \put(0.4918099,0.2635868){\color[rgb]{0,0,0}\makebox(0,0)[lt]{\begin{minipage}{0.14894428\unitlength}\raggedright $\varrho$\end{minipage}}}%
    \put(0.06752372,0.03060515){\color[rgb]{0,0,0}\makebox(0,0)[lt]{\begin{minipage}{0.15304367\unitlength}\raggedright $\rho$\end{minipage}}}%
    \put(0.04444455,0.30049675){\color[rgb]{0,0,0}\makebox(0,0)[lt]{\begin{minipage}{0.15304367\unitlength}\raggedright $\rho$\end{minipage}}}%
    \put(0.91193231,0.20181423){\color[rgb]{0,0,0}\rotatebox{-70.0165464}{\makebox(0,0)[lt]{\begin{minipage}{0.05826649\unitlength}\raggedright {\Small $\bowtie$}\end{minipage}}}}%
  \end{picture}%
\endgroup%
}
  \caption{The affect of $\varrho$ on arcs of $(\textbf{S}, \textbf{M})$.}
   \label{fig:rho}
\end{figure}

The following theorem is an consequence of \cite[Theorem 3.8]{brustle2015tagged} and the remarks in the introduction of \cite{brustle2015tagged}. It will be useful for our construction of minimal length maximal green sequences.

\begin{theorem}\label{univ_tagged_thm}
Let $\textbf{T}$ be a triangulation of a marked surface $(\textbf{S}, \textbf{M})$. If $\textbf{i} \in \text{MGS}(Q_\textbf{T})$ and $\underline{\mu}$ is the corresponding sequence of flips, then $\underline{\mu}(\textbf{T}) = \varrho(\textbf{T})$.
\end{theorem}

\section{Type IV quivers}\label{Sec_type_IV}

Let $Q$ be a quiver of Type IV where each of its subquivers $Q^{(i)}$ with $i = 1, \ldots, k$ either consists of a single vertex or is the empty quiver. Any such quiver $Q$ is the signed adjacency quiver of a triangulation $\textbf{T} = \{\gamma: \gamma \in \textbf{T}\}$ of the once-punctured disk $(\textbf{S}, \textbf{M})$ where $\textbf{M} = \textbf{M}^\prime \sqcup \{p\}$ and where $p$ denotes the unique puncture in the marked surface. A triangulation $\textbf{T}$ giving rise to such a quiver $Q$ is defined by the following properties (see Figure~\ref{typeIVtriang}): \begin{itemize} \item if $\gamma$ is not connected to $p$, it belongs to a triangle whose other two sides are boundary arcs and \item if $\gamma$ is connected to $p$, it is plain at $p$. \end{itemize} The triangulation \textbf{T} is the unique triangulation satisfying $Q_{\textbf{T}} = Q$, up to the action of the universal tagged rotation on the once-punctured disk and up to the action of simultaneously changing the tagging of all ends of arcs connected to $p$. It immediately follows, that the vertices of the quivers $Q^{(i)}$ with $i = 1, \ldots, k$ are in bijection with the triangles of $\textbf{T}$ containing two boundary arcs. We will construct a minimal length maximal green sequence of $Q$ in terms of the triangulation \textbf{T} satisfying $Q_{\textbf{T}} = Q$ just described.

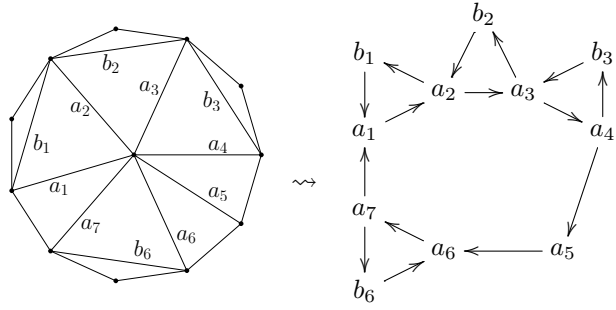
\begin{figure}

$$\begin{array}{rcl}\includegraphics[scale=.5]{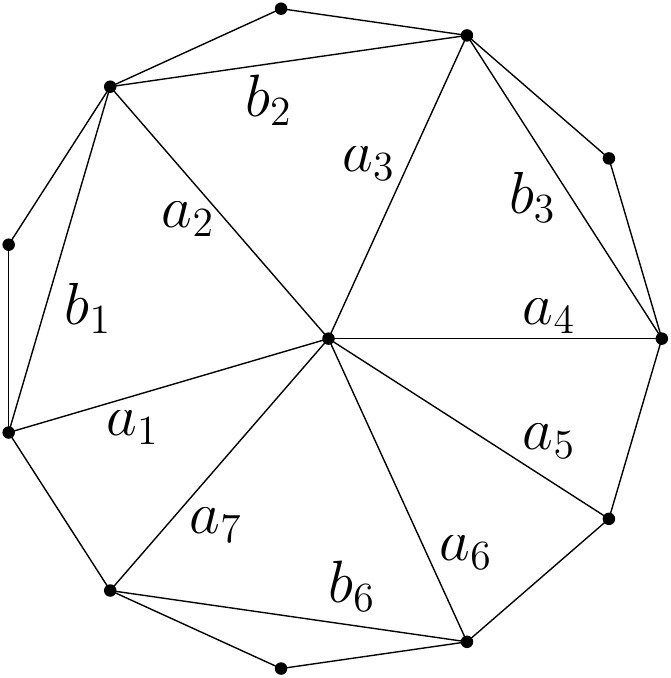} & \raisebox{.5in}{$\leadsto$} & \begin{xy} 0;<1pt,0pt>:<0pt,-1pt>:: 
(0,-55) *+{a_1} ="0",
(0,-85) *+{b_1} ="1",%
(30,-70) *+{a_2} ="2",
(45,-100) *+{b_2} ="3",
(60,-70) *+{a_3} ="4",
(90,-85) *+{b_3} ="5",
(90,-55) *+{a_4} ="6",
(75,-10) *+{a_5} ="7",
(30,-10) *+{a_6} ="8",
(0,5) *+{b_6} ="9",
(0,-25) *+{a_7} ="10",
"1", {\ar"0"},
"0", {\ar"2"},
"10", {\ar"0"},
"2", {\ar"1"},
"3", {\ar"2"},
"2", {\ar"4"},
"4", {\ar"3"},
"5", {\ar"4"},
"4", {\ar"6"},
"6", {\ar"5"},
"6", {\ar"7"},
"7", {\ar"8"},
"9", {\ar"8"},
"8", {\ar"10"},
"10", {\ar"9"},
\end{xy} \end{array}$$
\caption{The triangulation \textbf{T} that gives rise to $Q.$}
\label{typeIVtriang}
\end{figure}

\subsection{Lower bound on lengths of maximal green sequences}  
In this section we establish a lower bound on the lengths of maximal green sequences for quivers $Q_{\bf T}$ as defined above.  Without loss of generality we give the same labels to the vertices of $Q_{\textbf{T}}$ and the corresponding arcs of $\textbf{T}$.  Given a vertex $\gamma$ of $Q_{\textbf{T}}$ we write $\mu_{\gamma}(Q_{\textbf{T}}) = Q_{\mu_{\gamma} ({\textbf{T}})}$, where $\mu_{\gamma}({\textbf{T}})$ denotes the triangulation corresponding to the flip of the arc $\gamma \in {\textbf{T}}$.  
Throughout this section we fix a maximal green sequence $\mu_I$ for $Q_{\textbf{T}}$.  It follows from Theorem \ref {univ_tagged_thm} that $\mu_I ({\textbf{T}}) = \varrho({\textbf{T}})=\{\varrho(\gamma )\mid \gamma \in {\textbf{T}}\}$.  Given an arc $\gamma \in {\textbf{T}}$ let $\bar{\gamma}$ denote the same arc with a different tagging if possible, otherwise let $\bar{\gamma}=\gamma$.    

Let  $\mu_I=\mu_{i_q}\dots\mu_{i_2}\mu_{i_1}$, then we define a triangulation $\mathcal{B}^r_{\mu_I}({\textbf{T}}) = \mu_{i_r}\dots \mu_{i_1}({\textbf{T}})$ for $r\in[q]$ and two sets 

\[ \mathcal{B}_{\mu_I}({\textbf{T}}) = \bigcup ^q _{i=1} \mathcal{B}_{\mu_I}^i ({\textbf{T}}) \setminus {\textbf{T}} \hspace{2cm} \mathcal{D}_{\mu_I}({\textbf{T}})=\bigcup ^q _{i=1} \mathcal{B}_{\mu_I}^i ({\textbf{T}})\setminus \varrho({\textbf{T}}). \] 

In particular, we see that $\varrho({\textbf{T}}) = \mathcal{B}_{\mu_I}^q ({\textbf{T}})$ for a maximal green sequence $\mu_I$.  Also, the sets $\mathcal{B}^i_{\mu_I}({\textbf{T}})$ and $\mathcal{B}_{\mu_I}^{i-1}({\textbf{T}})$ differ only by a single arc, hence the length of $\mu_I$ is equal to the size of $\mathcal{B}_{\mu_I}({\textbf{T}})$.

\begin{definition}
Consider a map on the set of arcs

\[ f_{\mu_I} : \mathcal{D}_{\mu_I}({\textbf{T}}) \to \mathcal{B}_{\mu_I}({\textbf{T}})\]
defined as follows.  Given an arc $\gamma\in \mathcal{D}_{\mu_I}({\textbf{T}})$, there exists some $r\in [q-1]$ such that $\gamma\in\mathcal{B}_{\mu_I}^r({\textbf{T}})$.    Let $s$ be the largest integer such that  $\gamma \in \mathcal{B}_{\mu_I}^s ({\textbf{T}})$.   Observe that $s$ exists because $\gamma\not\in \varrho({\textbf{T}})$.   Then $f_{\mu_I}(\gamma)$ is defined as the unique arc belonging to $\mathcal{B}_{\mu_I}^{s+1}({\textbf{T}})\setminus \mathcal{B}_{\mu_I}^{s}({\textbf{T}})$.  In other words, $f_{\mu_I}(\gamma)$ is the arc obtained as a result of flipping $\gamma$ when performing the mutation sequence $\mu_I$.  To simplify the notation, we omit the subscript $\mu_I$ in $f_{\mu_I}$.

Similarly, given a nonempty set of arcs $S$ such that $S\subset\mathcal{B}_{\mu_I}^r({\textbf{T}})$ for some $r\in[q-1]$.  We define $\phi(S)$ to be $f(\gamma)$ where $\gamma \in S$ is the first arc in $S$ to be flipped along $\mu_{i_q} \dots \mu_{i_{r+1}}$.  More precisely, $\phi(S)$ is the unique arc in $\mathcal{B}^{s+1}_{\mu_I}({\textbf{T}})\setminus \mathcal{B}^{s}_{\mu_I}({\textbf{T}})$ where $s>r$ is the largest integer such that $S\subset \mathcal{B}_{\mu_I}^{s}({\textbf{T}})$.  We will often consider sets of arcs $S\subset \mathcal{B}^r_{\mu_I}({\textbf{T}})$ that define a triangle in $\mathcal{B}^r_{\mu_I}({\textbf{T}})$, and we will refer to these sets as triangles.  
\end{definition}

The following lemma says that once an arc $\gamma$ that appears along $\mu_I$ is flipped then it cannot appear again.

\begin{lemma}\label{arcs}
If $f(\gamma)\in\mathcal{B}^{r+1}_{\mu_I}(\textup{\textbf{T}})\setminus \mathcal{B}^{r}_{\mu_I}(\textup{\textbf{T}})$, then $\gamma\not\in \bigcup_{i=r+1}^{q}\mathcal{B}^i_{\mu_I}(\textup{\textbf{T}})$.  
\end{lemma}

\begin{proof}
Given a once punctured polygon $({\bf S}, {\bf M})$ there exists the corresponding cluster category $\mathcal{C}$ of type $\mathbb{D}_n$ such that indecomposable objects of $\mathcal{C}$ are in bijection with tagged arcs in $({\bf S}, {\bf M})$.  Furthermore, triangulations of the surface are in bijection with cluster-tilting objects in $\mathcal{C}$, \cite{schiffler2008geometric}. Thus, given a triangulation ${\bf T}$ of the surface we obtain the associated cluster-tilted algebra $B_{\bf T}$, and it follows from \cite{adachi2014tau} that cluster-tilting objects in $\mathcal{C}$ are in bijection with support $\tau$-tilting modules of $B_{\bf T}$.  Combining the two results we obtain a correspondence between triangulations of $({\bf S}, {\bf M})$ and support $\tau$-tilting modules of $B_{\bf T}$.  Moreover, this correspondence behaves well under mutations, so that a maximal green sequence for $Q_{\bf T}$ yields a maximal chain of support $\tau$-tilting modules in $B_{\bf T}$.  

Now to prove the lemma, it suffices to show that an indecomposable $\tau$-rigid module $M$ over an algebra $B_{\bf T}$ does not appear twice along a maximal chain of support $\tau$-tilting modules
$$(B_{\bf T})_{B_{\bf T}}=T_1, T_2,\dots, T_q=0.$$  
If $M$ belongs to $\text{add}\, T_r$ for some $r$, and there is a negative mutation of $T_r$ at $M$ resulting in the new support $\tau$-tilting module $T_{r+1}$, then by definition of mutation it follows that $M$ does not belong to $\text{Fac}\, T_{r+1}$.  Since in a maximal green sequence we have that $\text{Fac}\, T_i \subset \text{Fac}\, T_{i-1}$ for all $i$, we deduce that $M$ does not belong to $\text{add}\, T_i$ for any $i > r$.  Therefore, $M$ cannot appear again, so the corresponding arc of $({\bf S}, {\bf M})$ cannot appear more then once along a maximal green sequence. 
\end{proof}

Recall that we are interested in quivers $Q_{\textbf{T}}$ of Type IV, where each $Q^{(i)}$ is either empty or a single vertex.  The corresponding triangulation ${\textbf{T}}$ consists of a set of arcs $\mathcal{K}=\{a_i\mid i\in [k]\}$ that are attached to the puncture $p$, and a set of arcs $\mathcal{T}=\{b_i \mid Q^{(i)} \text{ is nonempty}\}$ that are not attached to $p$ and lie in a triangle that contains two boundary segments.  Here we consider the subscripts $i,j$ in the notation $a_i, b_j$ modulo $k$.  Moreover, we assume that all arcs in $\mathcal{K}$ are tagged plain.

Thus, $\abs{\mathcal{K}}=k$ and let $\abs{\mathcal{T}}=t$.  Hence, ${\textbf{T}} = \mathcal{K}\sqcup \mathcal{T}$ and $\abs{{\textbf{T}}} = k+t$.  Let 

$$\mathcal{M}=\{b_i \in \mathcal{T} \mid \varrho^2(b_i )\in \mathcal{T}\}$$  

\noindent and $\abs{\mathcal{M}}=m$.

\begin{definition}
Let $\mathcal{C}_{\mu_I}({\textbf{T}})$ be a collection of arcs such that 
\begin{itemize} 
\item $\mathcal{C}_{\mu_I}({\textbf{T}})\subset \mathcal{B}_{\mu_I}({\textbf{T}})\setminus \varrho({\textbf{T}})$ 
\item the arcs of $\mathcal{C}_{\mu_I}({\textbf{T}})$ are not attached to $p$
\item $f(\gamma)\in \mathcal{C}_{\mu_I}({\textbf{T}})$ whenever $\gamma\in \mathcal{D}_{\mu_I}({\textbf{T}})$ is attached to $p$ and tagged plain.     
\end{itemize}
\end{definition}

\begin{lemma}\label{01}
 $\abs{\mathcal{C}_{\mu_I}(\textup{\textbf{T}})}\geq k-2$.  
\end{lemma}

\begin{proof}
Observe that for any maximal green sequence $\mu_I$ for $Q_{\textbf{T}}$  there exists an arc $\alpha \in \mathcal{D}_{\mu_I}({\textbf{T}})$ such that $f(\alpha)$ is the first tagged arc to appear along the mutation sequence $\mu_I ({\textbf{T}})$.  That is, $f(\alpha)\in\mathcal{B}_{\mu_I}^s({\textbf{T}})\setminus \mathcal{B}^{s-1}_{\mu_I}({\textbf{T}})$ and the set $\bigcup _{i=1}^{s-1}\mathcal{B}^i_{\mu_I}({\textbf{T}})$ does not contain any tagged arcs.  Moreover, there are precisely two plain tagged arcs attached to the puncture in the triangulation $\mathcal{B}^{s-1}_{\mu_I}({\textbf{T}})$.  One of these arcs is $\alpha$, and call the other arc $\alpha'$.  This implies that all arcs in $\mathcal{K'}=\mathcal{K}\setminus\{\alpha, \alpha'\}$ need to be flipped prior to a flip in $\alpha$. 

Given $a_i\in\mathcal{K'}$ we make the following definition.  

\[ H(a_i) = \begin{cases} 
      f(a_i) & f(a_i)\not\in\varrho({\textbf{T}}) \\
      f(\varrho^{-1}(\bar{a_i})) & f(a_i)\in\varrho({\textbf{T}}) \;\text{and}\; \varrho^{-1}(\bar{a_i})\not\in \{\alpha,\alpha'\}\\
      NA & \text{otherwise} 
   \end{cases}\]

We claim that when defined $H(a_i)\in \mathcal{C}_{\mu_I}({\textbf{T}})$.  First, observe that by construction whenever $a_i\in \mathcal{K}'$ then $f(a_i)$ is not attached to the puncture.  Suppose $a_i \in \mathcal{K'}$ such that $f(a_i)\not\in \varrho({\textbf{T}})$. By definition $a_i$ is tagged plain and attached to $p$ while $f(a_i)$ is not attached to $p$, which implies that $f(a_i)\in \mathcal{C}_{\mu_I}({\textbf{T}})$.

Now, suppose that $a_i\in \mathcal{K'}$ such that $f(a_i)\in\varrho({\textbf{T}})$ and $\varrho^{-1}(\bar{a_i})\not\in \{\alpha,\alpha'\}$ and we want to show that $f(\varrho^{-1}(\bar{a}_i))\in\mathcal{C}_{\mu_I}({\textbf{T}})$.  Since $f(a_i)$ is not attached to the puncture and $f(a_i)\in\varrho({\textbf{T}})$ we conclude that $f(a_i)\in \varrho(\mathcal{T})$.  Hence, $a_i$ intersects some arc in $\varrho(\mathcal{T})$. However, the particular structure of the triangulation ${\textbf{T}}$ implies that there exists a unique arc $b_{i-1}\in \mathcal{T}$ such that $a_i$ intersects $\varrho(b_{i-1})$, see Figure \ref{fig:ha}.  Therefore, we have that $f(a_i)=\varrho(b_{i-1})$, meaning $a_i$ is flipped to $\varrho(b_{i-1})$.  This implies, that there exists a triangulation $\mathcal{B}^r_{\mu_I}({\textbf{T}})$ where $r<s-1$ such that $\{a_i, \varrho(\bar{a}_i), \varrho^{-1}(\bar{a}_i)\}\subset \mathcal{B}^r_{\mu_I}({\textbf{T}})$.  

\begin{figure}[htb]
  \centering
  {\begingroup%
  \makeatletter%
  \def\svgwidth{150pt}
  \providecommand\color[2][]{%
    \errmessage{(Inkscape) Color is used for the text in Inkscape, but the package 'color.sty' is not loaded}%
    \renewcommand\color[2][]{}%
  }%
  \providecommand\transparent[1]{%
    \errmessage{(Inkscape) Transparency is used (non-zero) for the text in Inkscape, but the package 'transparent.sty' is not loaded}%
    \renewcommand\transparent[1]{}%
  }%
  \providecommand\rotatebox[2]{#2}%
  \ifx\svgwidth\undefined%
    \setlength{\unitlength}{178.7390378bp}%
    \ifx\svgscale\undefined%
      \relax%
    \else%
      \setlength{\unitlength}{\unitlength * \real{\svgscale}}%
    \fi%
  \else%
    \setlength{\unitlength}{\svgwidth}%
  \fi%
  \global\let\svgwidth\undefined%
  \global\let\svgscale\undefined%
  \makeatother%
  \begin{picture}(1,0.90036846)%
    \put(0,0){\includegraphics[width=\unitlength]{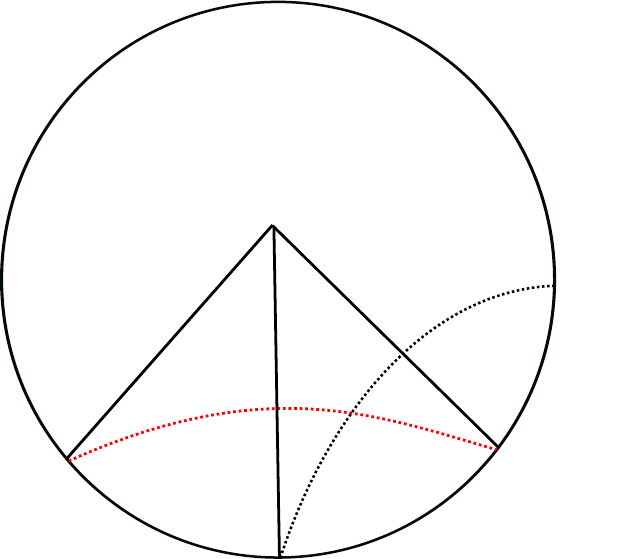}}%
    \put(0.10788463,0.46382257){\color[rgb]{0,0,0}\makebox(0,0)[lt]{\begin{minipage}{0.39610821\unitlength}\raggedright $\varrho(\bar{a}_i)$\end{minipage}}}%
    \put(0.35181567,0.3452139){\color[rgb]{0,0,0}\makebox(0,0)[lt]{\begin{minipage}{0.39387031\unitlength}\raggedright $a_i$\end{minipage}}}%
    \put(0.50623074,0.55333855){\color[rgb]{0,0,0}\makebox(0,0)[lt]{\begin{minipage}{0.4386283\unitlength}\raggedright $\varrho^{-1}(\bar{a}_i)$\end{minipage}}}%
    \put(0.74344808,0.38773399){\color[rgb]{0,0,0}\makebox(0,0)[lt]{\begin{minipage}{0.67360774\unitlength}\raggedright $b_{i-1}$\end{minipage}}}%
    \put(0.20187641,0.19975043){\color[rgb]{0,0,0}\makebox(0,0)[lt]{\begin{minipage}{0.51247898\unitlength}\raggedright ${\color{red}\varrho(b_{i-1})}$\end{minipage}}}%
  \end{picture}%
\endgroup%
}
  \caption{Local configuration around $a_i$ when $f(a_i)\in \varrho(T)$.}
   \label{fig:ha}
\end{figure}
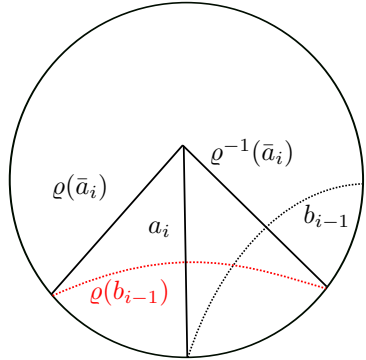

Next we show that if ${\varrho}^{-1}(\bar{a}_{i}) \not \in \{\alpha, \alpha'\}$ then $f({\varrho}^{-1}(\bar{a}_{i})) \in \mathcal{C}_{\mu_I}({\textbf{T}})$.  Note that $\varrho^{-1}(\bar{a}_{i})$ is tagged plain, attached to $p$, and does not belong to ${\textbf{T}}$ as it intersects $b_{i-1}$.  Moreover, since ${\varrho}^{-1}(\bar{a}_{i}) \not \in \{\alpha, \alpha'\}$ then $f({\varrho}^{-1}(\bar{a}_{i}))$ cannot be a notched arc attached to the puncture.  Hence, $f({\varrho}^{-1}(\bar{a}_{i}))$ is not attached to the puncture.   Finally, since $\varrho^{-1}(\bar{a}_i)$ does not intersect any arc in $\varrho(\mathcal{T})$, it follows that its flip $f({\varrho}^{-1}(\bar{a}_{i}))$ does not belong to $\varrho({\textbf{T}})$.  This implies that $f({\varrho}^{-1}(\bar{a}_{i}))\in \mathcal{C}{\mu_I}({\textbf{T}})$.  It also shows the claim that $H(a_i)\in\mathcal{C}_{\mu_I}({\textbf{T}})$ if $H(a_i)$ is defined.  Moreover, it is easy to see that $H(a_i)=H(a_j)$ if and only if $a_i=a_j$.  

Let 

\[\mathcal{H}=\{H(a_i)\mid a_i\in \mathcal{K'} \;\text{and}\; H(a_i) \; \text{is defined}\}\]

and by above we know that $\mathcal{H}\subset\mathcal{C}_{\mu_I}({\textbf{T}})$.  Since, $\abs{\mathcal{K'}}\geq k-2$ and $H(a_i)$ is defined for every $a_i\in\mathcal{K'}$ with at most two exceptions, we have $\abs{\mathcal{H}}\geq k-4$.

We claim that $\abs{\mathcal{H}}\geq k-2$ which in turn implies that $\abs{\mathcal{C}_{\mu_I}({\textbf{T}})}\geq k-2$ and completes the proof of the lemma.  First, we observe that $H(a_i)$ is not defined only if $f(a_i)\in\varrho({\textbf{T}})$ and $\varrho^{-1}(\bar{a_i})\in\{\alpha,\alpha'\}$.   From Figure~\ref{fig:ha} we can see that $\varrho^{-1}(\bar{a_i})$ intersects $b_{i-1}\in {\textbf{T}}$, so in particular $\varrho^{-1}(\bar{a_i})$ does not belong to $\mathcal{K}$.  

If $H$ is defined for all arcs of $\mathcal{K}'$, then the claim holds.  Suppose $a_j\in\mathcal{K'}$ is the unique arc such that $H(a_j)$ is not defined.  Then $\abs{\mathcal{K'}}\geq k-1$ since $\varrho^{-1}(\bar{a}_j)$ equals $\alpha$ or $\alpha'$ and it does not belong to $\mathcal{K}$.  By assumption, $H$ is defined for every arc in $\mathcal{K'}\setminus \{a_j\}$ and $\abs{\mathcal{K'}\setminus\{a_j\}}\geq k-2$, since $\abs{\mathcal{K'}}\geq k-1$.  This shows the claim in the case when $H$ is not defined for exactly one arc in $\mathcal{K}'$.  A similar argument shows that the claim holds whenever there are precisely two arcs $a_j, a_{j'}$ such that $H(a_j), H(a_{j'})$ are not defined.  Therefore, we obtain $\abs{\mathcal{H}}\geq k-2$.  
\end{proof}

\begin{remark}\label{rm 04}
It follows from the definition of $\mathcal{H}$ that for every $\sigma\in \{\alpha,\alpha'\}$ such that $\sigma\not\in {\textbf{T}}$, the cardinality of $\mathcal{H}$ increases by 1 unless $\sigma = \varrho^{-1}({\bar{a}_i})$ and $f(a_i)=\varrho(b_{i-1})$ for some $i$.   
\end{remark}

\begin{definition}
Define $\tilde{\mathcal{C}}_{\mu_I}({\textbf{T}})$ to be any subset of $\mathcal{H}$ of cardinality $k-2$.  In particular, we have the following inclusions.
$$\tilde{\mathcal{C}}_{\mu_I}({\textbf{T}})\subset \mathcal{H}\subset \mathcal{C}_{\mu_I}({\textbf{T}})$$
\end{definition}

\begin{remark}\label{rm 02}
By Remark \ref{rm 04}, the set $\mathcal{H}\setminus \tilde{\mathcal{C}}_{\mu_I}({\textbf{T}})$ is nonempty if there exists $\sigma\in \{\alpha,\alpha'\}$ such that $\sigma\not\in {\textbf{T}}$ unless $\sigma = \varrho^{-1}({\bar{a}_i})$ and $f(a_i)=\varrho(b_{i-1})$ for some $i$.   
\end{remark}

We also define 
$$\tilde{\mathcal{B}}_{\mu_I}({\textbf{T}}) = \mathcal{B}_{\mu_I}({\textbf{T}})\setminus (\tilde{\mathcal{C}}_{\mu_I}({\textbf{T}})\cup \varrho({\textbf{T}})).$$ 
With this notation we can write $\mathcal{B}_{\mu_I}({\textbf{T}})$ as a disjoint union of three sets. 
$$\mathcal{B}_{\mu_I}({\textbf{T}}) = \tilde{\mathcal{B}}_{\mu_I}({\textbf{T}}) \sqcup \tilde{\mathcal{C}}_{\mu_I}({\textbf{T}}) \sqcup \varrho({\textbf{T}})$$
To find a lower bound on the cardinality of $\mathcal{B}_{\mu_I}({\textbf{T}})$ we need to find a lower bound on the cardinality of $\tilde{\mathcal{B}}_{\mu_I}({\textbf{T}})$, since we have that $\abs{\tilde{\mathcal{C}}_{\mu_I}({\textbf{T}})}=k-2$ and $\abs{\varrho({\textbf{T}})}=k+t$. We achieve this by first constructing a map $\mathcal{M} \to \tilde{\mathcal{B}}_{\mu_I}({\textbf{T}})$ in Lemma \ref{02}, and then by modifying it in Lemma \ref{03} to make it injective.  

\begin{lemma}\label{02}
For every $b_i \in \mathcal{M}$ there exists an arc $\beta_i \in \tilde{\mathcal{B}}_{\mu_I}(\textup{\textbf{T}})$. 
\end{lemma}

\begin{proof}
For every $b_{i-1}\in\mathcal{M}$ there exists $a_i\in\mathcal{K}$ such that $a_i$ intersects $\varrho(b_{i-1})$.  Based on the properties of $f(a_i)$ we proceed to define the corresponding $\beta_i$.  The flowchart in Figure~\ref{flowchart} illustrates various cases explained in detail below.   

First we observe that $f(a_i)$ is either a notched arc attached to the puncture or it is not attached to the puncture.   Let $r$ be the unique integer such that $f(a_i)\in\mathcal{B}^r_{\mu_I}({\textbf{T}})\setminus \mathcal{B}^{r-1}_{\mu_I}({\textbf{T}})$.  

Suppose $f(a_i)$ is not attached to the puncture.  Then there exists a triangle $\Delta_i$ consisting of $f(a_i), \delta_i$ and $\gamma_i$ as shown in Figure~\ref{fig:betas}(a).  If neither $\delta_i$ nor $\gamma_i$ are boundary segments, then $\Delta_i\subset \mathcal{B}^{r}_{\mu_I}({\textbf{T}})$, and we define ${\beta}_i=\phi(\Delta_i)$.  In other words $\beta_i=f(\sigma)$, where $\sigma$ is the next arc of $\Delta_i$ that is flipped along $\mu_{i_q}\dots \mu_{i_{r+1}}$.  Observe that $\beta_i$ is indeed an element of $\tilde{\mathcal{B}}_{\mu_I}({\textbf{T}})$, because by construction $\beta_i \not\in \tilde{\mathcal{C}}_{\mu_I}({\textbf{T}})$, and since neither $\delta_i$ nor $\gamma_i$ are boundary segments, it is easy to see that $\beta_i\not\in\varrho({\textbf{T}})$.    

Now we consider what happens if $\delta_i$ or $\gamma_i$ are boundary segments.  Since $f(a_i)$ is not attached to the puncture, there exists another triangle in $\mathcal{B}^r_{\mu_I}({\textbf{T}})$ consisting of arcs $c_i, d_i$ and $f(a_i)$ as shown in Figure~\ref{fig:betas}(a).  Suppose $\delta_i$ is a boundary segment then $c_i$ does not belong to the original triangulation ${\textbf{T}}$.  Indeed, by assumption $b_{i-1}\in\mathcal{M}$ which implies that $b_i=\varrho^2(b_{i-1})\in\mathcal{T}$, and in this case $c_i$ and $b_i$ cross.  Moreover, since $c_i$ is tagged plain it does not belong to $\varrho({\textbf{T}})$ and since it is attached to the puncture it does not belong to $\tilde{\mathcal{C}}_{\mu_I}({\textbf{T}})$.  Hence, $c_i\in\tilde{\mathcal{B}}_{\mu_I}({\textbf{T}})$ and we define $\beta_i = c_i$.   On the other hand, if $\delta_i$ is not a boundary segment but $\gamma_i$ is a boundary segment then $d_i$ does not belong to the original triangulation ${\textbf{T}}$, because it crosses $b_{i-1}$.  So by similar reasoning as above $d_i\not \in \varrho({\textbf{T}})\cup \tilde{\mathcal{C}}_{\mu_I}({\textbf{T}})$, and we define $\beta_i = d_i$. 

Suppose $f(a_i)$ is a notched arc attached to the puncture.  If $f(a_i)\not\in\varrho({\textbf{T}})$, then let $\beta_i=f(a_i)$.  Again, in this case $f(a_i)\not\in\tilde{\mathcal{C}}_{\mu_I}({\textbf{T}})$, because $\tilde{\mathcal{C}}_{\mu_I}({\textbf{T}})\subset\mathcal{C}_{\mu_I}({\textbf{T}})$ does not contain arcs attached to the puncture, so $f(a_i)\in \tilde{\mathcal{B}}_{\mu_I}({\textbf{T}})$.  

If $f(a_i)\in\varrho({\textbf{T}})$, then in $\mathcal{B}^r_{\mu_I}({\textbf{T}})$  there are either two notched arcs attached to the puncture, or there is one notched arc and one tagged plain arc attached to the puncture, see Figure~\ref{fig:betas} parts (b) and (c) respectively.  In the first case, since $\epsilon_i$ intersects $\varrho(b_{i-1})$ it follows that $\epsilon_i \not \in \varrho({\textbf{T}})$.  Also, since $\epsilon_i$ is attached to the puncture it does not belong to $\tilde{\mathcal{C}}_{\mu_I}({\textbf{T}})$, so it must belong to $\tilde{\mathcal{B}}_{\mu_I}({\textbf{T}})$.  Therefore, we let $\beta_i = \epsilon_i$.  In the second case, we let $\beta_i=\phi(\{\delta_i, \gamma_i, \epsilon_i\})$.  If $\beta_i=f(\epsilon_i)$ then we are in the situation of case one, see Figure~\ref{fig:betas}(b) and replace $\epsilon_i$ by $f(\epsilon_i)$.  Hence, by the same reasoning as in case one $f(\epsilon_i)\in\tilde{\mathcal{B}}_{\mu_I}({\textbf{T}})$.  If $\beta_i$ equals $f(\gamma_i)$ or $f(\delta_i)$, then these arcs are not attached to the puncture, so they cannot belong to $\mathcal{C}_{\mu_I}({\textbf{T}})$.  It is also easy to see that they cannot lie in $\varrho({\textbf{T}})$ as they enclose the puncture.  Hence, in this situation $\beta_i \in \tilde{\mathcal{B}}_{\mu_I}({\textbf{T}})$. 

Thus, in all scenarios we provided a definition of $\beta_i$ and showed that it is well defined, which finishes the proof of the lemma. 
\end{proof}

\begin{figure}[htb]

 (a) \hspace{5cm} (b) \hspace{5cm}(c) 
 \vspace{.3cm}

  \centering
  {\begingroup%
  \makeatletter%
  \def\svgwidth{500pt}
  \providecommand\color[2][]{%
    \errmessage{(Inkscape) Color is used for the text in Inkscape, but the package 'color.sty' is not loaded}%
    \renewcommand\color[2][]{}%
  }%
  \providecommand\transparent[1]{%
    \errmessage{(Inkscape) Transparency is used (non-zero) for the text in Inkscape, but the package 'transparent.sty' is not loaded}%
    \renewcommand\transparent[1]{}%
  }%
  \providecommand\rotatebox[2]{#2}%
  \ifx\svgwidth\undefined%
    \setlength{\unitlength}{555.8818378bp}%
    \ifx\svgscale\undefined%
      \relax%
    \else%
      \setlength{\unitlength}{\unitlength * \real{\svgscale}}%
    \fi%
  \else%
    \setlength{\unitlength}{\svgwidth}%
  \fi%
  \global\let\svgwidth\undefined%
  \global\let\svgscale\undefined%
  \makeatother%
  \begin{picture}(1,0.32904676)%
    \put(0,0){\includegraphics[width=\unitlength]{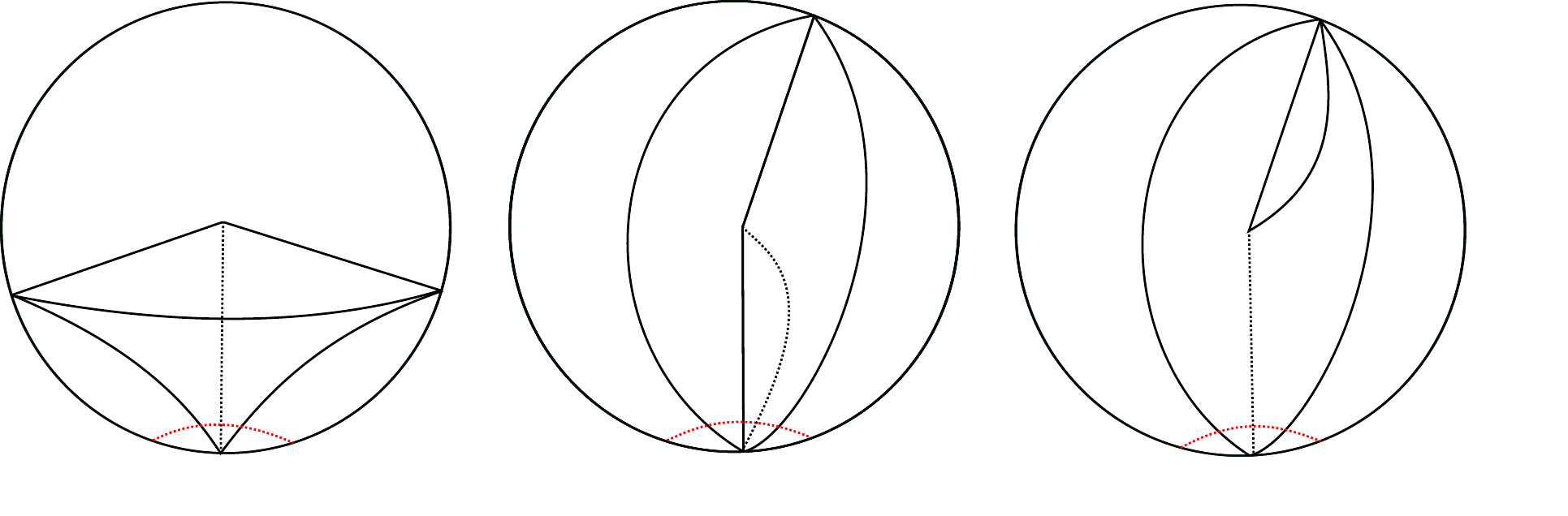}}%
    \put(0.06830382,0.18658133){\color[rgb]{0,0,0}\makebox(0,0)[lt]{\begin{minipage}{0.07504164\unitlength}\raggedright $c_i$\end{minipage}}}%
    \put(0.18549214,0.19994491){\color[rgb]{0,0,0}\makebox(0,0)[lt]{\begin{minipage}{0.07195774\unitlength}\raggedright $d_i$\end{minipage}}}%
    \put(0.17670016,0.15544122){\color[rgb]{0,0,0}\makebox(0,0)[lt]{\begin{minipage}{0.06784587\unitlength}\raggedright $f(a_i)$\end{minipage}}}%
    \put(0.19988367,0.09430748){\color[rgb]{0,0,0}\makebox(0,0)[lt]{\begin{minipage}{0.09354506\unitlength}\raggedright $\gamma_i$\end{minipage}}}%
    \put(0.05082837,0.1040512){\color[rgb]{0,0,0}\makebox(0,0)[lt]{\begin{minipage}{0.09354506\unitlength}\raggedright $\delta_i$\end{minipage}}}%
    \put(0.10736659,0.03444212){\color[rgb]{0,0,0}\makebox(0,0)[lt]{\begin{minipage}{0.23848852\unitlength}\raggedright ${\color{red} \varrho(b_{i-1})}$\\ \end{minipage}}}%
    \put(0.37669412,0.21536442){\color[rgb]{0,0,0}\makebox(0,0)[lt]{\begin{minipage}{0.08120943\unitlength}\raggedright $\delta_i$\end{minipage}}}%
    \put(0.70721816,0.17863252){\color[rgb]{0,0,0}\makebox(0,0)[lt]{\begin{minipage}{0.08120943\unitlength}\raggedright $\delta_i$\end{minipage}}}%
    \put(0.52456241,0.21271052){\color[rgb]{0,0,0}\makebox(0,0)[lt]{\begin{minipage}{0.09354506\unitlength}\raggedright $\gamma_i$\end{minipage}}}%
    \put(0.83157203,0.13355701){\color[rgb]{0,0,0}\makebox(0,0)[lt]{\begin{minipage}{0.09354506\unitlength}\raggedright $\gamma_i$\end{minipage}}}%
    \put(0.44600975,0.27733027){\color[rgb]{0,0,0}\makebox(0,0)[lt]{\begin{minipage}{0.06651099\unitlength}\raggedright $f(a_i)$\end{minipage}}}%
    \put(0.77082366,0.28041418){\color[rgb]{0,0,0}\makebox(0,0)[lt]{\begin{minipage}{0.06784587\unitlength}\raggedright $f(a_i)$\end{minipage}}}%
    \put(0.44925402,0.15784102){\color[rgb]{0,0,0}\makebox(0,0)[lt]{\begin{minipage}{0.09765694\unitlength}\raggedright $\epsilon_i$\end{minipage}}}%
    \put(0.82703934,0.20086268){\color[rgb]{0,0,0}\makebox(0,0)[lt]{\begin{minipage}{0.09765694\unitlength}\raggedright $\epsilon_i$\end{minipage}}}%
    \put(0.4411733,0.03678583){\color[rgb]{0,0,0}\makebox(0,0)[lt]{\begin{minipage}{0.23848852\unitlength}\raggedright ${\color{red} \varrho(b_{i-1})}$\\ \end{minipage}}}%
    \put(0.76189922,0.0347299){\color[rgb]{0,0,0}\makebox(0,0)[lt]{\begin{minipage}{0.23848852\unitlength}\raggedright ${\color{red} \varrho(b_{i-1})}$\\ \end{minipage}}}%
    \put(0.12018701,0.15961917){\color[rgb]{0,0,0}\makebox(0,0)[lt]{\begin{minipage}{0.05015493\unitlength}\raggedright $a_i$\end{minipage}}}%
    \put(0.49162416,0.17851812){\color[rgb]{0,0,0}\makebox(0,0)[lt]{\begin{minipage}{0.07196143\unitlength}\raggedright $a_i$\end{minipage}}}%
    \put(0.77510856,0.09710722){\color[rgb]{0,0,0}\makebox(0,0)[lt]{\begin{minipage}{0.05015495\unitlength}\raggedright $a_i$\end{minipage}}}%
    \put(0.47157908,0.20328049){\color[rgb]{0,0,0}\rotatebox{-22.39495158}{\makebox(0,0)[lt]{\begin{minipage}{0.05506745\unitlength}\raggedright {\Small $\bowtie$}\end{minipage}}}}%
    \put(0.46575191,0.1710291){\color[rgb]{0,0,0}\makebox(0,0)[lt]{\begin{minipage}{0.05052259\unitlength}\raggedright {\Small $\bowtie$}\end{minipage}}}%
    \put(0.79787746,0.20942883){\color[rgb]{0,0,0}\rotatebox{-22.39495158}{\makebox(0,0)[lt]{\begin{minipage}{0.05506745\unitlength}\raggedright {\Small $\bowtie$}\end{minipage}}}}%
  \end{picture}%
\endgroup%
}
  \caption{Each of these figures shows the quadrilateral containing $f(a_i)$ in $\mathcal{B}^r_{\mu_I}({\textbf{T}})$.}
   \label{fig:betas}
\end{figure}
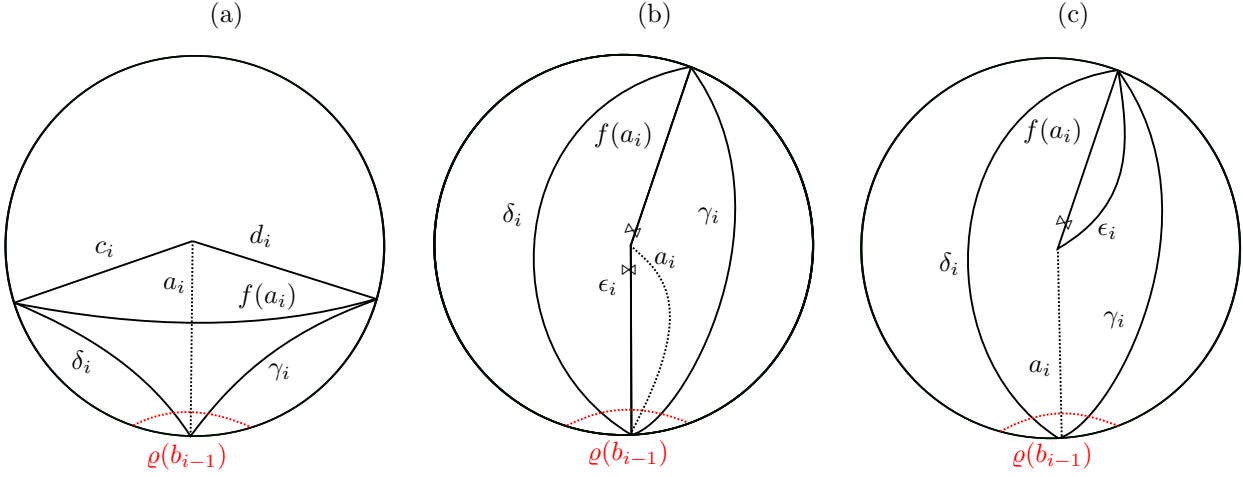

\begin{remark}\label{rm 03}
Observe that no $\beta_i$ as defined in the lemma above belongs to $\mathcal{H}$.  Recall that $\mathcal{H}$ consists of arcs that are not attached to the puncture, but come from flips of plain tagged arcs attached to the puncture.  Moreover, the arcs of $\mathcal{H}$ appear in the surface prior to the appearance of any notched arcs in the surface. 
\end{remark}

\begin{example}\label{ex01}
Consider a triangulation ${\textbf{T}}$ depicted in Figure \ref{ex1} on the left.  By definition  

\[\mathcal{K}=\{a_1, a_2, a_3, a_4, a_5\} \hspace{1cm} \mathcal{T}=\{b_1, b_2, b_3\} \hspace{1cm} \mathcal{M} = \{b_1, b_2\}\]
The corresponding quiver $Q_{\textbf{T}}$ admits a maximal green sequence 
\[\mu_I=\mu_{b_2}\mu_{a_5}\mu_{a_3}\mu_{a_2}\mu_{b_1}\mu_{b_2}\mu_{b_3}\mu_{a_1}\mu_{a_3}\mu_{a_2}\mu_{a_5}\mu_{a_4}\mu_{b_3}\]
and we see that 
\[\mathcal{B}_{\mu_I}({\textbf{T}}) = \{f(b_3), f(a_4), f(a_5), f(a_2), f(a_3), f(a_1), f^2(b_3), f(b_2), f(b_1), f^2(a_2), f^2(a_3), f^2(a_5), f^2(b_2)\}.\]
Moreover, 
\[\varrho({\textbf{T}}) = \{ f(a_1), f^2(a_2), f^2(a_3), f(a_4), f^2(a_5), f(b_1), f^2(b_2), f^2(b_3) \}\]
and it is depicted in Figure \ref{ex1} on the right.  We begin by computing $\tilde{\mathcal{C}}_{\mu_I}({\textbf{T}})$.  The triangulation $\mathcal{B}_{\mu_I}^7({\textbf{T}})$ is shown in Figure \ref{ex1} in the middle, and we see that $\{\alpha, \alpha'\} = \{a_1, b_3\}$.  Hence, $\mathcal{K'}=\mathcal{K}\setminus\{a_1\}$.   Observe that $f(a_4)=\varrho(b_3)$ and $b_3 = \alpha'$, which means $H(a_4)$ is not defined.  However, the function $H$ is defined for the rest of the arcs in $\mathcal{K'}$, and 
\[ \tilde{\mathcal{C}}_{\mu_I}({\textbf{T}}) = \mathcal{H} = \{f(a_2), f(a_3), f(a_5)\}.\]
Let us compute $\beta_1$ and $\beta_2$.  To do so we consider the corresponding $f(a_2)$ and $f(a_3)$.  Observe that $\Delta_2 = \{f(a_2), b_1, b_2\}$ and $b_2$ will be flipped first among the rest of the arcs of $\Delta_2$ along $\mu_I$.  This implies that $\beta_1=\phi(\Delta_2)=f(b_2)$.  On the other hand $\Delta_3$ consists of only two arcs, as $\delta_3$ is a boundary segment.  Therefore, by definition $\beta_2 = f(b_3)$, and in this example $\tilde{\mathcal{B}}_{\mu_I}({\textbf{T}}) = \{f(b_2), f(b_3)\}$.  
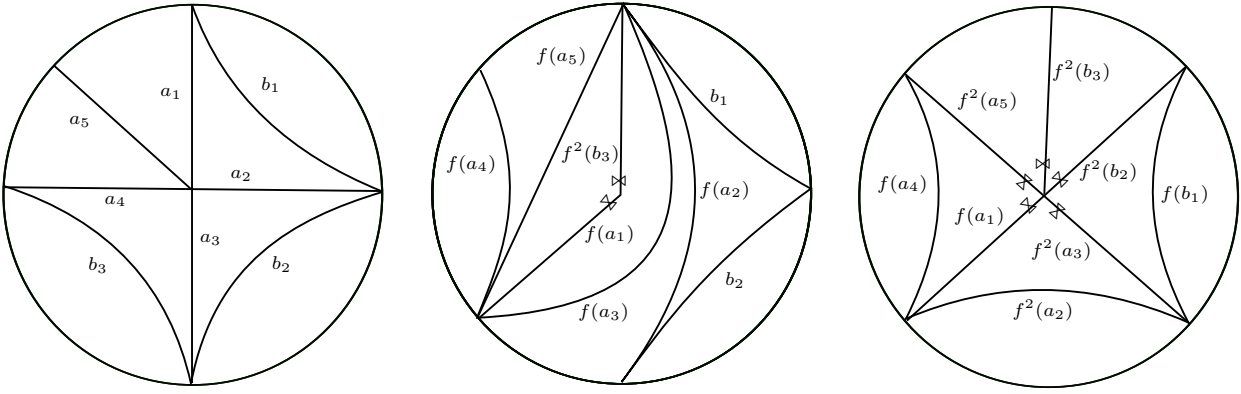
\begin{figure}[htb]
  \centering
  {\begingroup%
  \makeatletter%
  \def\svgwidth{500pt}
  \providecommand\color[2][]{%
    \errmessage{(Inkscape) Color is used for the text in Inkscape, but the package 'color.sty' is not loaded}%
    \renewcommand\color[2][]{}%
  }%
  \providecommand\transparent[1]{%
    \errmessage{(Inkscape) Transparency is used (non-zero) for the text in Inkscape, but the package 'transparent.sty' is not loaded}%
    \renewcommand\transparent[1]{}%
  }%
  \providecommand\rotatebox[2]{#2}%
  \ifx\svgwidth\undefined%
    \setlength{\unitlength}{555.8818378bp}%
    \ifx\svgscale\undefined%
      \relax%
    \else%
      \setlength{\unitlength}{\unitlength * \real{\svgscale}}%
    \fi%
  \else%
    \setlength{\unitlength}{\svgwidth}%
  \fi%
  \global\let\svgwidth\undefined%
  \global\let\svgscale\undefined%
  \makeatother%
  \begin{picture}(1,0.29260748)%
    \put(0,0){\includegraphics[width=\unitlength]{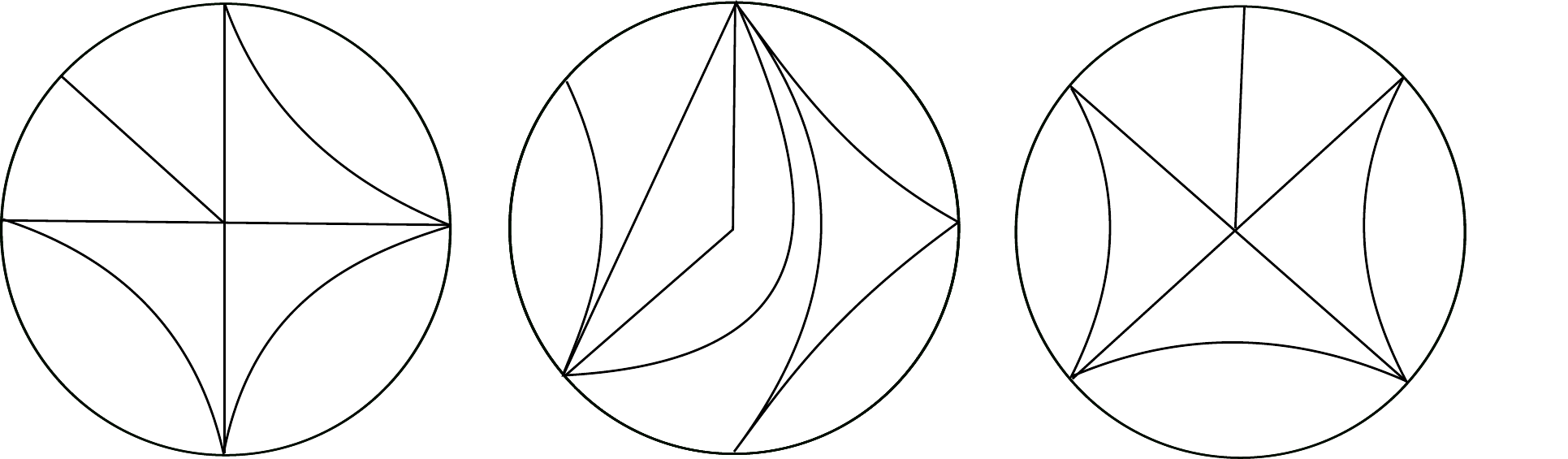}}%
    \put(0.45990041,0.16242355){\color[rgb]{0,0,0}\makebox(0,0)[lt]{\begin{minipage}{0.07504241\unitlength}\raggedright {\Small $\bowtie$}\end{minipage}}}%
    \put(0.11946012,0.22692057){\color[rgb]{0,0,0}\makebox(0,0)[lt]{\begin{minipage}{0.08649907\unitlength}\raggedright $\scriptstyle{a_1}$\end{minipage}}}%
    \put(0.17252258,0.16513552){\color[rgb]{0,0,0}\makebox(0,0)[lt]{\begin{minipage}{0.11339374\unitlength}\raggedright $\scriptstyle{a_2}$\end{minipage}}}%
    \put(0.14926232,0.11643436){\color[rgb]{0,0,0}\makebox(0,0)[lt]{\begin{minipage}{0.09449478\unitlength}\raggedright $\scriptstyle{a_3}$\end{minipage}}}%
    \put(0.07730091,0.14550968){\color[rgb]{0,0,0}\makebox(0,0)[lt]{\begin{minipage}{0.09522168\unitlength}\raggedright $\scriptstyle{a_4}$\end{minipage}}}%
    \put(0.05040624,0.20584097){\color[rgb]{0,0,0}\makebox(0,0)[lt]{\begin{minipage}{0.08431841\unitlength}\raggedright $\scriptstyle{a_5}$\end{minipage}}}%
    \put(0.19578282,0.23564316){\color[rgb]{0,0,0}\makebox(0,0)[lt]{\begin{minipage}{0.10321739\unitlength}\raggedright $\scriptstyle{b_1}$\end{minipage}}}%
    \put(0.20377855,0.09971606){\color[rgb]{0,0,0}\makebox(0,0)[lt]{\begin{minipage}{0.12647763\unitlength}\raggedright $\scriptstyle{b_2}$\end{minipage}}}%
    \put(0.0649439,0.09898917){\color[rgb]{0,0,0}\makebox(0,0)[lt]{\begin{minipage}{0.09304101\unitlength}\raggedright $\scriptstyle{b_3}$\end{minipage}}}%
    \put(0.33534435,0.17676564){\color[rgb]{0,0,0}\makebox(0,0)[lt]{\begin{minipage}{0.08504532\unitlength}\raggedright $\scriptstyle{f(a_4)}$\end{minipage}}}%
    \put(0.40367134,0.25817654){\color[rgb]{0,0,0}\makebox(0,0)[lt]{\begin{minipage}{0.08795283\unitlength}\raggedright $\scriptstyle{f(a_5)}$\end{minipage}}}%
    \put(0.4400155,0.12370319){\color[rgb]{0,0,0}\makebox(0,0)[lt]{\begin{minipage}{0.09522164\unitlength}\raggedright $\scriptstyle{f(a_1)}$\end{minipage}}}%
    \put(0.43565419,0.0640988){\color[rgb]{0,0,0}\makebox(0,0)[lt]{\begin{minipage}{0.09449478\unitlength}\raggedright $\scriptstyle{f(a_3)}$\end{minipage}}}%
    \put(0.52724144,0.15713981){\color[rgb]{0,0,0}\makebox(0,0)[lt]{\begin{minipage}{0.08504527\unitlength}\raggedright $\scriptstyle{f(a_2)}$\end{minipage}}}%
    \put(0.53523717,0.22619368){\color[rgb]{0,0,0}\makebox(0,0)[lt]{\begin{minipage}{0.09522164\unitlength}\raggedright $\scriptstyle{b_1}$\end{minipage}}}%
    \put(0.54686727,0.08735905){\color[rgb]{0,0,0}\makebox(0,0)[lt]{\begin{minipage}{0.0981292\unitlength}\raggedright $\scriptstyle{b_2}$\end{minipage}}}%
    \put(0.42257029,0.18621513){\color[rgb]{0,0,0}\makebox(0,0)[lt]{\begin{minipage}{0.10685184\unitlength}\raggedright $\scriptstyle{f^2(b_3)}$\end{minipage}}}%
    \put(0.79400749,0.24800017){\color[rgb]{0,0,0}\makebox(0,0)[lt]{\begin{minipage}{0.11048622\unitlength}\raggedright $\scriptstyle{f^2(b_3)}$\end{minipage}}}%
    \put(0.87614522,0.1535054){\color[rgb]{0,0,0}\makebox(0,0)[lt]{\begin{minipage}{0.1003099\unitlength}\raggedright $\scriptstyle{f(b_1)}$\end{minipage}}}%
    \put(0.76420526,0.06700632){\color[rgb]{0,0,0}\makebox(0,0)[lt]{\begin{minipage}{0.10903249\unitlength}\raggedright $\scriptstyle{f^2(a_2)}$\end{minipage}}}%
    \put(0.71986544,0.13678708){\color[rgb]{0,0,0}\makebox(0,0)[lt]{\begin{minipage}{0.11048622\unitlength}\raggedright $\scriptstyle{f(a_1)}$\end{minipage}}}%
    \put(0.77801606,0.11279995){\color[rgb]{0,0,0}\makebox(0,0)[lt]{\begin{minipage}{0.10321737\unitlength}\raggedright $\scriptstyle{f^2(a_3)}$\end{minipage}}}%
    \put(0.81436022,0.17240435){\color[rgb]{0,0,0}\makebox(0,0)[lt]{\begin{minipage}{0.1068518\unitlength}\raggedright $\scriptstyle{f^2(b_2)}$\end{minipage}}}%
    \put(0.72204608,0.22619368){\color[rgb]{0,0,0}\makebox(0,0)[lt]{\begin{minipage}{0.08867979\unitlength}\raggedright $\scriptstyle{f^2(a_5)}$\end{minipage}}}%
    \put(0.66171478,0.1615011){\color[rgb]{0,0,0}\makebox(0,0)[lt]{\begin{minipage}{0.08286471\unitlength}\raggedright $\scriptstyle{f(a_4)}$\end{minipage}}}%
    \put(0.45685657,0.14906771){\color[rgb]{0,0,0}\rotatebox{-38.89624967}{\makebox(0,0)[lt]{\begin{minipage}{0.08563479\unitlength}\raggedright {\Small $\bowtie$}\end{minipage}}}}%
    \put(0.78045342,0.17560279){\color[rgb]{0,0,0}\makebox(0,0)[lt]{\begin{minipage}{0.07504241\unitlength}\raggedright {\Small $\bowtie$}\end{minipage}}}%
    \put(0.77550303,0.14671107){\color[rgb]{0,0,0}\rotatebox{-45.20409643}{\makebox(0,0)[lt]{\begin{minipage}{0.08364755\unitlength}\raggedright {\Small $\bowtie$}\end{minipage}}}}%
    \put(0.79949902,0.16682881){\color[rgb]{0,0,0}\rotatebox{-44.8796434}{\makebox(0,0)[lt]{\begin{minipage}{0.08377474\unitlength}\raggedright {\Small $\bowtie$}\end{minipage}}}}%
    \put(0.76469128,0.15504405){\color[rgb]{0,0,0}\rotatebox{48.03769616}{\makebox(0,0)[lt]{\begin{minipage}{0.08242328\unitlength}\raggedright {\Small $\bowtie$}\end{minipage}}}}%
    \put(0.7896024,0.13323403){\color[rgb]{0,0,0}\rotatebox{48.03769616}{\makebox(0,0)[lt]{\begin{minipage}{0.08242328\unitlength}\raggedright {\Small $\bowtie$}\end{minipage}}}}%
  \end{picture}%
\endgroup%
}
  \caption{Triangulations $\textbf{T}$, $\mathcal{B}_{\mu_I}^7({\textbf{T}})$, and $\mathcal{B}_{\mu_I}^{13} ({\textbf{T}})$ for a maximal green sequence $\mu_I$ given in Example~\ref{ex01}.}
   \label{ex1}
\end{figure}

\end{example}

\begin{lemma}\label{03}
For every $b_i \in \mathcal{M}$ there exists a unique arc $\tilde{\beta}_i \in \tilde{\mathcal{B}}_{\mu_I}(\textup{\textbf{T}})$. 
\end{lemma}

\begin{proof}
The proof of Lemma \ref{02} provides a construction of $\beta_i$ for every $b_{i-1}\in\mathcal{M}$.  Recall that the definition of $\beta_i$ can also be seen in the flowchart in Figure~\ref{flowchart}.  However, it can happen that $\beta_i=\beta_j$ for some pair of distinct arcs $b_{i-1}$ and $b_{j-1}$.  There are three possible cases, which we discuss in detail below.

\medskip
\underline{Case i.} If $\beta_j = d_j = c_i =\beta_i$, where $\beta_i$ arrises in Case 2 and $\beta_j$ arrises in Case 3 in the flowchart in Figure~\ref{flowchart}. Then from the flowchart, we can see that $\gamma_j, \delta_i$ are boundary segments while $\delta_j$ is not a boundary segment.  Moreover, $d_j, \gamma_j, \delta_i$ share the same endpoint, which means $j=i+1$,  see Figure~\ref{fig:cases}(a).  Let $\tilde{\beta}_i = d_j$, and it remains to define $\tilde{\beta}_j$.  

First, we claim that if $f(d_j)$ is a notched arc attached to the puncture then $\mathcal{H}\setminus\tilde{\mathcal{C}}_{\mu_I}({\textbf{T}})$ is nonempty.  Let $\alpha, \alpha'$ be defined as in the proof of Lemma \ref{01}, that is there exists a triangulation $\mathcal{B}^s_{\mu_I}({\textbf{T}})$ containing plain tagged arcs $\alpha, \alpha'$ that are attached to the puncture such that $f(\alpha)$ is the first notched arc to appear along $\mu_I$.   First we show that $d_j\in\{\alpha, \alpha'\}$.  

By Figure~\ref{fig:cases}(a) there exists some $r'$ such that $\{d_j, a_j, a_i\}\subset \mathcal{B}^{r'}_{\mu_I}({\textbf{T}})$, because $a_i, a_j \in {\textbf{T}}$ and by construction neither of them can be flipped prior to the appearance of $d_j$ in the surface.  Since $a_i, a_j \in {\textbf{T}}$ these arcs must also belong to every triangulation $\mathcal{B}^{g}_{\mu_I}({\textbf{T}})$ for $g\leq r'$ by Lemma~\ref{arcs}.  This shows, that there cannot be notched arcs in $\bigcup_{i=1}^{i=r'}\mathcal{B}^i_{\mu_I}({\textbf{T}})$.  Therefore, if $f(d_j)$ equals $f(\alpha)$, which by definition is the first notched arc to appear along $\mu_I$, then $d_j=\alpha$.  On the other hand, if $d_j$ is flipped after $f(\alpha)$ appears in the surface then there exists a triangulation $\mathcal{B}^{r''}_{\mu_I}({\textbf{T}})$ where $r''>r'$ containing $f(\alpha)$ and $d_j$.  Therefore, $d_j=\alpha'$ because prior to a flip in $\alpha$, there are exactly two plain tagged arcs $d_j$, and $\alpha$.  This shows that $d_j \in \{\alpha, \alpha'\}$. 

Next, recall that $d_j\not\in {\textbf{T}}$, so by Remark~\ref{rm 02} the set $\mathcal{H}\setminus\tilde{\mathcal{C}}_{\mu_I}({\textbf{T}})$ is nonempty unless $d_j = \varrho^{-1}({\bar{a}_h})$ and $f(a_h)=\varrho(b_{h-1})$ for some $h$.  Suppose $d_j = \varrho^{-1}({\bar{a}_h})$ and $f(a_h)=\varrho(b_{h-1})$ for some $h$.  By construction we know that $d_j=\varrho^{-1}(\bar{a}_j)$, see Figure~\ref{fig:cases}(a). Hence we see that $h=j$ and $f(a_j)=\varrho(b_{j-1})$.  However, $f(a_j)=\varrho(b_{j-1})$ implies that $\delta_j$ is a boundary segment, but this contradicts the assumption we made in the beginning of Case i.  Therefore, we see that $\mathcal{H}\setminus\tilde{\mathcal{C}}_{\mu_I}({\textbf{T}})$ is nonempty and the claim holds.  

Recall that we let $\tilde{\beta}_i = d_j$, and we want to define $\tilde{\beta}_j$. If $f(d_j)$ is a notched arc attached to the puncture, by the claim above the set $\mathcal{H}\setminus\tilde{\mathcal{C}}_{\mu_I}({\textbf{T}})$ is nonempty, so we define $\tilde{\beta}_j$ as some arc in $\mathcal{H}\setminus\tilde{\mathcal{C}}_{\mu_I}({\textbf{T}})$. Otherwise,  we let $\tilde{\beta}_j=f(d_j)$ provided that $f(d_j)$.  To see that this definition makes sense, we need to justify that whenever $f(d_j)$ is not attached to the puncture, then $f(d_j)\in \tilde{\mathcal{B}}_{\mu_I}({\textbf{T}})$, or equivalently $f(d_j)\not\in\varrho({\textbf{T}})\cup\tilde{\mathcal{C}}_{\mu_I}({\textbf{T}})$.

Note that because $f(d_j)$ is not attached to the puncture and it does not intersect any arc in $\varrho(\mathcal{T})$ we conclude that $f(d_j)\not\in\varrho({\textbf{T}})$.  If $f(d_j)\in\tilde{\mathcal{C}}_{\mu_I}({\textbf{T}})$ then since $d_j\not\in {\textbf{T}}$, by definition of $H$, we must have that $f(a_j)=\varrho(b_{j-1})$.  As before, we conclude that the corresponding $\delta_j$ is a boundary segment, which is a contradiction to the assumption we made in the beginning of Case i.  Hence, $f(d_j)\not\in \tilde{\mathcal{C}}_{\mu_I}({\textbf{T}})$.   This shows that if $f(d_j)$ is not attached to the puncture, then it belongs to $\tilde{\mathcal{B}}_{\mu_I}({\textbf{T}})$.  

Now it remains to justify that $\tilde{\beta}_j$ defined in this way does not coincide with some other $\beta_l$ as defined in Lemma~\ref{02} and the corresponding flowchart in Figure~\ref{flowchart}.  Suppose $\tilde{\beta}_j=f(d_j)$, meaning $f(d_j)$ is not attached to the puncture.  Then $\tilde{\beta}_j$ cannot equal some $\beta_l$ arising in Cases 2-5, because $f(d_j)$ is not attached to the puncture.  Also, we note that whenever $f(\gamma)=f(\gamma')$ then $\gamma=\gamma'$, hence it is easy to see that if $\beta_l$ arises in Case 1 or Case 6 and is not attached to the puncture, then it does not equal a flip of a plain tagged arc.  This implies that $\tilde{\beta}_j=f(d_j)$ does not equal some other $\beta_l$ arising in any of the cases.  If on the other hand $\tilde{\beta}_j\in\mathcal{H}\setminus\tilde{\mathcal{C}}$ then by Remark~\ref{03} it follows that $\tilde{\beta}_j$ does not equal some $\beta_l$ for any $l$.  

This completes Case i, and let us sum up that here we define $\tilde{\beta_i} = \beta_i$ and $\tilde{\beta_j}=f(d_j)$ if $f(d_j)$ is not attached to the puncture, and otherwise $\tilde{\beta}_j$ is some arc in $\mathcal{H}\setminus\tilde{\mathcal{C}}_{\mu_I}({\textbf{T}})$.

\medskip

\underline{Case ii(a).}  Suppose $\beta_i=\phi(\Delta_i)=\phi(\Delta_j)=\beta_j$,  where both $\beta_i$ and $\beta_j$ arise in Case 1 in the flowchart in Figure~\ref{flowchart}.  Then from the flowchart we can see that $\Delta_i, \Delta_j$ are two triangles that do not contain boundary segments.  For example see Figure~\ref{fig:cases}(b).  To solve this case we consider a more general situation where there are a number of such triangles that share arcs with other triangles.  

Consider a triangle $\Delta_{i_1}$ corresponding to an arc $a_{i_1}\in\mathcal{K}$ such that no arc of $\Delta_{i_1}$ is a boundary segment.  By definition we have $\beta_{i_1}=\phi(\Delta_{i_1})$ and $\beta_{i_1}\in\tilde{\mathcal{B}}_{\mu_I}({\textbf{T}})$.  If we also have that $\beta_{i_1}=\phi(\Delta_{i_2})=\beta_{i_2}$ for some different triangle $\Delta_{i_2}$ where again no arc of $\Delta_{i_2}$ is a boundary segment, then let $\Delta^1=\Delta_{i_1}$ and 

\[\Delta^2=\Delta^1\cup \Delta_{i_2}.\] 

Since $\phi(\Delta^1)=\phi(\Delta_{i_2})$, there exists some $r_2$ such that the boundary $\partial \Delta^2$ of $\Delta^2$ lies in $\mathcal{B}^{r_2}_{\mu_I}({\textbf{T}})$, where 
$$\partial \Delta^2 = \Delta^2\setminus\{\text{common arcs of } \Delta_{i_1},\Delta_{i_2}\}.$$
Observe that by construction $\phi(\partial \Delta^2)\not\in\tilde{\mathcal{C}}_{\mu_I}({\textbf{T}})$, as none of the arcs in $\partial\Delta^2$ are attached to the puncture.  Also, from Figure~\ref{fig:cases}(b) it is easy to see that if $\phi(\partial \Delta^2)$ is a notched arc attached to the puncture, then the corresponding plain tagged arc $\overline{\phi(\partial \Delta^2)}$ is an arc in $\mathcal{K}$.  This implies that $\phi(\partial \Delta^2)\not\in\varrho(\mathcal{K})$.  On the other hand, if $\phi(\partial \Delta^2)$ is not attached to the puncture then it cannot belong to $\varrho(\mathcal{T})$, again see Figure~\ref{fig:cases}(b).  Hence, we conclude that $\phi(\partial\Delta^2)\not\in\tilde{\mathcal{C}}_{\mu_I}({\textbf{T}})\cup\varrho({\textbf{T}})$, so $\phi(\partial\Delta^2)\in\tilde{\mathcal{B}}_{\mu_I}({\textbf{T}})$.   

If we also have that 

\[\phi(\partial \Delta^2) = \beta_{i_3}=\phi(\Delta_{i_3})\] 

\noindent for some new triangle $\Delta_{i_3}$ where no arc of $\Delta_{i_3}$ is a boundary segment, then let 

\[ \Delta^3=\Delta^2\cup \Delta_{i_3}.\]

\noindent Because $\phi(\partial \Delta^2) = \phi(\partial \Delta_{i_3})$ there exists some $r_3$ such that $\partial \Delta^3\subset \mathcal{B}^{r_3}_{\mu_I}({\textbf{T}})$, where
 
\[  \partial \Delta^3 = \Delta^3\setminus\{\text{pairwise common arcs of } \Delta_{i_1},\Delta_{i_2}, \Delta_{i_3}\}. \]

\noindent By the same argument as before, we conclude that $\phi(\partial\Delta^3)\in\tilde{\mathcal{B}}_{\mu_I}({\textbf{T}})$.  Continue in this way to obtain a sequence of distinct triangles 

\[ \Delta_{i_1}, \Delta_{i_2}, \Delta_{i_3}, \dots, \Delta_{i_{q-1}}, \Delta_{i_q}\]

\noindent and the corresponding sequence of distinct arcs in $\tilde{\mathcal{B}}_{\mu_I}({\textbf{T}})$

\[\phi(\partial\Delta^1), \phi(\partial\Delta^2), \phi(\partial\Delta^3), \dots, \phi(\partial\Delta^{q-1}), \phi(\partial\Delta^q)\]

\noindent  where $\phi(\partial\Delta^q)$ does not coincide with $\beta_{i_{q+1}}$ for some new triangle $\Delta_{i_{q+1}}$.  For an example of such a construction when $q=4$ see Figure~\ref{delta}.  Now, we define 

\[\tilde{\beta}_{i_j}=\phi(\partial\Delta^{i_j}) \;\text{for all}\; j\in[q].  \]

Next we justify what $\tilde{\beta}_{i_j}$ defined in Case ii(a) do not coincide with $\tilde{\beta}$'s defined in Case i or with other $\beta$'s defined in the flowchart coming from Cases 1-5.  Recall that in Case ii(a) we already considered what happens if $\tilde{\beta}_{i_j}$ equals some other $\beta_i$ coming from Case 1.  Note that $\tilde{\beta}_{i_j}$ defined in this way cannot be a plain tagged arc attached to the puncture, because it would equal some arc in $\mathcal{K}$, which would contradict Lemma~\ref{arcs}.  Hence, $\tilde{\beta}_{i_j}$ cannot equal some $\beta_l$ where $\beta_l$ arrises in Cases 2, 3.  By construction $\tilde{\beta}_{i_j}$ cannot equal some $\beta_l$ that arrises in Cases 4.  It also cannot equal $\beta_l$ in Case 5, because here the corresponding $\beta_i=\epsilon_i$ is a flip of some plain tagged arc attached to the puncture.  Finally, $\tilde{\beta}_{i_j}$ does not equal other $\tilde{\beta}_l$ constructed in Case i, as these arcs are either plain tagged arcs attached to the puncture or flips of such arcs.  Therefore, it remains to investigate the case when $\tilde{\beta}_{i_j}$ coincides with some $\beta_i$  arising in Case 6.  We consider this situation next.

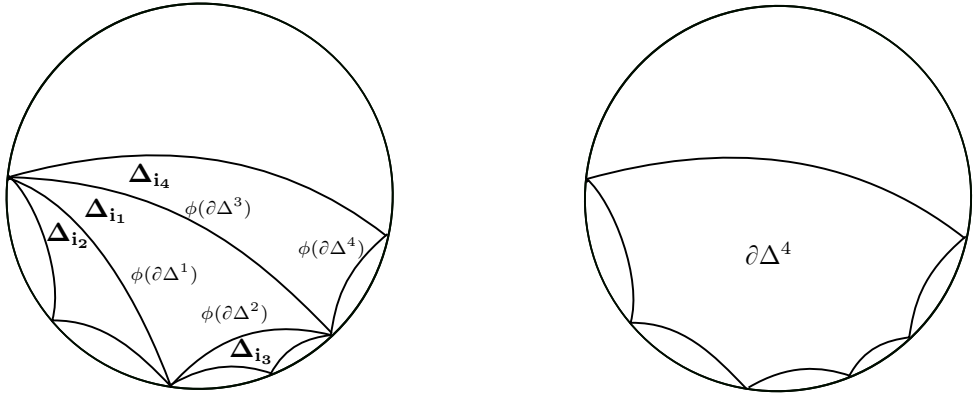
\begin{figure}[htb]

  \hspace*{2cm}
    {\begingroup%
  \makeatletter%
   \def\svgwidth{500pt}
  \providecommand\color[2][]{%
    \errmessage{(Inkscape) Color is used for the text in Inkscape, but the package 'color.sty' is not loaded}%
    \renewcommand\color[2][]{}%
  }%
  \providecommand\transparent[1]{%
    \errmessage{(Inkscape) Transparency is used (non-zero) for the text in Inkscape, but the package 'transparent.sty' is not loaded}%
    \renewcommand\transparent[1]{}%
  }%
  \providecommand\rotatebox[2]{#2}%
  \ifx\svgwidth\undefined%
    \setlength{\unitlength}{546.8796923bp}%
    \ifx\svgscale\undefined%
      \relax%
    \else%
      \setlength{\unitlength}{\unitlength * \real{\svgscale}}%
    \fi%
  \else%
    \setlength{\unitlength}{\svgwidth}%
  \fi%
  \global\let\svgwidth\undefined%
  \global\let\svgscale\undefined%
  \makeatother%
  \begin{picture}(1,0.33044918)%
    \put(0,0){\includegraphics[width=\unitlength]{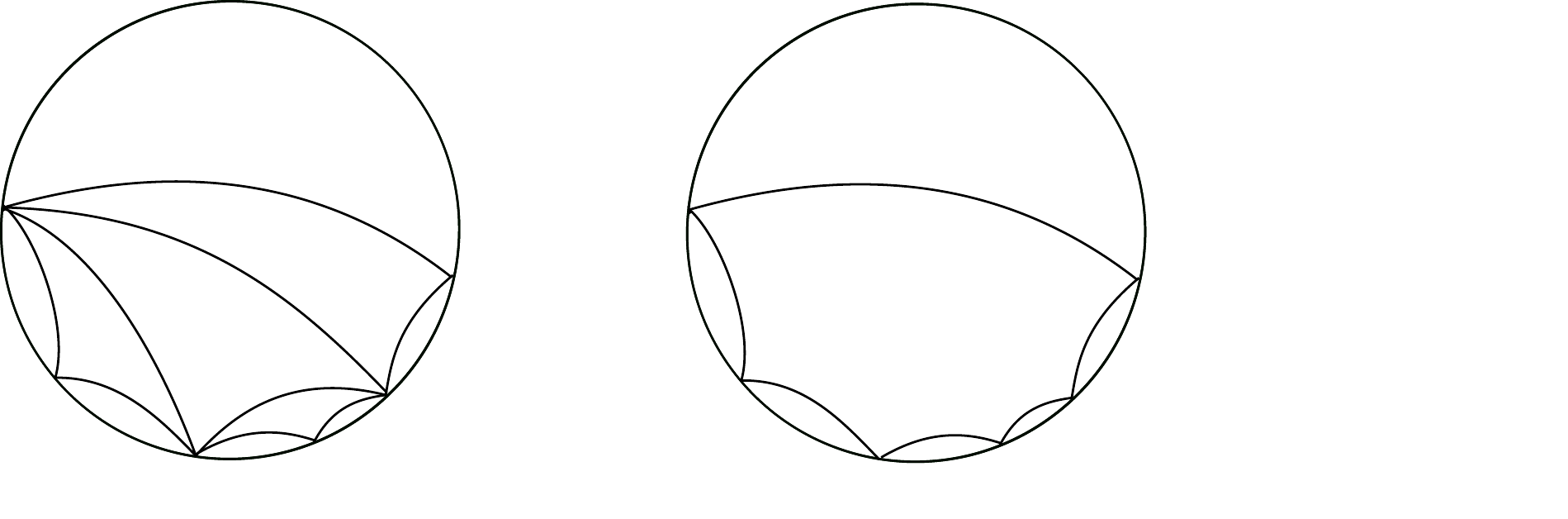}}%
    \put(0.55935718,0.14880389){\color[rgb]{0,0,0}\makebox(0,0)[lt]{\begin{minipage}{0.11156605\unitlength}\raggedright $\partial\Delta^4$\end{minipage}}}%
    \put(0.05915704,0.18205206){\color[rgb]{0,0,0}\makebox(0,0)[lt]{\begin{minipage}{0.14259768\unitlength}\raggedright ${\bf \Delta_{i_1}}$\end{minipage}}}%
    \put(0.0924052,0.20865059){\color[rgb]{0,0,0}\makebox(0,0)[lt]{\begin{minipage}{0.0953114\unitlength}\raggedright ${\bf \Delta_{i_4}}$\end{minipage}}}%
    \put(0.03034196,0.16358086){\color[rgb]{0,0,0}\makebox(0,0)[lt]{\begin{minipage}{0.0768402\unitlength}\raggedright ${\bf \Delta_{i_2}}$\end{minipage}}}%
    \put(0.16924539,0.07565793){\color[rgb]{0,0,0}\makebox(0,0)[lt]{\begin{minipage}{0.0938337\unitlength}\raggedright ${\bf \Delta_{i_3}}$\end{minipage}}}%
    \put(0.09462174,0.13254923){\color[rgb]{0,0,0}\makebox(0,0)[lt]{\begin{minipage}{0.13299265\unitlength}\raggedright $\scriptstyle{\phi(\partial\Delta^1)}$\end{minipage}}}%
    \put(0.14781881,0.10373417){\color[rgb]{0,0,0}\makebox(0,0)[lt]{\begin{minipage}{0.1625466\unitlength}\raggedright $\scriptstyle{\phi(\partial\Delta^2)}$\end{minipage}}}%
    \put(0.13525839,0.1827909){\color[rgb]{0,0,0}\makebox(0,0)[lt]{\begin{minipage}{0.25638028\unitlength}\raggedright $\scriptstyle{\phi(\partial\Delta^3)}$\end{minipage}}}%
    \put(0.22096476,0.15249813){\color[rgb]{0,0,0}\makebox(0,0)[lt]{\begin{minipage}{0.17584586\unitlength}\raggedright $\scriptstyle{\phi(\partial\Delta^4)}$\end{minipage}}}%
  \end{picture}%
\endgroup%
}
  \caption{Example of $\Delta^4$ on the left and the corresponding $\partial\Delta^4$ on the right.}
   \label{delta}
\end{figure}

\medskip
\underline{Case ii(b).}  If $\beta_i=\tilde{\beta}_{j_1}$, where $\beta_i$ arrises in Case 6, then 

$$\beta_i=\phi(\{\gamma_i, \delta_i, \epsilon_i\})=\phi(\partial \Delta^{j_1})$$

\noindent for some triangulated polygon $\Delta^{j_1}$.  Recall that by Case ii(a) the polygon is made up of triangles 

\[\Delta^{j_1} = \Delta_{i_1}\cup \Delta_{i_2}\cup \dots \cup \Delta_{i_j}\]

\noindent where no triangle contains a boundary segment.  Therefore, there exists a triangulation $\mathcal{B}^r_{\mu_I}({\textbf{T}})$ containing both $\partial\Delta^{j_1}$ and $\{\gamma_i, \delta_i, \epsilon_i\}$.  Observe that $\beta_i$ cannot equal $f(\epsilon_i)$, because in accordance with Case 6 the arc $\epsilon_i$ cannot lie on the boundary of some triangulated polygon, see Figure~\ref{fig:betas}(c).  Without loss of generality let $\beta_i=f(\delta_i)$, see Figure~\ref{fig:cases}(c).   The same argument can be made if $\beta_i=f(\gamma_i)$.  

Since we assume that $\delta_i\in\partial\Delta^{j_1}$, it follows from the structure of $\Delta^{j_1}$ that $\delta_i$ belongs to a triangle $\Delta_{l_1}$ for some $l_1\in \{i_1, \dots, i_j\}$.  Hence, $\delta_i \in \{f(a_{l_1}), \delta_{l_1}, \gamma_{l_1}\}$, but from the construction it is clear that $\delta_i=f(a_{l_1})$, see Figure~\ref{fig:cases}(c).  Also, in accordance with Case ii(a) we have $\tilde{\beta}_{j_1}=f(\delta_i)$, therefore it remains to define $\tilde{\beta}_i$.   

First, notice that by construction of $\Delta^{j_1}$ the set $\partial\Delta^{j_1}\setminus\{\delta_i\}$ consists of arcs of the form $\delta_{s_u}$ and $\gamma_{s_v}$ where the corresponding triangles $\Delta_{s_u}, \Delta_{s_v}\subset \Delta^{j_1}$.  Moreover, two arcs in $\partial\Delta^{j_1}\setminus\{\delta_i\}$ share an endpoint $P$ if and only if there exists some $a_{p}\in\mathcal{K}$ ending in the point $P$ and the corresponding $b_{p-1}\in\mathcal{T}$.  In particular we see that any arc having $P$ as an endpint does not belong to $\varrho({\textbf{T}})$ as it crosses $\varrho(b_{p-1})$.  

Next, we consider the set of arcs $\{f(\delta_i), \epsilon_i, \delta_{i_2}\}$ as shown in Figure~\ref{fig:cases}(c).  Observe that these arcs are in the analogous situation as the arcs in $\{\gamma_i, \delta_i, \epsilon_i\}$ were prior to a flip in $\delta_i$.  Next, we claim that $\phi(\{f(\delta_i), \epsilon_i, \delta_{i_2}\})\in\tilde{\mathcal{B}}_{\mu_I}({\textbf{T}})$.  Let $f(\sigma)=\phi(\{f(\delta_i), \epsilon_i, \delta_{i_2}\})$ and indeed it is easy to see that if $\sigma$ equals $\delta_{i_2}$ or $f(\delta_i)$ then $f(\sigma)\not\in\tilde{\mathcal{C}}_{\mu_I}({\textbf{T}})\cup\varrho({\textbf{T}})$.   If $\sigma=\epsilon_i$ then by definition of the set $\tilde{\mathcal{C}}_{\mu_I}({\textbf{T}})$ we have that $f(\epsilon_i)\not\in\tilde{\mathcal{C}}_{\mu_I}({\textbf{T}})$. But also $f(\epsilon_i)$ has the endpoint  $P$ in common with $\partial\Delta^{j_1}\setminus \{\delta_i\}$ located on the boundary of the disk.  By properties of $\partial\Delta^{j_1}$ discussed in the paragraph above we see that $f(\epsilon_i)\not\in\varrho({\textbf{T}})$.  This shows the claim that $\phi(\{f(\delta_i), \epsilon_i, \delta_{i_2}\})\in\tilde{\mathcal{B}}_{\mu_I}({\textbf{T}})$.  

In addition, we claim that $f(\sigma)$ for any $\sigma$ does not coincide with any other $\beta$ or $\tilde{\beta}$ coming from Case i, and Cases 2-6. We can see that $f(\sigma)$ cannot be a plain tagged arc attached to the puncture, so it cannot equal any $\beta_s$ arising in Case 2 or Case 3.  By construction $f(\sigma)$ cannot equal some $\beta_s=\epsilon_s$ arising in Case 5 because a notched arc $\epsilon_s$ would have to be present in the triangulation prior to a flip in $\delta_i$.  The same conclusion holds in Case 6, because we already mutated in $\delta_i$.  Observe that in Case 4, if $\beta_s=f(a_s)=f(\sigma)$, then we must have $\sigma=\epsilon_i=a_s$.  Recall that by assumption $f(a_i)\in\varrho({\textbf{T}})$, so in particular $f(a_i)$ does not cross any arc of $\varrho(\mathcal{T})$.  However, if $a_s\in\mathcal{K}$ such that $a_s$ crosses $\varrho(b_{s-1})$, then we would also have that $f(a_i)$ crosses $\varrho(b_{s-1})$, because $\epsilon_i=a_s$ and $f(a_i)$ share the same endpoint.  But this is a contradiction to $f(a_i)\in\varrho({\textbf{T}})$.  This shows that $f(\sigma)$ cannot equal some $\beta_s$ arising in Case 4.   Finally, $f(\sigma)$ does not equal some $\tilde{\beta}$ arising in Case i, because such arcs are never notched and come from flips of plain tagged arcs attached to the puncture.  

However, it can happen that there exists some $\tilde{\beta}_{j_2}$ such that $f(\sigma)=f(\delta_{i_2})=\tilde{\beta}_{j_2}=\phi(\partial\Delta^{j_2})$ for some triangulated polygon $\Delta^{j_2}$ as defined in Case ii(a).  If no such $\tilde{\beta}_{j_2}$ exists, then set $\tilde{\beta}_i=f(\sigma)$.  Otherwise, set $\tilde{\beta}_{j_2}=f(\delta_{i_2})$, but then again we need to define $\tilde{\beta}_i$.  

We can proceed in the same manner as before with the new set of arcs  $\{f(\delta_{i_2}), \delta_{i_2}, \epsilon_i\}$.  Continuing in this way, we obtain a sequence of sets 

\[\{\gamma_i, \delta_i, \epsilon_i\},\{f(\delta_i),\delta_{i_2}, \epsilon_i\},   \{f(\delta_{i_2}), \delta_{i_3}, \epsilon_i\}, \dots, \{f(\delta_{i_t}),\delta_{i_{t+1}}, \epsilon_i\}\]

\noindent and a sequence of triangulated polygons 

\[ \Delta^{j_1}, \Delta^{j_2}, \Delta^{j_3}, \dots, \Delta^{j_{t}}\]

\noindent such that $\delta_{i_u}\in\partial\Delta^{j_u}$ for all $u\in\{2, \dots, t-1\}$.  Also, we define $\tilde{\beta}_{j_u}=f(\delta_{i_u})$ for $u\in\{2, \dots, t-1\}$ and finally $\phi(\{f(\delta_{i_t}),\delta_{i_{t+1}}, \epsilon_i\}) \not= \tilde{\beta}_{j_t}$ for any $\tilde{\beta}_{j_t}=\phi(\partial\Delta^{j_t})$.  Indeed this holds, because the sizes of the triangulated polygons $\Delta^{j_u}$ must eventually decrease to zero as the endpoints of $\delta_{i_u}$ move closer together as $u$ increases.   Hence, we define $\tilde{\beta}_i = \phi(\{f(\delta_{i_t}),\delta_{i_{t+1}}, \epsilon_i\})$.  This completes Case ii(b).  Moreover, we managed to define $\tilde{\beta}$'s in Cases ii(a) and ii(b) such that $\tilde{\beta}_i=\tilde{\beta}_j$ if and only if $b_{i-1}=b_{j-1}$.  
 
 \medskip
We already justified that $\tilde{\beta}_i$'s as defined in the three cases above do not coincide with other $\beta_j$'s or $\tilde{\beta}$'s.  Therefore, it remains to show that the scenarios discussed in this lemma are the only ones when $\beta_i=\beta_j$ for some pair of distinct arcs $b_{i-1}, b_{j-1}$.  Recall that above we investigated when Case 1 overlapped with Cases 1 and 6, and Case 2 overlapped with Case 3.  In Case 1 the arc $\beta_i$ cannot be tagged plain and attached to the puncture, so it cannot overlap with Cases 2 and 3.  Also, by definition this $\beta_i$ cannot equal $f(a_j)$ for some $j$, hence Cases 1 and 4 cannot overlap.  By construction Case 1 and Case 5 cannot overlap, because $\epsilon_i$ as defined in Case 5 is a flip of some plain tagged arc attached to the puncture.  It is easy to see that Case 2 cannot overlap with Cases 4, 5, 6, as none of these yield a plain tagged arc attached to the puncture.  Also, Case 2 cannot overlap with another Case 2, because $\beta_i$ are defined uniquely by the corresponding $a_i$.  Similarly, Case 3 cannot overlap with Cases 3, 4, 5, 6.  

Suppose $\beta_i=f(a_i)$ arrises in Case 4, then it cannot equal another $\beta_j$ arising in Cases 4, because whenever $f(a_j)=f(a_i)$ then $a_j=a_i$.  Suppose that $\beta_i=f(a_i)=\epsilon_j=\beta_j$, where $\beta_i$ arrises in Case 4 and $\beta_j$ arrises in Case 5.  Then by definition, see the flowchart in Figure~\ref{flowchart}, there exists a triangulation depicted in Figure~\ref{case4}, such that $f(a_j)$ is a notched arc attached to the puncture and belonging to $\varrho({\textbf{T}})$.  However, $f(a_j)$ and $a_i$ share an endpoint on the boundary of the disk, and since $\varrho(b_{i-1})$ and $a_i$ cross, see the flowchart, it follows that $f(a_j)$ and $\varrho(b_{i-1})$ also cross.  This implies that $f(a_j)\not\in\varrho({\textbf{T}})$, which is a contradiction.  Hence, Case 4 cannot overlap with Case 5.  Also, Case 4 cannot overlap with Case 6, because the arcs $\delta_j, \gamma_j$ arising in Case 6 are not attached to the puncture.   Moreover, if $\beta_i=\beta_j=f(\epsilon_j)= \phi(\{\delta_j, \gamma_j, \epsilon_j\})$ where $\beta_j$ is defined in Case 6, then we obtain the same triangulation as in Case 5.  Proceeding in the same manner as above, we also obtain a contradiction.

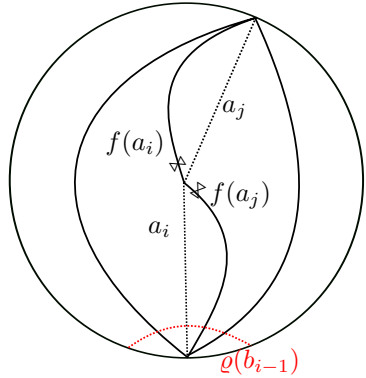
\begin{figure}[htb]
  \centering
  {\begingroup%
  \makeatletter%
  \def\svgwidth{150pt}
  \providecommand\color[2][]{%
    \errmessage{(Inkscape) Color is used for the text in Inkscape, but the package 'color.sty' is not loaded}%
    \renewcommand\color[2][]{}%
  }%
  \providecommand\transparent[1]{%
    \errmessage{(Inkscape) Transparency is used (non-zero) for the text in Inkscape, but the package 'transparent.sty' is not loaded}%
    \renewcommand\transparent[1]{}%
  }%
  \providecommand\rotatebox[2]{#2}%
  \ifx\svgwidth\undefined%
    \setlength{\unitlength}{178.7390378bp}%
    \ifx\svgscale\undefined%
      \relax%
    \else%
      \setlength{\unitlength}{\unitlength * \real{\svgscale}}%
    \fi%
  \else%
    \setlength{\unitlength}{\svgwidth}%
  \fi%
  \global\let\svgwidth\undefined%
  \global\let\svgscale\undefined%
  \makeatother%
  \begin{picture}(1,0.96296744)%
    \put(0,0){\includegraphics[width=\unitlength]{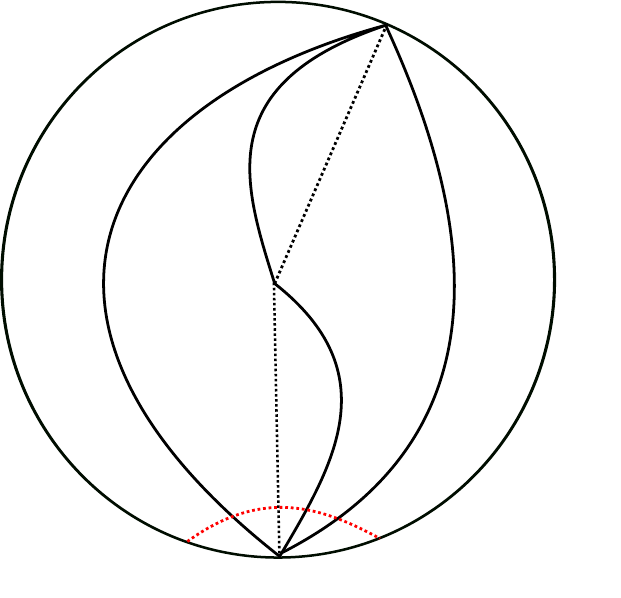}}%
    \put(0.35181572,0.4078165){\color[rgb]{0,0,0}\makebox(0,0)[lt]{\begin{minipage}{0.39387031\unitlength}\raggedright $a_i$\end{minipage}}}%
    \put(0.24439658,0.64017671){\color[rgb]{0,0,0}\makebox(0,0)[lt]{\begin{minipage}{0.23497945\unitlength}\raggedright $f(a_i)$\end{minipage}}}%
    \put(0.50655055,0.51087386){\color[rgb]{0,0,0}\makebox(0,0)[lt]{\begin{minipage}{0.29092693\unitlength}\raggedright $f(a_j)$\end{minipage}}}%
    \put(0.53852043,0.71388331){\color[rgb]{0,0,0}\makebox(0,0)[lt]{\begin{minipage}{0.25735844\unitlength}\raggedright $a_j$\end{minipage}}}%
    \put(0.530528,0.08886988){\color[rgb]{0,0,0}\makebox(0,0)[lt]{\begin{minipage}{0.3452759\unitlength}\raggedright ${\color{red}\varrho(b_{i-1})}$\end{minipage}}}%
    \put(0.39805121,0.55901256){\color[rgb]{0,0,0}\rotatebox{32.01696107}{\makebox(0,0)[lt]{\begin{minipage}{0.20867995\unitlength}\raggedright {\Small $\bowtie$} \end{minipage}}}}%
    \put(0.50494496,0.49622585){\color[rgb]{0,0,0}\rotatebox{-123.7174911}{\makebox(0,0)[lt]{\begin{minipage}{0.19122362\unitlength}\raggedright {\Small $\bowtie$} \end{minipage}}}}%
  \end{picture}%
\endgroup%
}
  \caption{Triangulation depicting $a_i$ in Case 4 and $a_j$ in Case 5.}
   \label{case4}
\end{figure}

Again, we can see that Case 5 can be though of as a special situation of Case 6 when $\beta_i=\phi(\{\delta_i, \gamma_i, \epsilon_i\})=f(\epsilon_i)$, but Case 5 cannot overlap with Case 5, because in this scenario there cannot be two notched arcs belonging to $\varrho({\textbf{T}})$.  Finally, by the same reasoning Case 6 cannot overlap with another Case 6.  Therefore, considering various scenarios we can see that Cases i, ii(a), ii(b) are the only possibilities when definitions of $\beta$'s can overlap.  In all other cases we let $\tilde{\beta}_i=\beta_i$, which completes the proof of the lemma. 

\begin{figure}[htb]
(a) \hspace{5cm} (b) \hspace{5cm} (c)
\vspace{.3cm}
  \centering
  {\begingroup%
  \makeatletter%
  \def\svgwidth{500pt}
  \providecommand\color[2][]{%
    \errmessage{(Inkscape) Color is used for the text in Inkscape, but the package 'color.sty' is not loaded}%
    \renewcommand\color[2][]{}%
  }%
  \providecommand\transparent[1]{%
    \errmessage{(Inkscape) Transparency is used (non-zero) for the text in Inkscape, but the package 'transparent.sty' is not loaded}%
    \renewcommand\transparent[1]{}%
  }%
  \providecommand\rotatebox[2]{#2}%
  \ifx\svgwidth\undefined%
    \setlength{\unitlength}{555.8818378bp}%
    \ifx\svgscale\undefined%
      \relax%
    \else%
      \setlength{\unitlength}{\unitlength * \real{\svgscale}}%
    \fi%
  \else%
    \setlength{\unitlength}{\svgwidth}%
  \fi%
  \global\let\svgwidth\undefined%
  \global\let\svgscale\undefined%
  \makeatother%
  \begin{picture}(1,0.32875897)%
    \put(0,0){\includegraphics[width=\unitlength]{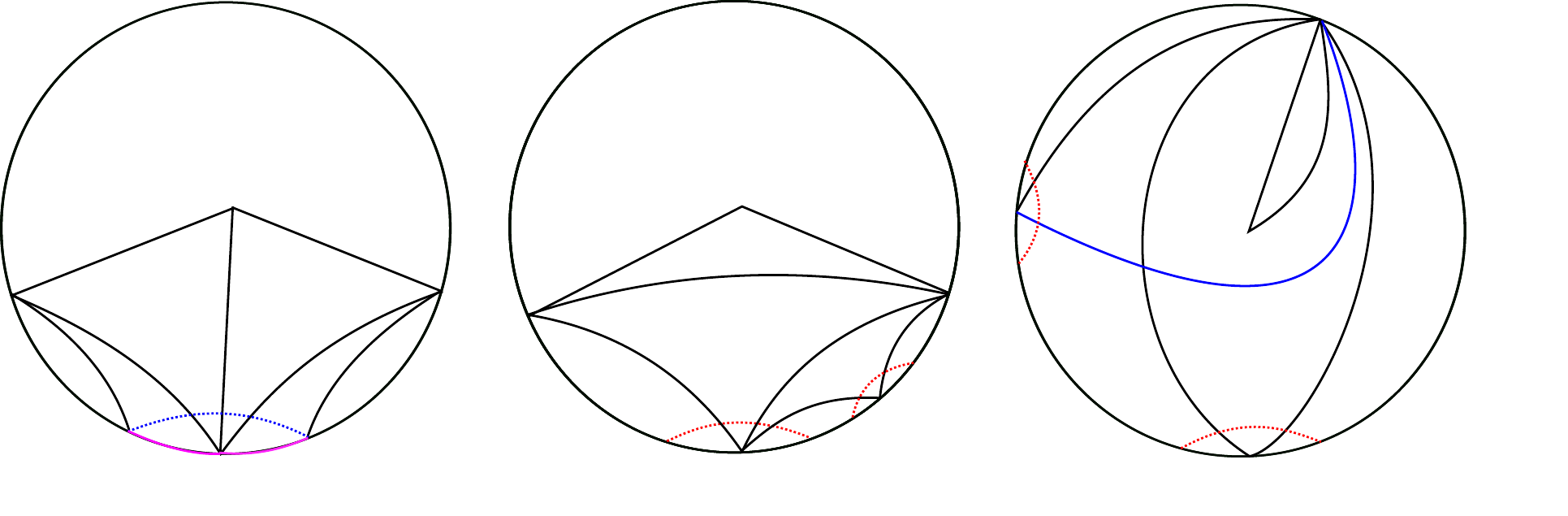}}%
    \put(0.1497147,0.16012576){\color[rgb]{0,0,0}\makebox(0,0)[lt]{\begin{minipage}{0.07504164\unitlength}\raggedright $c_i$\end{minipage}}}%
    \put(0.18549214,0.20692595){\color[rgb]{0,0,0}\makebox(0,0)[lt]{\begin{minipage}{0.07195774\unitlength}\raggedright $d_i$\end{minipage}}}%
    \put(0.18156398,0.14347165){\color[rgb]{0,0,0}\makebox(0,0)[lt]{\begin{minipage}{0.06784587\unitlength}\raggedright $f(a_i)$\end{minipage}}}%
    \put(0.22169016,0.09256593){\color[rgb]{0,0,0}\makebox(0,0)[lt]{\begin{minipage}{0.09354506\unitlength}\raggedright $\gamma_i$\end{minipage}}}%
    \put(0.38832425,0.09586784){\color[rgb]{0,0,0}\makebox(0,0)[lt]{\begin{minipage}{0.08120943\unitlength}\raggedright $\delta_i$\end{minipage}}}%
    \put(0.72757089,0.10129514){\color[rgb]{0,0,0}\makebox(0,0)[lt]{\begin{minipage}{0.08120943\unitlength}\raggedright $\delta_i$\end{minipage}}}%
    \put(0.82503008,0.10855521){\color[rgb]{0,0,0}\makebox(0,0)[lt]{\begin{minipage}{0.09354506\unitlength}\raggedright $\gamma_i$\end{minipage}}}%
    \put(0.45037105,0.18195721){\color[rgb]{0,0,0}\makebox(0,0)[lt]{\begin{minipage}{0.06651099\unitlength}\raggedright $f(a_i)$\end{minipage}}}%
    \put(0.77082366,0.2801264){\color[rgb]{0,0,0}\makebox(0,0)[lt]{\begin{minipage}{0.06784587\unitlength}\raggedright $f(a_i)$\end{minipage}}}%
    \put(0.45652285,0.11175962){\color[rgb]{0,0,0}\makebox(0,0)[lt]{\begin{minipage}{0.09765694\unitlength}\raggedright $f(a_j)$\end{minipage}}}%
    \put(0.82267803,0.19403295){\color[rgb]{0,0,0}\makebox(0,0)[lt]{\begin{minipage}{0.09765694\unitlength}\raggedright $\epsilon_i$\end{minipage}}}%
    \put(0.4411733,0.03649805){\color[rgb]{0,0,0}\makebox(0,0)[lt]{\begin{minipage}{0.23848852\unitlength}\raggedright ${\color{red} \varrho(b_{i-1})}$\\ \end{minipage}}}%
    \put(0.76189922,0.03444212){\color[rgb]{0,0,0}\makebox(0,0)[lt]{\begin{minipage}{0.23848852\unitlength}\raggedright ${\color{red} \varrho(b_{i-1})}$\\ \end{minipage}}}%
    \put(0.06668201,0.13573535){\color[rgb]{0,0,0}\makebox(0,0)[lt]{\begin{minipage}{0.06784587\unitlength}\raggedright $f(a_j)$\end{minipage}}}%
    \put(0.04677183,0.18404541){\color[rgb]{0,0,0}\makebox(0,0)[lt]{\begin{minipage}{0.07341517\unitlength}\raggedright $c_j$\end{minipage}}}%
    \put(0.0307804,0.11063022){\color[rgb]{0,0,0}\makebox(0,0)[lt]{\begin{minipage}{0.08359153\unitlength}\raggedright $\delta_j$\end{minipage}}}%
    \put(0.56340092,0.08194601){\color[rgb]{0,0,0}\makebox(0,0)[lt]{\begin{minipage}{0.23848852\unitlength}\raggedright ${\color{red} \varrho(b_{j-1})}$\\ \end{minipage}}}%
    \put(0.53160275,0.05102583){\color[rgb]{0,0,0}\makebox(0,0)[lt]{\begin{minipage}{0.1272045\unitlength}\raggedright $\delta_j$\end{minipage}}}%
    \put(0.59774909,0.11862596){\color[rgb]{0,0,0}\makebox(0,0)[lt]{\begin{minipage}{0.13738087\unitlength}\raggedright $\gamma_j$\end{minipage}}}%
    \put(0.65571507,0.15608808){\color[rgb]{0,0,0}\makebox(0,0)[lt]{\begin{minipage}{0.23848852\unitlength}\raggedright ${\color{red} \varrho(b_{p-1})}$\\ \end{minipage}}}%
    \put(0.75675178,0.14445795){\color[rgb]{0,0,0}\makebox(0,0)[lt]{\begin{minipage}{0.23848852\unitlength}\raggedright ${\color{blue} f(\delta_i)}$\\ \end{minipage}}}%
    \put(0.7067815,0.26327565){\color[rgb]{0,0,0}\makebox(0,0)[lt]{\begin{minipage}{0.16282178\unitlength}\raggedright $\delta_{i_2}$\end{minipage}}}%
    \put(0.19586367,0.04499963){\color[rgb]{0,0,0}\makebox(0,0)[lt]{\begin{minipage}{0.16910069\unitlength}\raggedright ${\color{blue} b_{j-1}}$\end{minipage}}}%
    \put(0.11946012,0.16950773){\color[rgb]{0,0,0}\makebox(0,0)[lt]{\begin{minipage}{0.04724739\unitlength}\raggedright $d_j$\end{minipage}}}%
    \put(0.62599153,0.21312071){\color[rgb]{0,0,0}\makebox(0,0)[lt]{\begin{minipage}{0.02264204\unitlength}\raggedright $P$\end{minipage}}}%
    \put(0.10274181,0.03721504){\color[rgb]{0,0,0}\makebox(0,0)[lt]{\begin{minipage}{0.16064113\unitlength}\raggedright ${\color{magenta}\gamma_j}$\end{minipage}}}%
    \put(0.15507738,0.03503441){\color[rgb]{0,0,0}\makebox(0,0)[lt]{\begin{minipage}{0.1911702\unitlength}\raggedright ${\color{magenta}\delta_i}$\end{minipage}}}%
    \put(0.79741354,0.20639084){\color[rgb]{0,0,0}\rotatebox{-30.48920197}{\makebox(0,0)[lt]{\begin{minipage}{0.105589\unitlength}\raggedright {\Small $\bowtie$}\end{minipage}}}}%
  \end{picture}%
\endgroup%
}
  \caption{Each figure shows a different scenario when $\beta_i=\beta_j$ for some distinct $i,j$.}
   \label{fig:cases}
\end{figure}
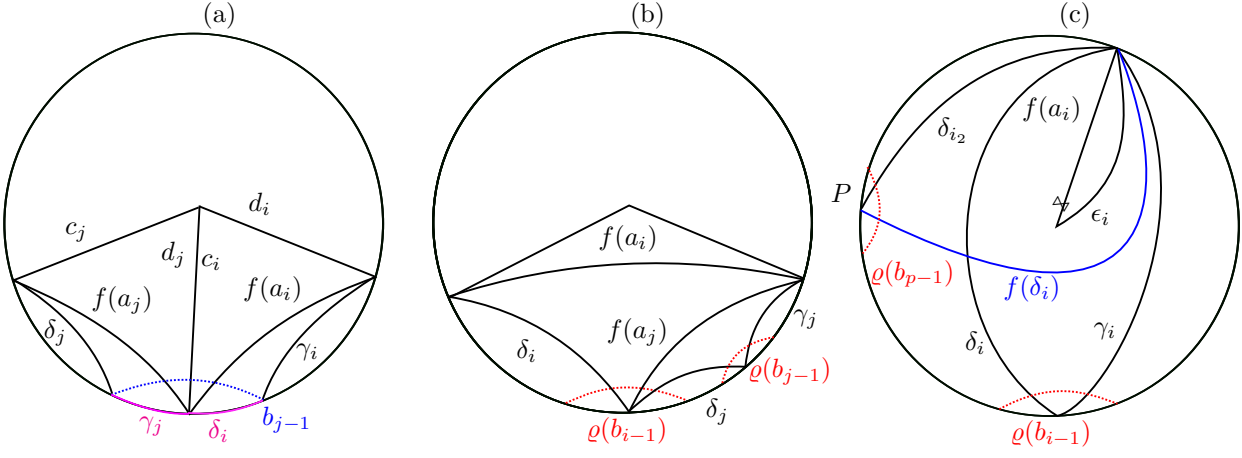

\end{proof}

\begin{theorem}\label{Dn_IV_lower_bound}
Suppose $Q_{\textbf{T}}$ is a quiver of Type IV, where each $Q^{(i)}$ is either empty or a single vertex.  The length of a maximal green sequence for $Q_{\textbf{T}}$ is at least $2k-2+t+m$.   
\end{theorem}

\begin{proof}
By definition  $\mathcal{B}_{\mu_I}({\textbf{T}})=\tilde{\mathcal{B}}_{\mu_I}({\textbf{T}})\cup\varrho({\textbf{T}}) \cup \tilde{\mathcal{C}}_{\mu_I}({\textbf{T}})$.  Moreover, the three sets are disjoint, and $\abs{\varrho({\textbf{T}})}=k+t$, $\abs{\tilde{\mathcal{C}}_{\mu_I}({\textbf{T}})}=k-2$, and by the lemma above $\abs{\tilde{\mathcal{B}}_ {\mu_I}({\textbf{T}})}\geq m$.  Therefore, $\abs{\mathcal{B}_{\mu_I}({\textbf{T}})}\geq 2k-2+t+m$.   
\end{proof}

\begin{figure}[htb]
  \centering
  \rotatebox[origin=c]{90}{\begingroup%
  \makeatletter%
  \def\svgwidth{500pt}
  \providecommand\color[2][]{%
    \errmessage{(Inkscape) Color is used for the text in Inkscape, but the package 'color.sty' is not loaded}%
    \renewcommand\color[2][]{}%
  }%
  \providecommand\transparent[1]{%
    \errmessage{(Inkscape) Transparency is used (non-zero) for the text in Inkscape, but the package 'transparent.sty' is not loaded}%
    \renewcommand\transparent[1]{}%
  }%
  \providecommand\rotatebox[2]{#2}%
  \ifx\svgwidth\undefined%
    \setlength{\unitlength}{529.28527925bp}%
    \ifx\svgscale\undefined%
      \relax%
    \else%
      \setlength{\unitlength}{\unitlength * \real{\svgscale}}%
    \fi%
  \else%
    \setlength{\unitlength}{\svgwidth}%
  \fi%
  \global\let\svgwidth\undefined%
  \global\let\svgscale\undefined%
  \makeatother%
  \begin{picture}(1,0.6096473)%
    \put(0,0){\includegraphics[width=\unitlength]{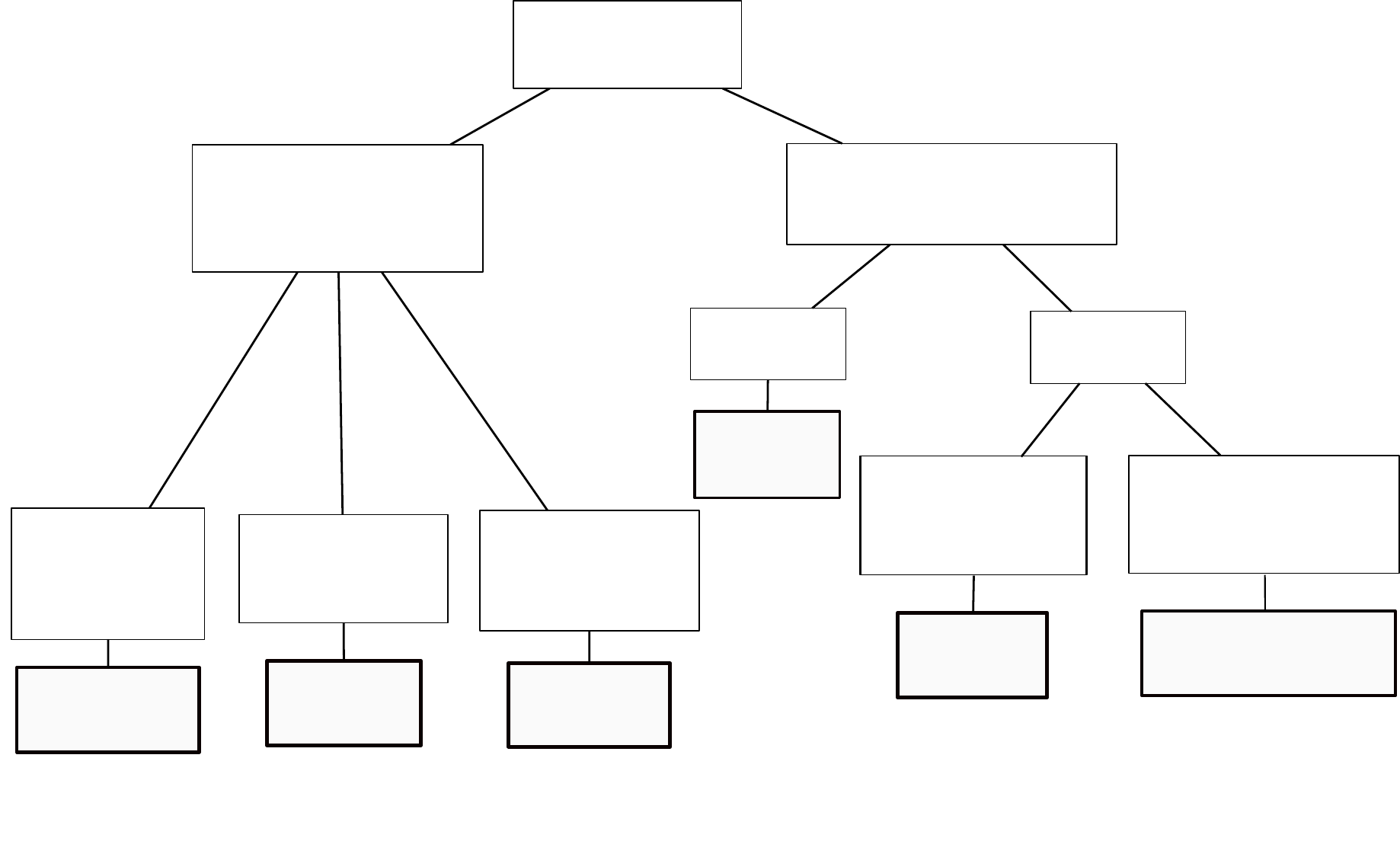}}%
    \put(0.36937639,0.59847337){\color[rgb]{0,0,0}\makebox(0,0)[lt]{\begin{minipage}{0.15739176\unitlength}\centering $\scriptstyle{a_i\in \mathcal{K}}$ {\tiny such that} $\scriptstyle{a_i}$ {\tiny crosses} $\scriptstyle{\varrho(b_{i-1})}$\end{minipage}}}%
    \put(0.15050375,0.49164212){\color[rgb]{0,0,0}\makebox(0,0)[lt]{\begin{minipage}{0.18026743\unitlength}\centering $\scriptstyle{f(a_i)}$ {\tiny is not attached to the puncture} \\ {\tiny see Figure \ref{fig:betas}(a)}\end{minipage}}}%
    \put(0.56942008,0.49432265){\color[rgb]{0,0,0}\makebox(0,0)[lt]{\begin{minipage}{0.22066723\unitlength}\centering $\scriptstyle{f(a_i)}$ {\tiny is (a notched arc) attached to the puncture }\end{minipage}}}%
    \put(0.73890205,0.37071462){\color[rgb]{0,0,0}\makebox(0,0)[lt]{\begin{minipage}{0.10458698\unitlength}\centering $\scriptstyle{f(a_i)\in\varrho(\textbf{T})}$\end{minipage}}}%
    \put(0.02582354,0.23449521){\color[rgb]{0,0,0}\makebox(0,0)[lt]{\begin{minipage}{0.10523305\unitlength}\centering $\scriptstyle{\delta_i, \gamma_i}$ {\tiny are not boundary segments} \end{minipage}}}%
    \put(0.18373558,0.22074783){\color[rgb]{0,0,0}\makebox(0,0)[lt]{\begin{minipage}{0.1314591\unitlength}\centering $\scriptstyle{\delta_i}$ {\tiny is a boundary segment} \end{minipage}}}%
    \put(0.6175485,0.26540984){\color[rgb]{0,0,0}\makebox(0,0)[lt]{\begin{minipage}{0.16205998\unitlength}\centering $\scriptstyle{\epsilon_i}$ {\tiny is  notched \\ see Figure \ref{fig:betas}(b)} \end{minipage}}}%
    \put(0.49766486,0.37228734){\color[rgb]{0,0,0}\makebox(0,0)[lt]{\begin{minipage}{0.10458698\unitlength}\centering $\scriptstyle{f(a_i)\not\in\varrho(\textbf{T})}$\end{minipage}}}%
    \put(0.48098574,0.30417999){\color[rgb]{0,0,0}\makebox(0,0)[lt]{\begin{minipage}{0.13184132\unitlength}\centering {\tiny \bf Case 4}\\ $\scriptstyle{\beta_i=f(a_i)}$\end{minipage}}}%
    \put(0.63486322,0.15918148){\color[rgb]{0,0,0}\makebox(0,0)[lt]{\begin{minipage}{0.1143018\unitlength}\centering {\tiny \bf Case 5}\\ $\scriptstyle{\beta_i=\epsilon_i}$\end{minipage}}}%
    \put(0.81822593,0.1678216){\color[rgb]{0,0,0}\makebox(0,0)[lt]{\begin{minipage}{0.18112571\unitlength}\centering {\tiny \bf Case 6}\\ $\scriptstyle{\beta_i =\phi(\{\delta_i, \gamma_i, \epsilon_i\})}$\end{minipage}}}%
    \put(0.16126293,0.12831239){\color[rgb]{0,0,0}\makebox(0,0)[lt]{\begin{minipage}{0.16411257\unitlength}\centering {\tiny \bf Case 2}\\ $\scriptstyle{\beta_i=c_i}$\end{minipage}}}%
    \put(0.34567117,0.12528007){\color[rgb]{0,0,0}\makebox(0,0)[lt]{\begin{minipage}{0.15100698\unitlength}\centering {\tiny \bf Case 3}\\ $\scriptstyle{\beta_i=d_i}$\end{minipage}}}%
    \put(-0.00646646,0.1213139){\color[rgb]{0,0,0}\makebox(0,0)[lt]{\begin{minipage}{0.16061542\unitlength}\centering {\tiny \bf Case 1}\\ $\scriptstyle{{\beta}_i =\phi(\Delta_i)}$\end{minipage}}}%
    \put(0.33811237,0.22795716){\color[rgb]{0,0,0}\makebox(0,0)[lt]{\begin{minipage}{0.16413289\unitlength}\centering $\scriptstyle{\gamma_i}$ {\tiny is a boundary segment but} $\scriptstyle{\delta_i}$ \\ {\tiny is not}\end{minipage}}}%
    \put(0.82002327,0.26409652){\color[rgb]{0,0,0}\makebox(0,0)[lt]{\begin{minipage}{0.17152281\unitlength}\centering $\scriptstyle{\epsilon_i}$ {\tiny is tagged plain\\ see Figure \ref{fig:betas}(c)}\end{minipage}}}%
  \end{picture}%
\endgroup}

  \caption{The definition of the arc $\beta_i$ associated to $b_{i-1}$ in the proof of Lemma \ref{02}}
   \label{flowchart}
\end{figure}

\subsection{Minimal length maximal green sequences} For convenience, throughout this section, we refer to all tagged triangulations (resp., tagged arcs) simply as triangulations (resp., arcs). We construct minimal length maximal green sequences of $Q_\textbf{T}$ using the notion of a \textbf{fan} of arcs. A \textbf{fan} with respect to $m \in \textbf{M}$ is a sequence $\mathcal{F} = (\gamma_{1}, \ldots, \gamma_{t})$ of arcs in a given triangulation $\textbf{T}^\prime$ of $(\textbf{S}, \textbf{M})$ where \begin{itemize} \item[a)] each $\gamma_{i}$ has $m$ as one of its endpoints and \item[b)] the arc $\gamma_{j+1}$ is immediately clockwise  from $\gamma_{j}$ about $m$ in the triangulation $\textbf{T}^\prime$ of $(\textbf{S}, \textbf{M})$. \end{itemize} A fan $\mathcal{F} = (\gamma_{1}, \ldots, \gamma_{t})$ with respect to $m$ is \textbf{complete} if every arc of $\textbf{T}^\prime$ that is incident to $m$ appears in $\mathcal{F}$ exactly once. 


Given any sequence $\mathcal{S}$ of arcs of a triangulation $\textbf{T}^\prime$, we let $\textbf{i}_{\mathcal{S}}$ denote the sequence of vertices of $Q_{\textbf{T}^\prime}$ in the same order as the arcs of $\mathcal{S}$. In particular, any fan $\mathcal{F}$ defines such a sequence $\textbf{i}_\mathcal{F}$ of vertices of $Q_{\textbf{T}^\prime}$. Similarly, we let $\underline{\mu}_{\mathcal{S}}$ to denote corresponding sequence of flips of arcs in $\mathcal{S}$. If $\mathcal{S}$ is the empty sequence of arcs, then $\textbf{i}_{\mathcal{S}}$ and $\underline{\mu}_{\mathcal{S}}$ will also be empty sequences.

Each nonempty quiver $Q^{(i)}$, determines a complete fan of arcs in $\textbf{T}$ as follows. Define $$\begin{array}{lcl} \mathcal{F}(i) & := & \left\{\begin{array}{lcl} ({b_{i+1}}, {a_{i+1}}, {b_{i}}) & : & \text{if $Q^{(i+1)}$ is nonempty}\\ ({a_{i+1}}, {b_i}) & : & \text{otherwise} \end{array}\right. \end{array}$$ where $i+1$ is calculated\footnote{We will henceforth perform all such calculations mod $k$.} mod $k$ and where  we abuse notation and identify the vertices of $Q$ with the arcs of $\textbf{T}$.

Define the following sequence of vertices of $Q$. If at least one of the quivers $Q^{(i)}$ is the empty quiver or if none of the quivers $Q^{(i)}$ is the empty quiver and there is an even number of quivers $Q^{(i)}$, we let $$\begin{array}{rcl} \textbf{i}_1 & := & \textbf{i}_{\mathcal{F}(i_1)} \circ \cdots \circ \textbf{i}_{\mathcal{F}(i_t)} \\ \underline{\mu}_1 & := & \underline{\mu}_{\mathcal{F}(i_t)}\cdots \underline{\mu}_{\mathcal{F}(i_1)} \end{array}$$ where $\{\mathcal{F}(i_j)\}_{j \in [t]}$ is a set of complete fans of arcs of $\textbf{T}$ as defined in the previous paragraph, enumerated in any order, with the property that an arc of $\textbf{T}$ belongs to at most one of these fans and the set of all arcs appearing in these fans is maximal with respect to this property. 

On the other hand, if none of the quivers $Q^{(i)}$ is the empty quiver and there is an odd number of quivers $Q^{(i)}$, we first fix an arc ${b_i} \in \textbf{T}$. Then we let $$\begin{array}{rcl} \textbf{i}_1 & := & \textbf{i}_{({b_i})}\circ\textbf{i}_{\mathcal{F}(i_1)} \circ \cdots \circ \textbf{i}_{\mathcal{F}(i_t)} \\ \underline{\mu}_1 & := & \underline{\mu}_{\mathcal{F}(i_t)}\cdots \underline{\mu}_{\mathcal{F}(i_1)}{\mu}_{{b_i}} \end{array}$$ where $\{\mathcal{F}(i_j)\}_{j \in [t]}$ is a set of complete fans of arcs of $\mu_{{b_i}}(\textbf{T})$, enumerated in any order, with the property that an arc of $\mu_{{b_i}}(\textbf{T})$ belongs to at most one of these fans and the set of all arcs appearing in these fans is maximal with respect to this property. 

Next, we let $$\begin{array}{rcl} \textbf{i}_2 & := & \textbf{i}_{\mathcal{F}(p,1)}\circ \cdots \circ \textbf{i}_{\mathcal{F}(p,s)}\\ \underline{\mu}_2 & := & \underline{\mu}_{\mathcal{F}(p,s)} \cdots \underline{\mu}_{\mathcal{F}(p,1)}\end{array}$$ where $\{\mathcal{F}(p,j)\}_{j \in [s]}$ is a set of fans of arcs in $\textbf{T} \cap \underline{\mu}_{1}(\textbf{T})$ with respect to $p$, enumerated in any order, such that every arc of $\textbf{T} \cap \underline{\mu}_{1}(\textbf{T})$ with an endpoint at $p$ belongs to one of the fans $\mathcal{F}(p, j)$ and where the set of arcs appearing in each fan $\mathcal{F}(p, j)$ is as large as possible. 

Let $\mathcal{C} := (\alpha_1, \ldots, \alpha_t)$ denote the sequence of arcs in $\underline{\mu}_{2}\underline{\mu}_{1}(\textbf{T}) \cap \underline{\mu}_1(\textbf{T})$, enumerated in any order, where each arc $\alpha_i$ appears in a triangle of $\underline{\mu}_{1}(\textbf{T})$ with the following properties: \begin{itemize} \item one of its other arcs is $\gamma_1$ from $\mathcal{F}(p,j) = (\gamma_1, \ldots, \gamma_t)$---one of the fans defining $\textbf{i}_2$---and  \item the arc $\alpha_i$ is immediately counterclockwise from $\gamma_1$ about their common endpoint\footnote{In other words, there is an arrow in the quiver $Q_{\underline{\mu}_{2}\underline{\mu}_{1}(\textbf{T})}$ whose source (resp., target) is the vertex of $Q_{\underline{\mu}_{2}\underline{\mu}_{1}(\textbf{T})}$ corresponding to $\alpha_i$ (resp., $\gamma_1$).}.  

\end{itemize} {In addition, we require that each arc $\alpha_i$ does not have $p$ as an endpoint.  Let $$\textbf{i}_3 := \left\{\begin{array}{rcl} \textbf{i}_\mathcal{C} & : & \text{if one of the fans defining $\textbf{i}_1$ consists of 3 arcs} \\ \emptyset & : & \text{otherwise,} \end{array}\right. $$ and let $\underline{\mu}_3$ denote the sequence of flips along $\textbf{i}_3$.

Let $\textbf{i}_4 := \textbf{i}_{\mathcal{F}_p}$ and $\underline{\mu}_4 := \underline{\mu}_{\mathcal{F}_p}$ where $\mathcal{F}_p$ is any fan of arcs of $\underline{\mu}_3\underline{\mu}_{2}\underline{\mu}_{1}(\textbf{T})$ with respect to $p$
where each arc appearing in $\mathcal{F}_p$ is plain at $p$ and where the underlying set of arcs of $\mathcal{F}_p$ is as large as possible.

The last component of our sequence is defined inductively. Let $\mathcal{S}_1 = (\sigma^{(1)}_{1}, \ldots, \sigma^{(1)}_{t})$ denote the sequence of arcs of $\underline{\mu}_{4}\underline{\mu}_{3}\underline{\mu}_{2}\underline{\mu}_{1}(\textbf{T})$, enumerated in any order, where each $\sigma^{(1)}_{j}$ satisfies the following: \begin{itemize} \item the arc $\sigma^{(1)}_{j}$ appears in a triangle of $\underline{\mu}_{4}\underline{\mu}_{3}\underline{\mu}_{2}\underline{\mu}_{1}(\textbf{T})$ whose other two arcs are notched at $p$, and \item the arc $\sigma^{(1)}_{j}$ does not appear in a triangle of $\underline{\mu}_{4}\underline{\mu}_{3}\underline{\mu}_{2}\underline{\mu}_{1}(\textbf{T})$ whose other two arcs are boundary arcs. \end{itemize} We define $\mathcal{S}_i$ analogously. Assume that $\mathcal{S}_{i-1}$ has been defined and now let $\mathcal{S}_i = (\sigma^{(i)}_{1}, \ldots, \sigma^{(i)}_{t})$ denote the sequence of arcs of $\underline{\mu}_{\mathcal{S}_{i-1}}\cdots\underline{\mu}_{\mathcal{S}_1}\underline{\mu}_{4}\underline{\mu}_{3}\underline{\mu}_{2}\underline{\mu}_{1}(\textbf{T})$, enumerated in any order, where each $\sigma^{(i)}_{j}$ satisfies the following: \begin{itemize} \item the arc $\sigma^{(i)}_{j}$ appears in a triangle of $\underline{\mu}_{\mathcal{S}_{i-1}}\cdots\underline{\mu}_{\mathcal{S}_1}\underline{\mu}_{4}\underline{\mu}_{3}\underline{\mu}_{2}\underline{\mu}_{1}(\textbf{T})$ whose other two arcs are notched at $p$, and \item the arc $\sigma^{(i)}_{j}$ does not appear in a triangle of $\underline{\mu}_{\mathcal{S}_{i-1}}\cdots\underline{\mu}_{\mathcal{S}_1}\underline{\mu}_{4}\underline{\mu}_{3}\underline{\mu}_{2}\underline{\mu}_{1}(\textbf{T})$ whose other two arcs are boundary arcs. \end{itemize} Define $\textbf{i}_5 := \textbf{i}_{\mathcal{S}_1} \circ \cdots \circ \textbf{i}_{\mathcal{S}_r}$ and $\underline{\mu}_5 := \underline{\mu}_{\mathcal{S}_r} \cdots \underline{\mu}_{\mathcal{S}_1}$ where $\mathcal{S}_{r+1}$ is the empty sequence of arcs. 

\begin{theorem}\label{Thm:Dn_short_seq}
Let $Q$ be a Type IV quiver where each of its subquivers $Q^{(i)}$ with $i = 1,\ldots, k$ either consists of a single vertex or is the empty quiver. Then the sequence $\textbf{i}_1\circ \textbf{i}_2 \circ \textbf{i}_3 \circ \textbf{i}_4\circ \textbf{i}_5$ is a maximal green sequence. 
\end{theorem}

\begin{corollary}\label{Cor:Dn_short_seq_length}
The length of the maximal green sequence $\textbf{i}_1\circ \textbf{i}_2\circ \textbf{i}_3 \circ \textbf{i}_4 \circ \textbf{i}_5$ is $2k_\text{in} - 2 + m + k_\text{out}$ where $k_\text{in} = |\{\text{arcs of $\textbf{T}$ connected to $p$}\}|$, $k_\text{out} = |\{\text{arcs of $\textbf{T}$ not connected to $p$}\}|$, and $$m = |\{\text{arcs $\gamma \in \textbf{T}$ not connected to $p$ where $\varrho^2(\gamma) \in \textbf{T}$ is also not connected to $p$}\}|.\footnote{Note that $k_\text{in} = k$ and $m = \{j \in [k]: \text{deg}(a_j) = 4\}$.}$$
\end{corollary}

\begin{example}
Here we give an example of the maximal green sequence $\textbf{i}_1\circ \textbf{i}_2 \circ \textbf{i}_3 \circ \textbf{i}_4 \circ \textbf{i}_5$ from Theorem~\ref{Thm:Dn_short_seq} in the context of Figure~\ref{Dn_seq_example}. We have that $\textbf{i}_1 = (a_7, b_6, a_4, b_3, b_2, a_2, b_1),$ $\textbf{i}_2 = (a_3, a_5, a_6, a_1),$ $\textbf{i}_3 = (i_\alpha)$, $\textbf{i}_4 = (i_\delta)$, and $\textbf{i}_5 = (i_{\sigma^{(1)}_1}, i_{\sigma^{(1)}_2}, i_{\sigma^{(2)}_1}, i_{\sigma^{(2)}_2}, i_{\sigma^{(3)}_1}).$ Notice that the arc $\alpha^\prime$ obtained by flipping $\alpha$ also satisfies $\alpha^\prime = \varrho({b_2})$, as desired.
\end{example}

\begin{figure}
\includegraphics[scale=.8]{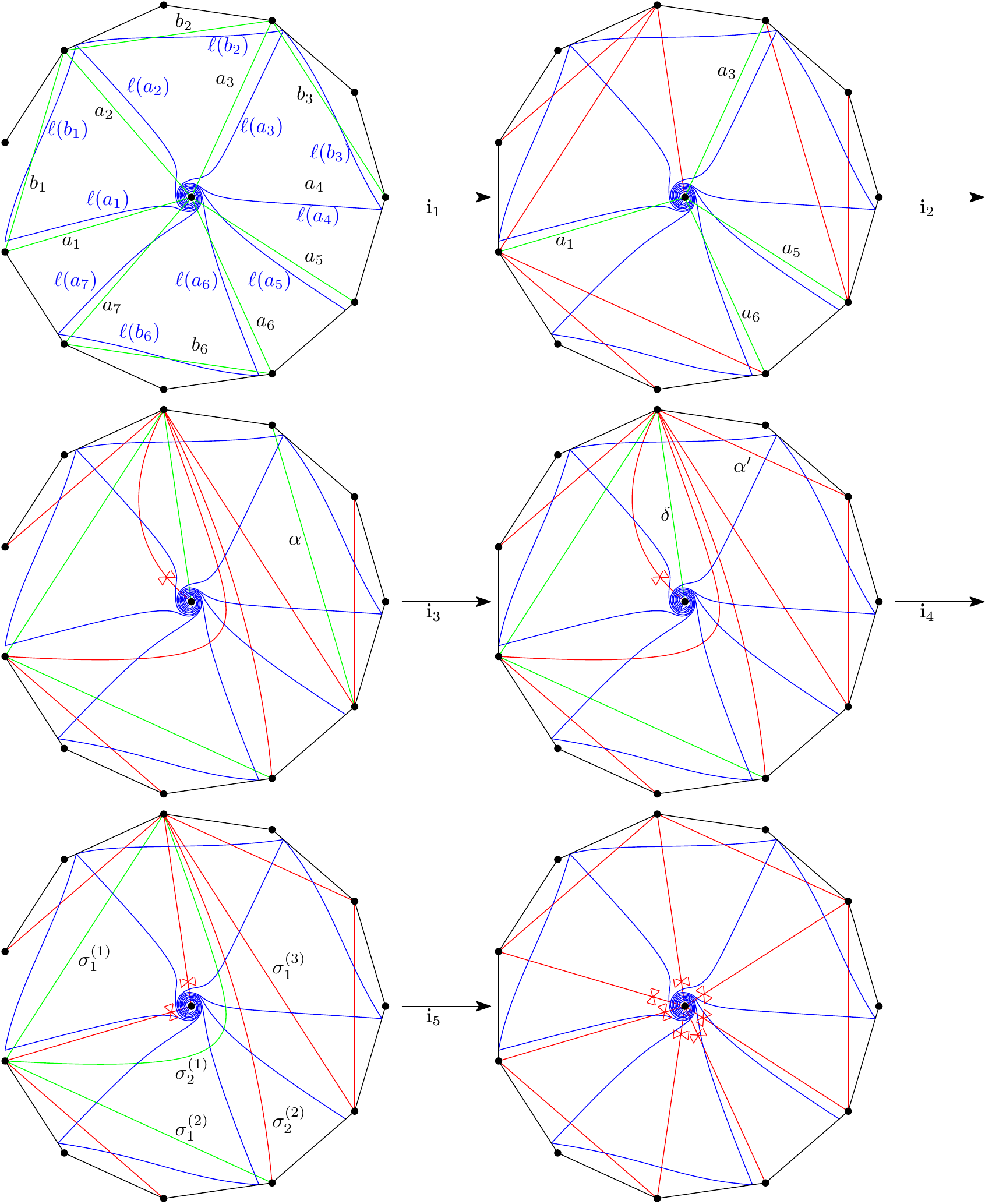}
\caption{An example of a maximal green sequence $\textbf{i}_1\circ\textbf{i}_2\circ\textbf{i}_3\circ\textbf{i}_4\circ\textbf{i}_5$.}
\label{Dn_seq_example}
\end{figure}


In the proof of Theorem~\ref{Thm:Dn_short_seq}, we will need to consider how a specific part of a triangulation $\textbf{T}_1$ of $(\textbf{S}, \textbf{M})$ is affected by flips. In particular, we will need to understand how a single flip $\mu_\gamma: \textbf{T}_1 \to \textbf{T}_1\backslash\{\gamma\}\sqcup\{\gamma^\prime\}$ affects the colors of arcs in $\textbf{T}_1\backslash\{\gamma\}$ that are close to $\gamma$ in $\textbf{T}_1$. We formulate what we mean by a specific part of a triangulation $\textbf{T}_1$ as follows. 

Let $\textbf{T}_1$ be a triangulation of $(\textbf{S}, \textbf{M})$ and let $\mathcal{C} \subset \textbf{T}_1$. We define ${(\textbf{S}(\mathcal{C}), \textbf{M}(\mathcal{C}))}$ to be the marked surface obtained from $(\textbf{S}, \textbf{M})$ by taking the disjoint union of all ideal quadrilaterals in $\textbf{T}_1$ of arcs appearing in $\mathcal{C}$ and then identifying any common triangles appearing in these ideal quadrilaterals. The resulting (possibly disconnected) surface ${(\textbf{S}(\mathcal{C}), \textbf{M}(\mathcal{C}))}$ clearly has a triangulation given by the arcs of $\mathcal{C}$.  If we are working with a specific set of arcs $\{\alpha_1, \ldots, \alpha_t\} \in \textbf{T}_1$, for convenience, we write ${(\textbf{S}(\alpha_1, \ldots, \alpha_t), \textbf{M}(\alpha_1,\ldots, \alpha_t))}$ to denote ${(\textbf{S}(\{\alpha_1, \ldots, \alpha_t\}), \textbf{M}(\{\alpha_1,\ldots, \alpha_t\}))}$. Similarly, when we are considering a fan $\mathcal{F}$ of arcs in $\textbf{T}_1$, we write ${(\textbf{S}(\mathcal{F}), \textbf{M}(\mathcal{F}))}$ to denote ${(\textbf{S}(\mathcal{C}), \textbf{M}(\mathcal{C}))}$ where $\mathcal{C}$ is the underlying set of arcs in $\mathcal{F}$. Note that for such a fan $\mathcal{F}$ the surface ${(\textbf{S}(\mathcal{F}), \textbf{M}(\mathcal{F}))}$ will be connected. 

The following facts, which we will use in the proof of Theorem~\ref{Thm:Dn_short_seq}, are straightforward to verify.

\begin{lemma}\label{Lem:subsurfaces}
Let $\mathcal{C}$ and $\mathcal{C}^\prime$ be two subsets of arcs of a triangulation $\textbf{T}_1 \in \overrightarrow{EG}(\textbf{T})$ of $(\textbf{S}, \textbf{M})$, let $Q^\prime_\textbf{T} \in \overrightarrow{EG}(\widehat{Q}_\textbf{T})$ be the quiver corresponding to $\textbf{T}_1$, and let $(\gamma_1,\ldots, \gamma_s)$ and $(\delta_1,\ldots, \delta_t)$ be total orderings of the arcs in $\mathcal{C}$ and $\mathcal{C}^\prime$, respectively.
\begin{itemize}
\item[$a)$] If the intersection of $(\textbf{S}(\mathcal{C}),\textbf{M}(\mathcal{C}))$ and $(\textbf{S}(\mathcal{C}^\prime),\textbf{M}(\mathcal{C}^\prime))$, regarding the two as subspaces of $(\textbf{S},\textbf{M})$, is empty or consists only of boundary arcs of these two surfaces, then $${\mu}_{\gamma_s}\cdots\mu_{\gamma_1}{\mu}_{\delta_t}\cdots\mu_{\delta_1}Q^\prime_\textbf{T} = {\mu}_{\delta_t}\cdots\mu_{\delta_1}{\mu}_{\gamma_s}\cdots\mu_{\gamma_1}Q^\prime_\textbf{T}.$$

\item[$b)$] If an arc $\gamma \in \textbf{T}_1$ is neither an interior arc nor a boundary arc of $(\textbf{S}(\mathcal{C}),\textbf{M}(\mathcal{C}))$, then the color of $\gamma$ as an arc of ${\mu}_{\gamma_s}\cdots\mu_{\gamma_1}(\textbf{T}_1)$ is the same as its color as an arc of $\textbf{T}_1.$ 
\end{itemize}
\end{lemma}

\begin{proof}[Proof of Theorem~\ref{Thm:Dn_short_seq}]
We assume that at least one of the quivers $Q^{(i)}$ is the empty quiver or that none of the quivers $Q^{(i)}$ is the empty quiver and $k$ is even. The proof when each of the quivers $Q^{(i)}$ consists of a single vertex and $k$ is odd is similar so we omit it. We prove that $\textbf{i}_1\circ \textbf{i}_2 \circ \textbf{i}_3 \circ \textbf{i}_4\circ \textbf{i}_5$ is a maximal green sequence by analyzing the effect each subsequence $\underline{\mu}_j$ of $\underline{\mu}_5\cdots \underline{\mu}_1$ has on the triangulation $\underline{\mu}_{j-1} \cdots \underline{\mu}_1(\textbf{T})$. In the course of the proof, we also justify why in the definition of the sequence $\textbf{i}_i$ for $i = 1, 2, 3, 4$ we can enumerate the subsequences defining it in any order. Similarly, we explain why the arcs in $\mathcal{S}_i$ for $i = 1, \ldots, r$ may be flipped in any order.


We first consider the effect of $\underline{\mu}_1$ on $\textbf{T}$. Let $\textbf{i}_1 = \textbf{i}_{\mathcal{F}(i_1)} \circ \cdots \circ \textbf{i}_{\mathcal{F}(i_t)}$. One observes that the intersection of the surfaces $(\textbf{S}(\mathcal{F}(i_j)), \textbf{M}(\mathcal{F}(i_{j})))$ and $(\textbf{S}(\mathcal{F}(i_{j^\prime})), \textbf{M}(\mathcal{F}(i_{j^\prime})))$ where $j \neq j^\prime$ is either empty or consists of a boundary arc of these surfaces. Therefore, by Lemma~\ref{Lem:subsurfaces} $a)$, we can assume $i_1 > i_2 > \cdots > i_t$. Now the sequence $\textbf{i}_1$ can be expressed as the composition of a unique family of sequences of the form $\textbf{i}_{\mathcal{F}(i)}\circ \textbf{i}_{\mathcal{F}(i-2)}\circ \cdots \circ \textbf{i}_{\mathcal{F}(j+2)} \circ \textbf{i}_{\mathcal{F}(j)}$ where $j = 1$ and $i = k-1$ if each of the quivers $Q^{(1)}, \ldots, Q^{(k)}$ is a single vertex (and, in this case, there is exactly one sequence in this family) or $Q^{(i+2)}$ and $Q^{(j-1)}$ are empty quivers. Given such a sequence $\textbf{i}_{\mathcal{F}(i)}\circ\textbf{i}_{\mathcal{F}(i-2)} \circ \cdots \circ \textbf{i}_{\mathcal{F}(j+2)} \circ \textbf{i}_{\mathcal{F}(j)}$, we show the two possible ways that $\underline{\mu}_{\mathcal{F}(j)}\underline{\mu}_{\mathcal{F}(j+2)} \cdots \underline{\mu}_{\mathcal{F}(i-2)}\underline{\mu}_{\mathcal{F}(i)}$ may affect the triangulation $\textbf{T}$ in Figure~\ref{i1_transformation}. It is clear that all flips in $\underline{\mu}_1$ take place at green arcs. 


\begin{figure}
$$\begin{array}{cccc}\includegraphics[scale=.9]{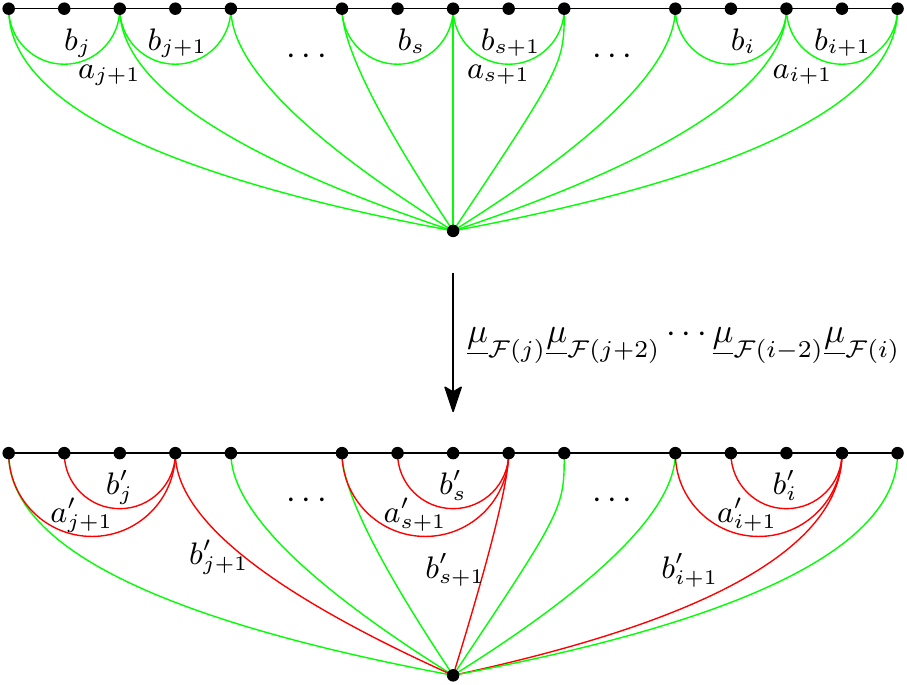} & \includegraphics[scale=.9]{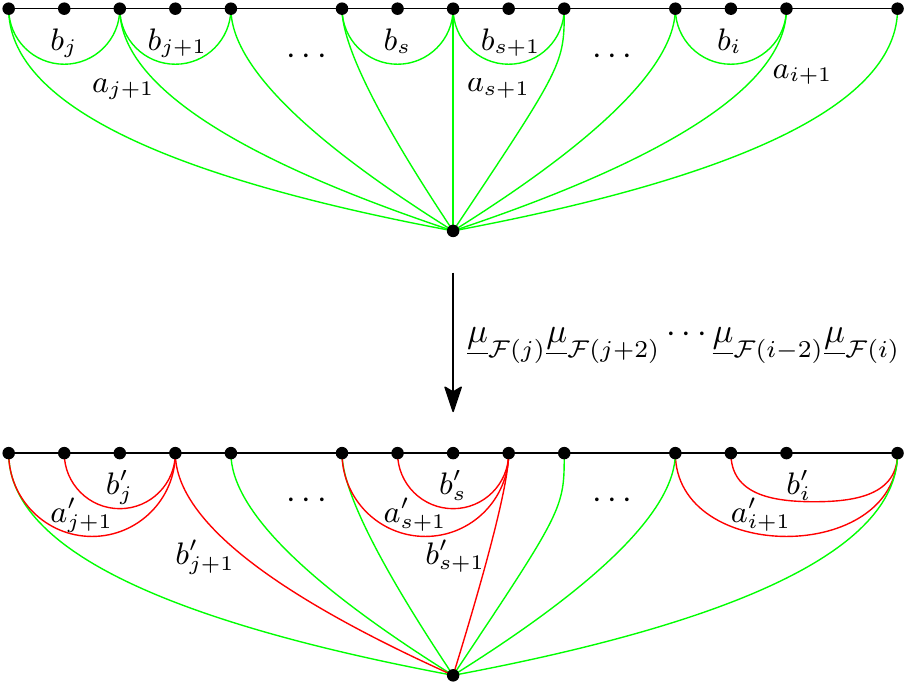} \\ (a) & (b) \end{array}$$
\caption{The effect of the sequence of flips $\underline{\mu}_{\mathcal{F}(j)}\underline{\mu}_{\mathcal{F}(j+2)} \cdots \underline{\mu}_{\mathcal{F}(i-2)}\underline{\mu}_{\mathcal{F}(i)}$ on the triangulation $\textbf{T}$. The quiver $Q^{(i+1)}$ is a single vertex in $(a)$ and it is empty in $(b)$.}
\label{i1_transformation}
\end{figure}



Next, we consider the effect of $\underline{\mu}_2$ on $\underline{\mu}_1(\textbf{T})$. From the description of $\underline{\mu}_1(\textbf{T})$ given by Figure~\ref{i1_transformation}, the surfaces $(\textbf{S}(\mathcal{F}(p,i)), \textbf{M}(\mathcal{F}(p, i)))$ and $(\textbf{S}(\mathcal{F}(p,j)), \textbf{M}(\mathcal{F}(p, j)))$ with $i \neq j$ can only intersect at their boundaries. By Lemma~\ref{Lem:subsurfaces} $a)$, we can enumerate the fans appearing in $\textbf{i}_2$ in any order.

\begin{figure}

$$\begin{array}{cccc} \includegraphics[scale=.9]{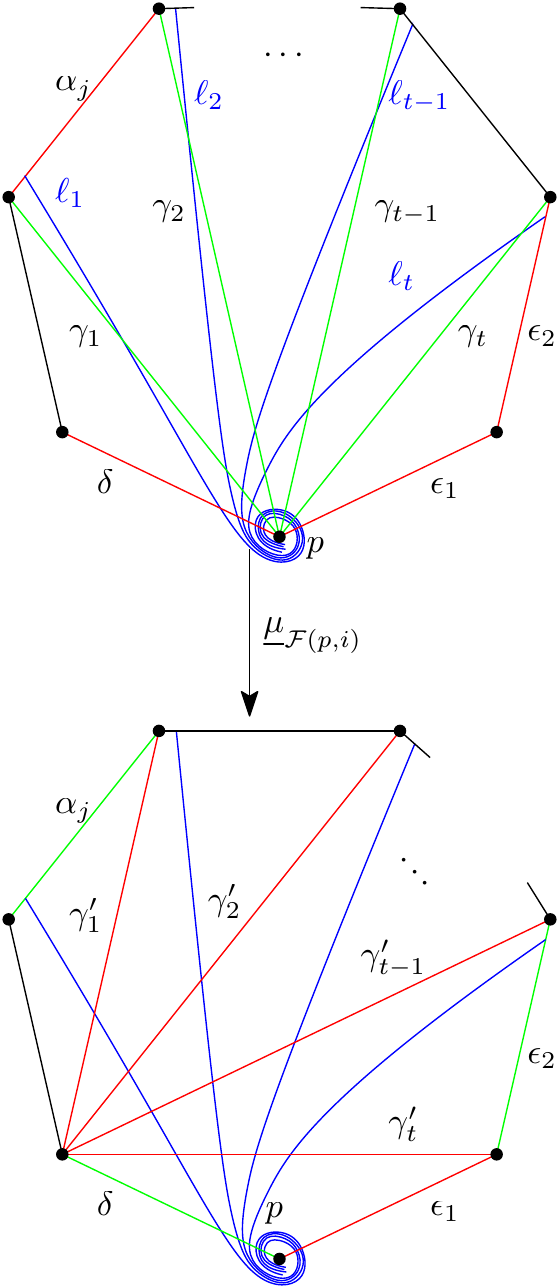} & \raisebox{.24in}{\includegraphics[scale=.9]{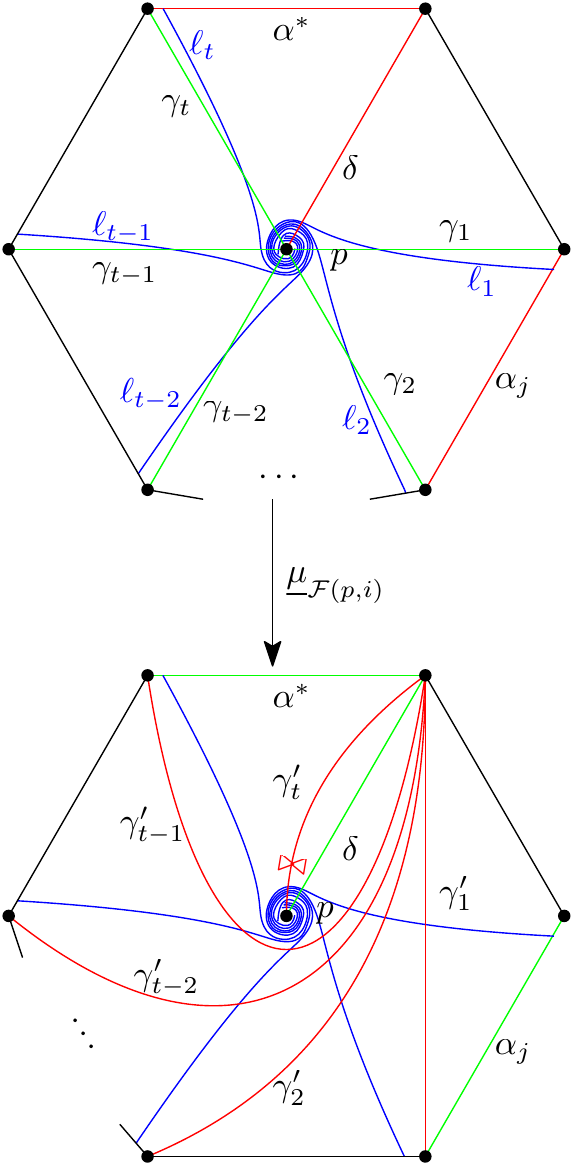}} & \raisebox{.24in}{\includegraphics[scale=.9]{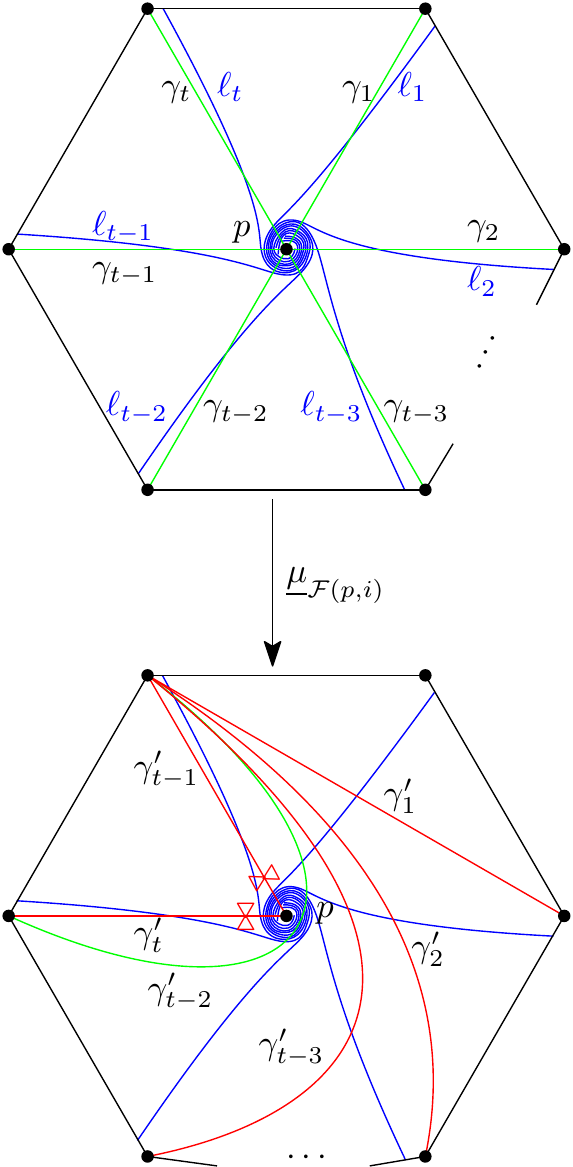}} \\ (a) & (b) & (c) \end{array}$$
\caption{The three ways in which $\underline{\mu}_{\mathcal{F}(p,i)}$ may affect $\underline{\mu}_1(\textbf{T})$ where $\mathcal{F}(p,i) = (\gamma_1, \ldots, \gamma_t)$ is any fan appearing in $\textbf{i}_2$. Here we only show the effect on the triangulated surface $(\textbf{S}(\mathcal{F}(p,i)), \textbf{M}(\mathcal{F}(p,i)))$. In $(a),$ this surface is an unpunctured disk, and in $(b)$ and $(c)$, this surface is a once-punctured disk. In this figure and in later figures, we let $\ell_i$ denote the elementary lamination defined by $\gamma_i$. We let $\delta = {b^\prime_r}$, $\epsilon_1 = {b^\prime_\ell}$, and $\epsilon_2 = {a^\prime_\ell}$  where ${b_r}$, ${b_\ell}$, and ${a_\ell}$ were flipped when we applied $\underline{\mu}_1$. We use the same notation for these arcs in Figures~\ref{i2_transformation_a_intermed} and \ref{i2_transformation_b_intermed}. The arc $\alpha^* = {a^\prime_\ell}$ in $(b)$ where $a_\ell$ was flipped when we applied $\underline{\mu}_1$, and this is the same arc $\alpha^*$ that appears in later figures. In $(a)$ and $(b)$, if $\alpha_j$ is not a boundary arc of $(\textbf{S}, \textbf{M})$, then it appears in $\textbf{i}_3$.}
\label{i2_transformation}
\end{figure}



\begin{figure}
$$\begin{array}{cccccc} \includegraphics[scale=.5]{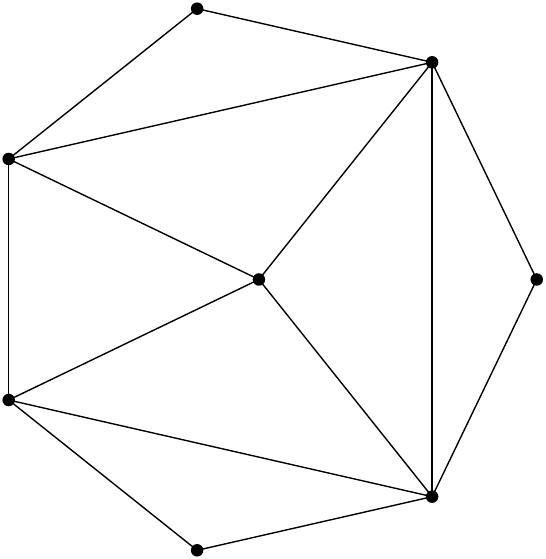} & \includegraphics[scale=.5]{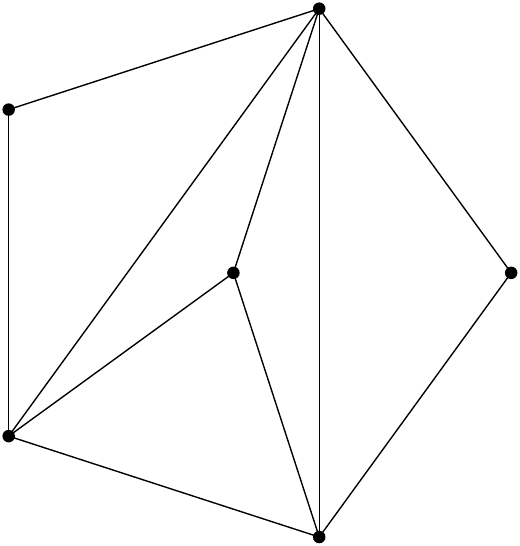} & \includegraphics[scale=.5]{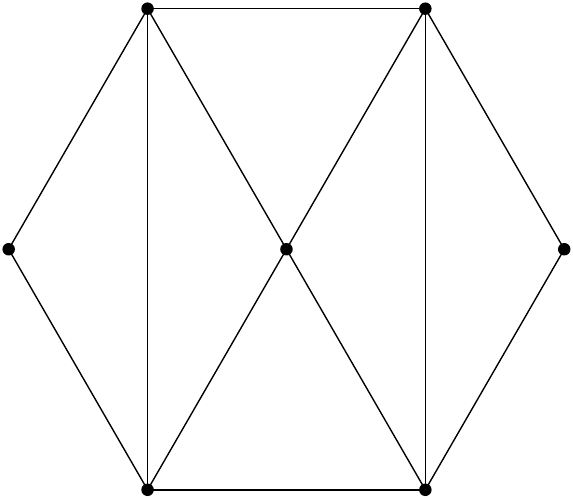} \\ (b_1^\prime) & (b_2^\prime) & (c^\prime) \end{array}$$
\caption{In $(c^\prime)$ (resp., $(b_1^\prime)$ and $(b_2^\prime)$), we show the triangulation(s) where $\textbf{i}_2$ consists of a single fan with exactly two arcs.}
\label{special_cases_b_c}
\end{figure}

The sequence $\textbf{i}_2$ is either defined by at least two fans $\mathcal{F}(p,i)$ and $\mathcal{F}(p,j)$ or it is defined by a single fan $\mathcal{F}(p,i)$. Let $\mathcal{F}(p,i) = (\gamma_1, \ldots, \gamma_t)$ be one of the fans defining $\textbf{i}_2$. In the former case, the surface $(\textbf{S}(\mathcal{F}(p,i)), \textbf{M}(\mathcal{F}(p,i)))$ is an unpunctured disk, triangulated by a fan triangulation (see Figure~\ref{i2_transformation} $(a)$). In the latter case, the surface is a once-punctured disk, triangulated by a fan of arcs about the puncture $p$ (see Figure~\ref{i2_transformation} $(b)$ and $(c)$)\footnote{These three cases can equivalently be described by saying Figure~\ref{i2_transformation} $(a)$, Figure~\ref{i2_transformation} $(b)$, and Figure~\ref{i2_transformation} $(c)$, respectively, is the case in which among the fans defining $\textbf{i}_1$ at least two fans consist of three arcs, exactly one fan consists of three arcs, and no fans consist of three arcs.}. Since Figure~\ref{i1_transformation} shows that the only green arcs of $\underline{\mu}_1(\textbf{T})$ are those in $\textbf{T}\cap \underline{\mu}_1(\textbf{T})$, all flips in $\underline{\mu}_2$ take place at green arcs.

One checks that if $\textbf{i}_2 = \textbf{i}_{\mathcal{F}(p,1)}$, then $\mathcal{F}(p,1)$ consists of at least two arcs. In the case when $\mathcal{F}(p,1)$ consists of exactly two arcs, the triangulation $\textbf{T}$ is one of three shown in Figure~\ref{special_cases_b_c}. In these cases, we leave it to the reader to verify that $\textbf{i}_1 \circ \textbf{i}_2 \circ \textbf{i}_3 \circ \textbf{i}_4 \circ \textbf{i}_5$ is a maximal green sequence, and we focus on the generic case when $t > 2$. 

In Figure~\ref{i2_transformation}, we also show the effect of performing the transformation $\underline{\mu}_{\mathcal{F}(p,i)}$ on the triangulated surface $(\textbf{S}(\mathcal{F}(p,i)), \textbf{M}(\mathcal{F}(p,i)))$ in each of these three cases. The Figures~\ref{i2_transformation_a_intermed}, \ref{i2_transformation_b_intermed}, and \ref{i2_transformation_c_intermed} justify why the resulting triangulation of $(\textbf{S}(\mathcal{F}(p,i)), \textbf{M}(\mathcal{F}(p,i)))$ is what appears in the lower image in Figure~\ref{i2_transformation} and that the arcs in that triangulation have the indicated colors.  



\begin{figure}
$$\begin{array}{cccc} \includegraphics[scale=.9]{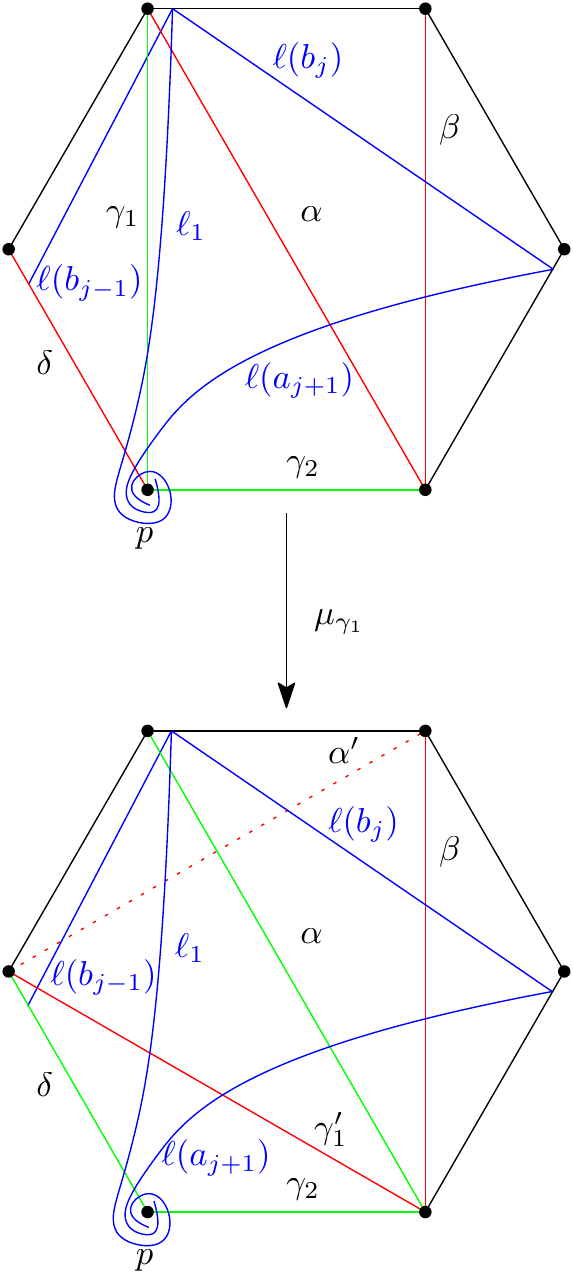} & \includegraphics[scale=.9]{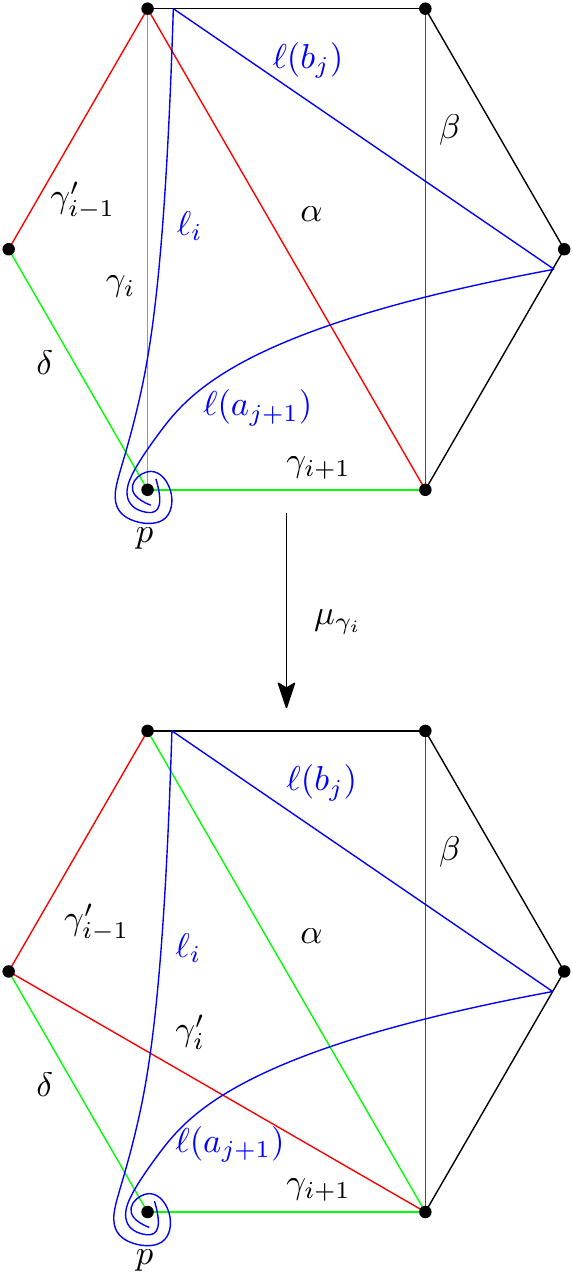} & \includegraphics[scale=.9]{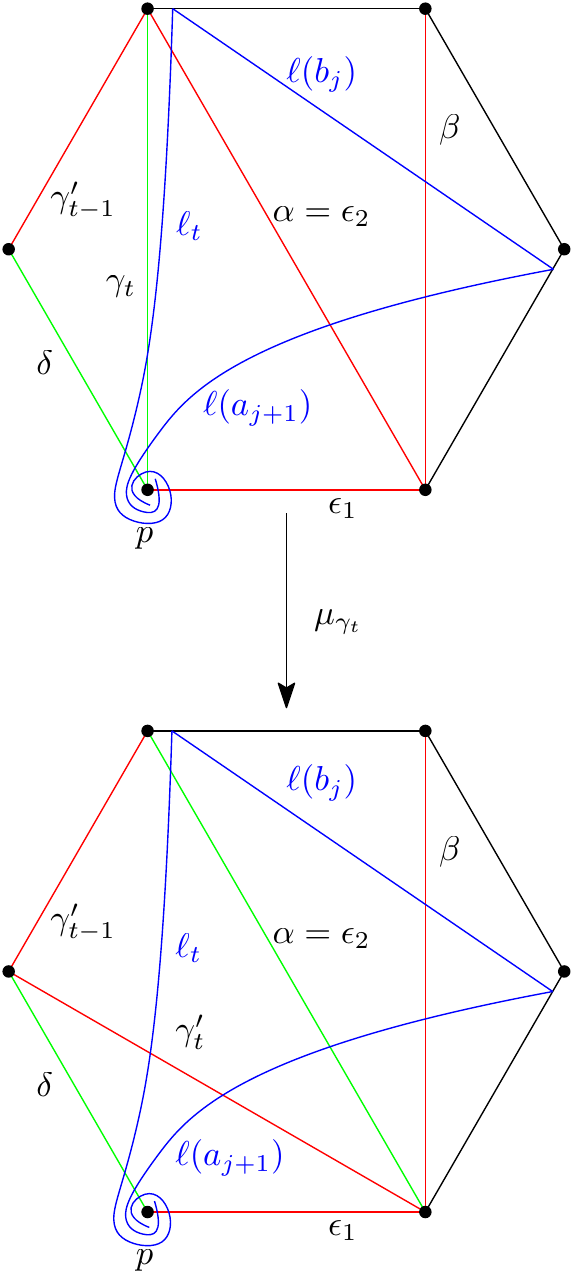} \\ (a_1) & (a_2) & (a_3) \end{array}$$\caption{The effect of applying $\underline{\mu}_{\mathcal{F}(p,i)}$ to $\underline{\mu}_1(\textbf{T})$ when the surface $(\textbf{S}(\mathcal{F}(p,i)), \textbf{M}(\mathcal{F}(p,i)))$ is an unpunctured disk, as in Figure~\ref{i2_transformation} $(a)$. The arc $\alpha$ may be a boundary arc of $(\textbf{S}, \textbf{M})$, in which case the arc $\beta$ and the black arcs in the ideal quadrilateral containing $\beta$ should be removed from the picture. If not, then $\alpha ={a^\prime_{j+1}}$ and $\beta ={b^\prime_j}$ where $a_{j+1}$ and $b_j$ were flipped when we applied $\underline{\mu}_1$. The black arcs are boundary arcs of $(\textbf{S}, \textbf{M})$.}
\label{i2_transformation_a_intermed}
\end{figure}

\begin{figure}
$$\begin{array}{cccc} \includegraphics[scale=.9]{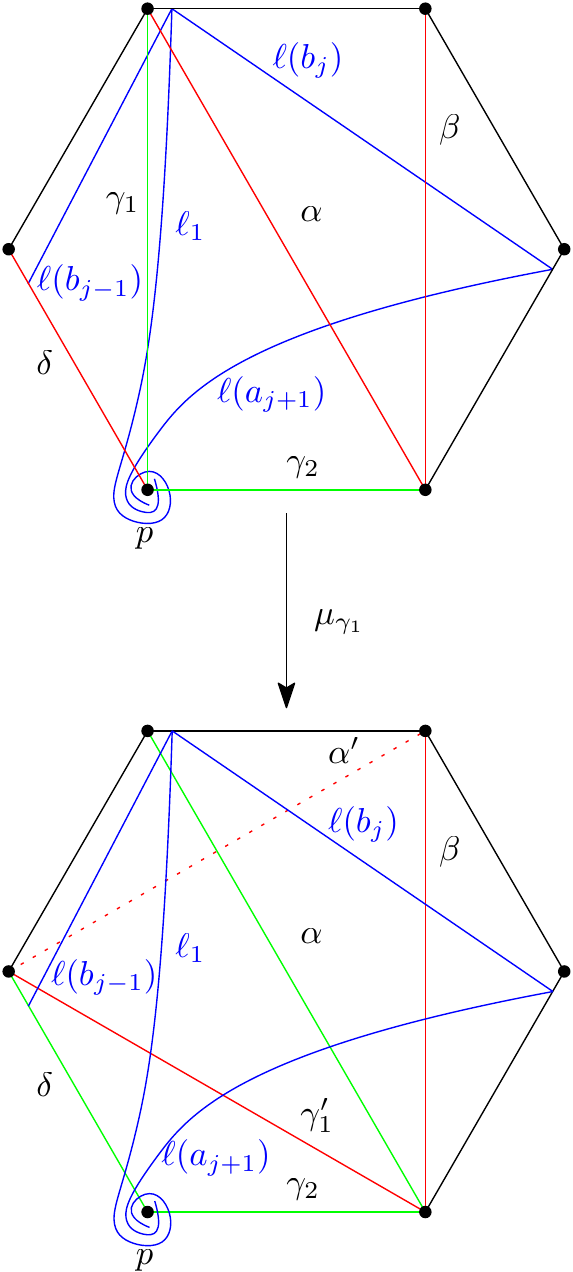} & \raisebox{.2in}{\includegraphics[scale=.9]{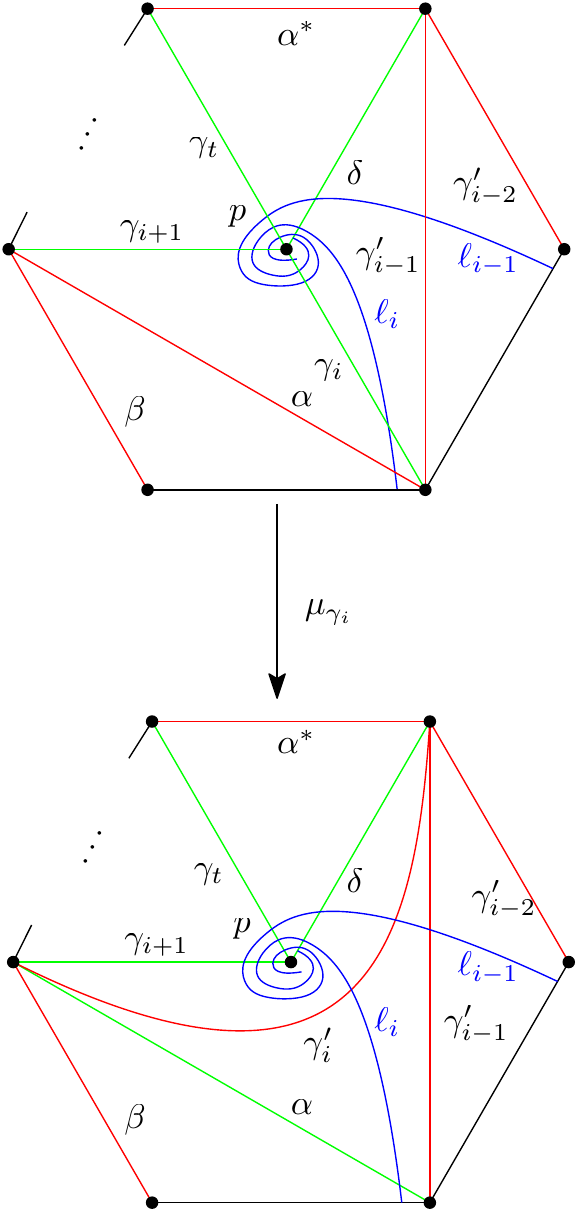}} & \raisebox{.1in}{\includegraphics[scale=.9]{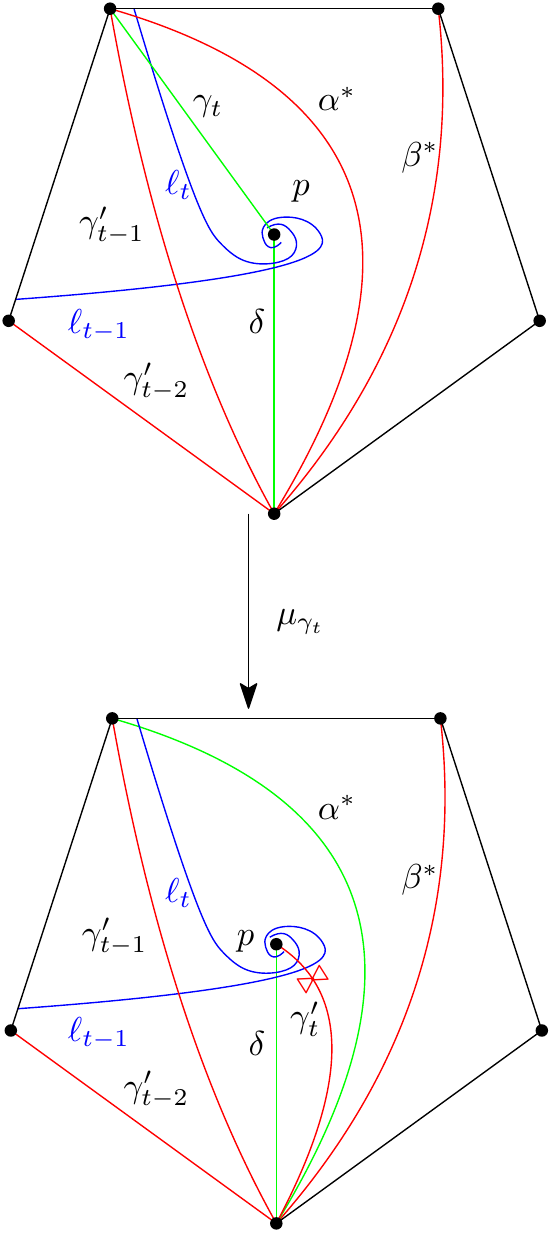}} \\ (b_1) & (b_2) & (b_3) \end{array}$$\caption{The effect of applying $\underline{\mu}_{\mathcal{F}(p,i)}$ to $\underline{\mu}_1(\textbf{T})$ when the surface $(\textbf{S}(\mathcal{F}(p,i)), \textbf{M}(\mathcal{F}(p,i)))$ is a punctured disk, as in Figure~\ref{i2_transformation} $(b)$. Any arc $\alpha$ may be a boundary arc of $(\textbf{S}, \textbf{M})$, in which case the arc $\beta$ and the black arcs in the ideal quadrilateral containing $\beta$ should be removed from the picture. If not, then $\alpha ={a^\prime_{j+1}}$ and $\beta ={b^\prime_j}$ where $a_{j+1}$ and $b_j$ were flipped when we applied $\underline{\mu}_1$. Also, $\alpha^* = a^\prime_{\ell+1}$ and $\beta^* = b^\prime_\ell$ where $a_{\ell+1}$ and $b_\ell$ were flipped when we applied $\underline{\mu}_1$. The arc $\alpha^*$ is not a boundary arc of $(\textbf{S}, \textbf{M})$. The black arcs in $(b_1)$ are boundary arcs of $(\textbf{S}, \textbf{M})$.}
\label{i2_transformation_b_intermed}
\end{figure}

\begin{figure}
$$\begin{array}{cccccc} \includegraphics[scale=.9]{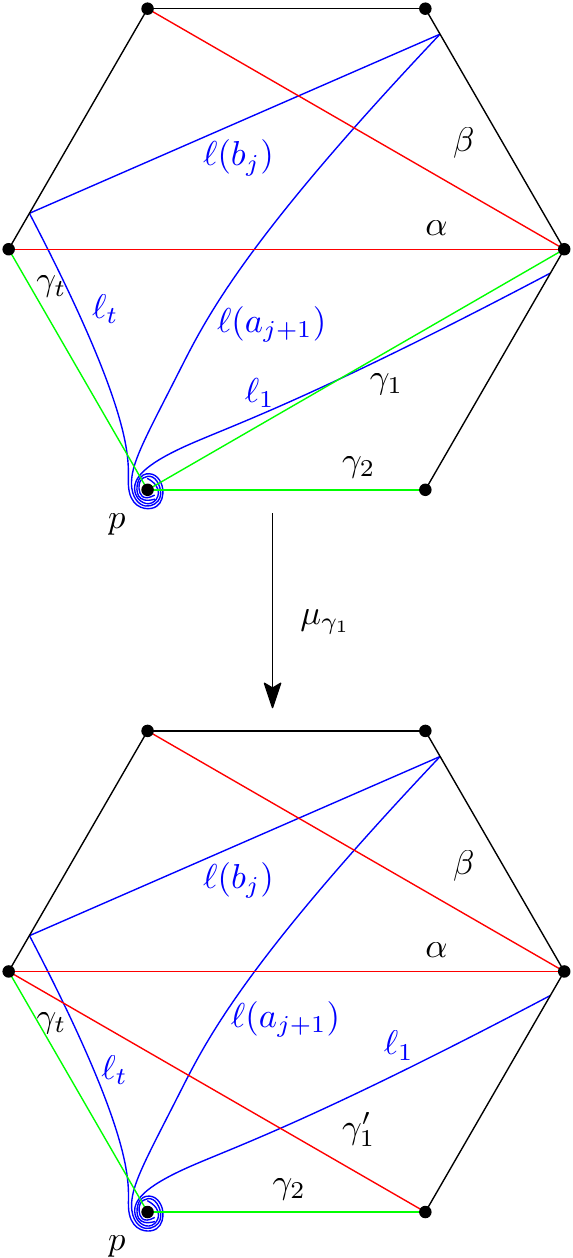} & \includegraphics[scale=.9]{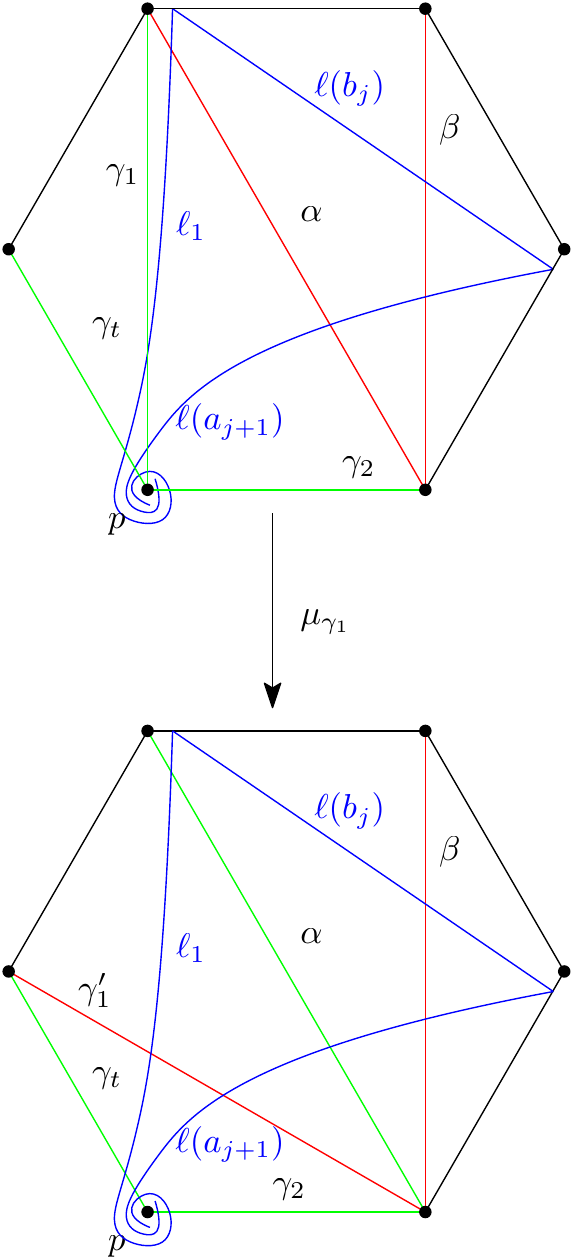} & \raisebox{.175in}{\includegraphics[scale=.9]{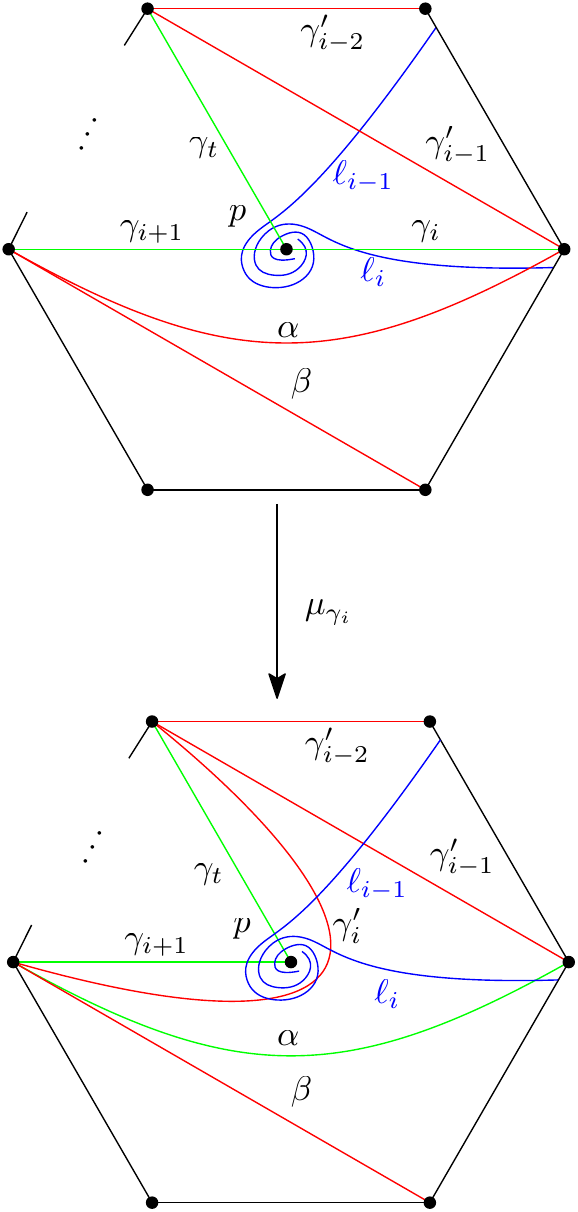}} \\ (c_1) & (c_1^\prime) & (c_2) \end{array}$$ $$\begin{array}{cccccccccc} \includegraphics[scale=.9]{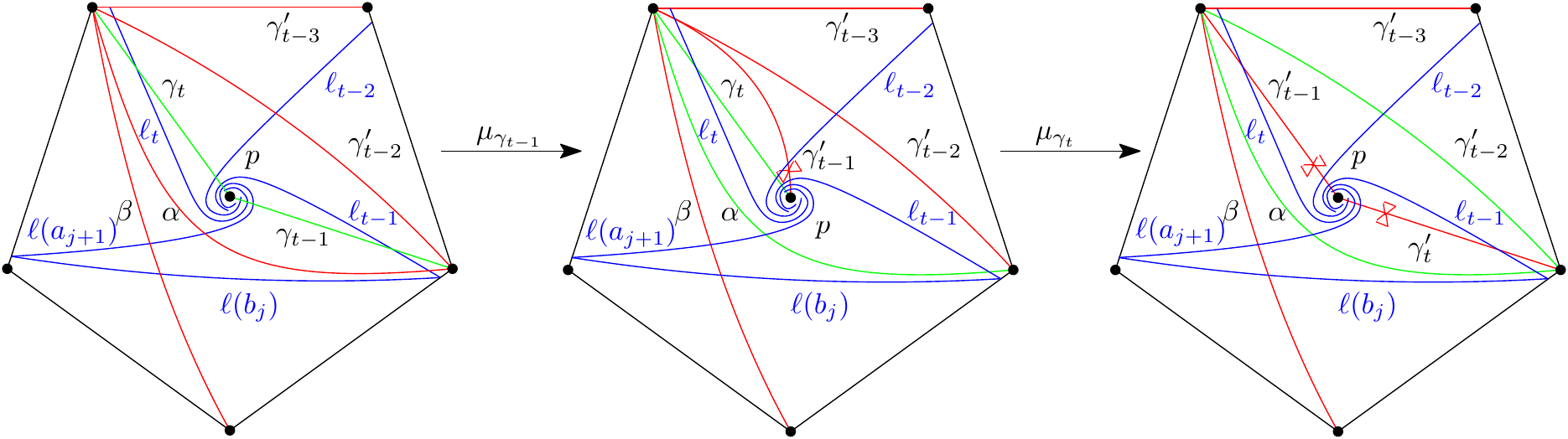} \\ (c_3) \end{array}$$
\caption{The effect of applying $\underline{\mu}_{\mathcal{F}(p,i)}$ to $\underline{\mu}_1(\textbf{T})$ when the surface $(\textbf{S}(\mathcal{F}(p,i)), \textbf{M}(\mathcal{F}(p,i)))$ is a punctured disk, as in Figure~\ref{i2_transformation} $(c)$. The arc $\alpha$ may be a boundary arc of $(\textbf{S}, \textbf{M})$, in which case the arc $\beta$ and the black arcs in the ideal quadrilateral containing $\beta$ should be removed from the picture. If not, then $\alpha ={a^\prime_{j+1}}$ and $\beta ={b^\prime_j}$ where $a_{j+1}$ and $b_j$ were flipped when we applied $\underline{\mu}_1$. In $(c_1)$ and $(c_1^\prime)$, we show the two possible configurations in which $\gamma_1$ may appear. The black arcs in $(c_1)$ and $(c_1^\prime)$ are boundary arcs of $(\textbf{S}, \textbf{M}).$}
\label{i2_transformation_c_intermed}
\end{figure}

The Figures~\ref{i2_transformation_a_intermed}, \ref{i2_transformation_b_intermed}, and \ref{i2_transformation_c_intermed} also show that an arc of $\underline{\mu}_2\underline{\mu}_1(\textbf{T})$ one of whose endpoints is $p$ is green if and only if the arc is plain at $p$. Furthermore, we conclude that in the cases of Figures~\ref{i2_transformation} $(a)$ and $(b)$ an arc $\gamma$ of $\underline{\mu}_2\underline{\mu}_1(\textbf{T})$ none of whose endpoints is $p$ is green if and only if $\gamma = {a^\prime_{j+1}}$ for some $j+1 \in [k]$ where ${a_{j+1}}$ was flipped when applying $\underline{\mu}_1$. Lastly, in the case of Figure~\ref{i2_transformation} $(c)$, an arc $\gamma$ of $\underline{\mu}_2\underline{\mu}_1(\textbf{T})$ none of whose endpoints is $p$ is green if and only if $\gamma = {a^\prime_{j+1}}$ for some $j+1 \in [k]$ and the arcs $\gamma, \gamma_\ell,$ and $\gamma_{\ell+1}$ with $\ell = 2, \ldots, t-1$ form a triangle in $\underline{\mu}_1(\textbf{T})$ or $\gamma = \gamma_{t-2}^\prime$.


We now apply $\underline{\mu}_3$ to $\underline{\mu}_2\underline{\mu}_1(\textbf{T}).$ By referring to Figures~\ref{i2_transformation_a_intermed}, \ref{i2_transformation_b_intermed}, and \ref{i2_transformation_c_intermed}, we see that applying $\underline{\mu}_3$ to $\underline{\mu}_2\underline{\mu}_1(\textbf{T})$ means that we flip exactly the green arcs $\alpha$ from Figure~\ref{i2_transformation_a_intermed} $(a_1)$ and Figure~\ref{i2_transformation_b_intermed} $(b_1)$. Observe that none of these flips turn any red arcs green. Furthermore, the surfaces $(\textbf{S}(\mathcal{F}(\alpha_i)), \textbf{M}(\mathcal{F}(\alpha_i)))$ and $(\textbf{S}(\mathcal{F}(\alpha_j)), \textbf{M}(\mathcal{F}(\alpha_j)))$ do not intersect nontrivially so the arcs defining $\textbf{i}_3$ may be flipped in any order.   

Next, we apply $\underline{\mu}_4$ to $\underline{\mu}_3\underline{\mu}_2\underline{\mu}_1(\textbf{T}).$ Note that $\text{supp}(\textbf{i}_4) = \emptyset$ if and only if the sequence $\textbf{i}_2$ is defined by a single fan $\mathcal{F}(p,i)$ where $(\textbf{S}(\mathcal{F}(p,i)), \textbf{M}(\mathcal{F}(p,i)))$ is a once-punctured disk with a fan triangulation given by the arcs of $\mathcal{F}(p,i)$ as in Figure~\ref{i2_transformation} $(c)$. 


If $\text{supp}(\textbf{i}_4) \neq \emptyset$, then every boundary arc of the surface $(\textbf{S}(\mathcal{F}_p), \textbf{M}(\mathcal{F}_p))$ is an internal arc of the triangulation $\underline{\mu}_3\underline{\mu}_2\underline{\mu}_1(\textbf{T})$. Indeed, if $\textbf{i}_2$ is defined by at least two fans (resp., a single fan), the boundary arcs of $(\textbf{S}(\mathcal{F}_p), \textbf{M}(\mathcal{F}_p))$ are exactly the arcs $\gamma_t^\prime$ from Figure~\ref{i2_transformation} $(a)$ (resp., $\alpha^*$ and $\gamma_{t-1}^\prime$ from Figure~\ref{i2_transformation} $(b)$). In the former case, every boundary arc of $(\textbf{S}(\mathcal{F}_p), \textbf{M}(\mathcal{F}_p))$ is red in $\underline{\mu}_3\underline{\mu}_2\underline{\mu}_1(\textbf{T})$, and, in the latter case, the arc $\gamma_{t-1}^\prime$ is the only boundary arc of $(\textbf{S}(\mathcal{F}_p), \textbf{M}(\mathcal{F}_p))$ that is red in $\underline{\mu}_3\underline{\mu}_2\underline{\mu}_1(\textbf{T})$. 

Let $\mathcal{D} := \{\text{arcs of }\mathcal{F}_p\} \sqcup \{\text{boundary arcs of } (\textbf{S}(\mathcal{F}_p), \textbf{M}(\mathcal{F}_p))\}$ and consider $(\textbf{S}(\mathcal{D}), \textbf{M}(\mathcal{D}))$. We show how $\underline{\mu}_4$ affects $\underline{\mu}_3\underline{\mu}_2\underline{\mu}_1(\textbf{T})$ in Figure~\ref{i4_transformation}. In particular, the elementary laminations appearing in Figure~\ref{i4_transformation} $(a)$ do appear because we have previously determined that the arcs $\delta_1, \ldots, \delta_s$ (resp., the boundary arcs of $(\textbf{S}(\mathcal{F}_p), \textbf{M}(\mathcal{F}_p))$) are green (resp., red) in $\underline{\mu}_3\underline{\mu}_2\underline{\mu}_1(\textbf{T})$. Furthermore, every flip in $\underline{\mu}_4$ occurs at a green arc.

\begin{figure}
$$\begin{array}{cccccc} \includegraphics[scale=.9] {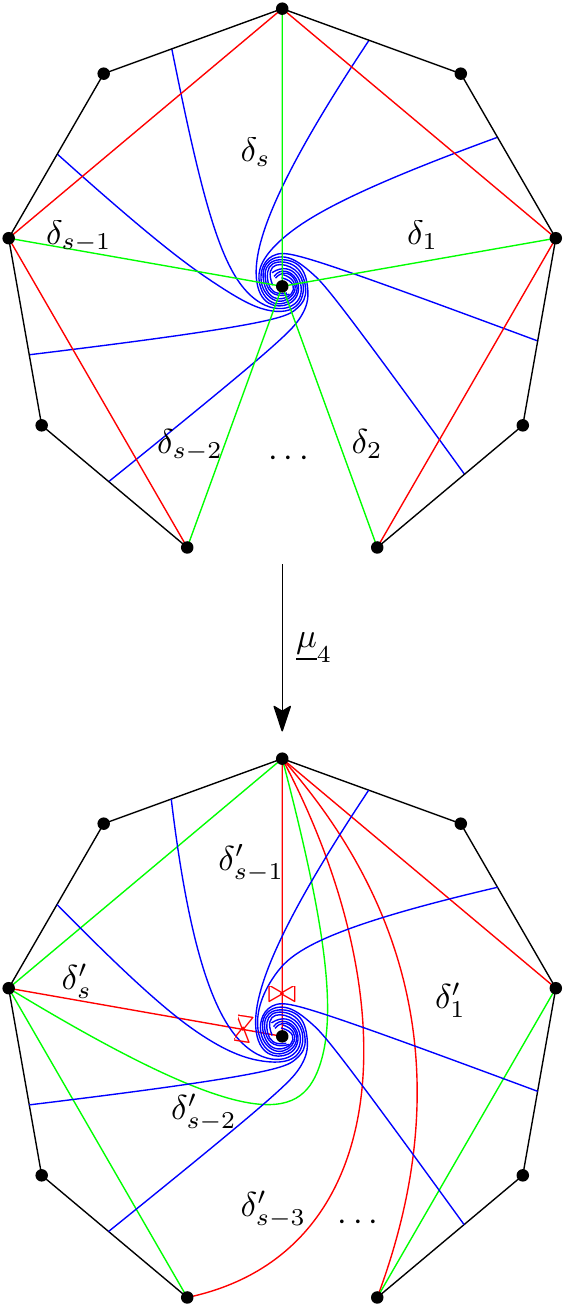} &\includegraphics[scale=.9]{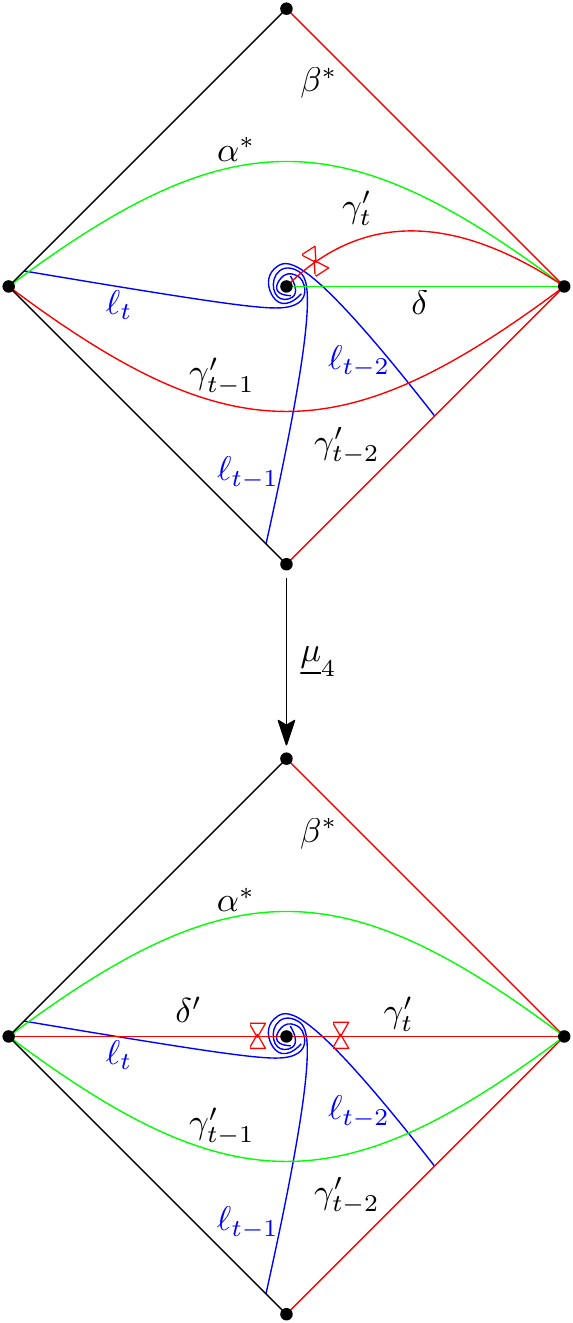} \\ (a) & (b)\end{array}$$
\caption{We show the effect of $\underline{\mu}_4$ on $\underline{\mu}_3\underline{\mu}_2\underline{\mu}_1(\textbf{T})$ by showing its effect on the arcs in $(\textbf{S}(\mathcal{D}), \textbf{M}(\mathcal{D}))$. In $(a)$, we write $\mathcal{F}_p = (\delta_1,\ldots, \delta_s)$ and $s \ge 2$. In $(b)$, we have $\mathcal{F}_p = (\delta)$.}
\label{i4_transformation}
\end{figure}

Lastly, we apply $\underline{\mu}_5 = \underline{\mu}_{\mathcal{S}_r}\cdots \underline{\mu}_{\mathcal{S}_1}$ to $\underline{\mu}_4\underline{\mu}_3\underline{\mu}_2\underline{\mu}_1(\textbf{T}).$ By the definition of the arcs in $\mathcal{S}_i$ and Lemma~\ref{Lem:subsurfaces} $a)$, we can enumerate the arcs appearing in $\mathcal{S}_i$ is any order. We notice from Figures~\ref{i2_transformation_c_intermed} and \ref{i4_transformation} and our description of the boundary arcs of $(\textbf{S}(\mathcal{D}), \textbf{M}(\mathcal{D}))$ that any arc appearing in $\mathcal{S}_1$ is of the form $\alpha, \alpha^*, \gamma^\prime_t, \gamma^\prime_{t-1}, \gamma^\prime_{t-2},$ or $\delta_{s-2}^\prime$ and each arc appearing in $\mathcal{S}_1$ is green in $\underline{\mu}_4\underline{\mu}_3\underline{\mu}_2\underline{\mu}_1(\textbf{T})$. In fact, Figures~\ref{i2_transformation}, \ref{i2_transformation_a_intermed}, \ref{i2_transformation_b_intermed}, \ref{i2_transformation_c_intermed}, and \ref{i4_transformation} and our description of the boundary arcs of $(\textbf{S}(\mathcal{D}), \textbf{M}(\mathcal{D}))$ show that each arc of $\mathcal{S}_i$ is of the form $\alpha, \alpha^*, \gamma^\prime_j,$ or $\delta_{j}^\prime$. 

Next, we show that each flip in $\underline{\mu}_5$ takes place at a green arc. We have already shown that each flip in $\underline{\mu}_{\mathcal{S}_1}$ takes place at a green arc. Now assume that for some $i < r$ the arcs in ${\mathcal{S}_j}$ are green in $\underline{\mu}_{\mathcal{S}_{j-1}} \cdots \underline{\mu}_{\mathcal{S}_1}\underline{\mu}_4\underline{\mu}_3\underline{\mu}_2\underline{\mu}_1(\textbf{T})$ for each $j<i$.  In Figure~\ref{i5_transformation} $(a)$, we show the generic configuration of arcs from $\underline{\mu}_{\mathcal{S}_{i-1}}\cdots \underline{\mu}_{\mathcal{S}_1}\underline{\mu}_{4}\underline{\mu}_3\underline{\mu}_2\underline{\mu}_1(\textbf{T})$ around an arc $\sigma^{(i)}_{j_1}$ appearing in $\mathcal{S}_i$.   Here, either $\sigma^{(i)}_{j_1} = \delta^\prime_\ell$ for some arc $\delta_\ell$ in $\mathcal{F}_p$ where $\ell \in [s-2]$ or $\sigma^{(i)}_{j_1} = \gamma_\ell^\prime$ for some $\ell > 1$ and $\gamma_\ell$ is an arc in one of the fans defining $\textbf{i}_2$. 


Suppose $\sigma^{(i)}_{j_1} = \delta^\prime_\ell$. Then Figures~\ref{i2_transformation}, \ref{i2_transformation_a_intermed}, \ref{i2_transformation_b_intermed}, \ref{i2_transformation_c_intermed}, and \ref{i4_transformation} and our description of the boundary arcs of $(\textbf{S}(\mathcal{D}), \textbf{M}(\mathcal{D}))$ shows that $\sigma^{(i+1)}_{j_2+1} = \gamma^\prime_t$ where $\gamma_t$ is the last arc to be flipped in some fan defining $\textbf{i}_2$ and $\sigma^{(i+2)}_{j_3+1}$ is either a boundary component of $(\textbf{S}, \textbf{M})$ or it is an arc of the form $\alpha$. Now assume $\sigma^{(i)}_{j_1} = \gamma_\ell^\prime$ where $\ell > 1$. Then Figures~\ref{i2_transformation}, \ref{i2_transformation_a_intermed}, \ref{i2_transformation_b_intermed}, and \ref{i2_transformation_c_intermed} show that the arcs $\sigma^{(i+1)}_{j_2+1}$ and $\sigma^{(i+2)}_{j_3+1}$ may be boundary components of $(\textbf{S}, \textbf{M})$ or these are arcs of the form $\alpha$. 

We can now conclude that if an arc in the left image in Figure~\ref{i5_transformation} $(a)$ is not an boundary arc of $(\textbf{S}, \textbf{M})$, then it has the indicated color in $\underline{\mu}_{\mathcal{S}_{i-1}}\cdots \underline{\mu}_{\mathcal{S}_1}\underline{\mu}_4\underline{\mu}_3\underline{\mu}_2\underline{\mu}_1(\textbf{T})$. The elementary laminations $\ell^{j_2}$ (resp., $\ell^{j_2+1}$) in Figure~\ref{i5_transformation} $(a)$ are those witnessing that $\sigma^{(i+1)}_{j_2}$ (resp., $\sigma^{(i+1)}_{j_2+1}$) is red (resp., green) in $\underline{\mu}_{\mathcal{S}_{i-1}}\cdots \underline{\mu}_{\mathcal{S}_1}\underline{\mu}_4\underline{\mu}_3\underline{\mu}_2\underline{\mu}_1(\textbf{T})$. These elementary laminations also show that the arcs other than $\tau$ in Figure~\ref{i5_transformation} $(a)$ have the indicated color in $\underline{\mu}_{\mathcal{S}_{i}}\cdots \underline{\mu}_{\mathcal{S}_1}\underline{\mu}_4\underline{\mu}_3\underline{\mu}_2\underline{\mu}_1(\textbf{T})$. We will justify that $\tau$ is red in $\underline{\mu}_{\mathcal{S}_{i}}\cdots \underline{\mu}_{\mathcal{S}_1}\underline{\mu}_4\underline{\mu}_3\underline{\mu}_2\underline{\mu}_1(\textbf{T})$ shortly.

The remaining cases are that $\sigma^{(i)}_{j_1} = \gamma_1^\prime$ where $\gamma_1$ is the first arc to be flipped in some fan defining $\textbf{i}_2$ or that $\sigma^{(i)}_{j_1}$ is an arc $\alpha$ or $\alpha^*$. If $\sigma^{(i)}_{j_1} = \gamma_1^\prime$, then Figures~\ref{i2_transformation}, \ref{i2_transformation_a_intermed}, \ref{i2_transformation_b_intermed}, \ref{i2_transformation_c_intermed} show that $\sigma^{(i)}_{j_1}$ appears in one of the leftmost configurations of arcs in Figure~\ref{i5_transformation} $(b)$, $(c)$, or $(d)$. In Figures~\ref{i5_transformation} $(b)$ and $(c)$, the arcs $\sigma^{(i+1)}_{j_2+1}$ and $\sigma^{(i+1)}_{j_2}$ are either boundary arcs of $(\textbf{S},\textbf{M})$ or they are arcs $\alpha$ or $\alpha^*$. If $\sigma^{(i)}_{j_1}$ is an arc $\alpha$ or $\alpha^*$,  then $\sigma^{(i)}_{j_1}$ plays the role of arc $\sigma^{(i+1)}_{j_2}$ in the middle configuration of Figures~\ref{i5_transformation} $(b)$ and $(c)$. 

Using Figures~\ref{i2_transformation}, \ref{i2_transformation_a_intermed}, \ref{i2_transformation_b_intermed}, \ref{i2_transformation_c_intermed}, one checks that the elementary laminations shown in Figure~\ref{i5_transformation} $(b)$, $(c)$, or $(d)$ are present in these configurations. Furthermore, one checks that these elementary laminations show that the arcs other than $\tau$ have the indicated colors in $\underline{\mu}_{\mathcal{S}_{i}}\cdots \underline{\mu}_{\mathcal{S}_1}\underline{\mu}_4\underline{\mu}_3\underline{\mu}_2\underline{\mu}_1(\textbf{T})$. However, each arc $\tau$ is realizable as an arc $\sigma$ in a configuration of one of the forms shown in Figure~\ref{i5_transformation} about some arc $\sigma^{(i^\prime)}_{\ell} \in \mathcal{S}_{i^\prime}$ where $\sigma^{(i^\prime)}_{\ell} \neq \sigma^{(i)}_{j_1}$, but $\sigma^{(i^\prime)}_{\ell}$ plays the role of $\sigma^{(i)}_{j_1}$ in the configuration. Thus $\tau$ will be red in $\underline{\mu}_{\mathcal{S}_{i}}\cdots \underline{\mu}_{\mathcal{S}_1}\underline{\mu}_4\underline{\mu}_3\underline{\mu}_2\underline{\mu}_1(\textbf{T})$.

We conclude that all flips in $\underline{\mu}_{\mathcal{S}_i}$ take place at green arcs. By induction, we see that all flips of $\underline{\mu}_{5}$ take place at green arcs. Moreover, the rightmost configurations in Figure~\ref{i5_transformation} $(b)$, $(c)$, and $(d)$ show that all arcs in $\underline{\mu}_5\underline{\mu}_4\underline{\mu}_3\underline{\mu}_2\underline{\mu}_1(\textbf{T})$ are red. Thus $\textbf{i}_1\circ \textbf{i}_2\circ \textbf{i}_3\circ \textbf{i}_4 \circ \textbf{i}_5$ is a maximal green sequence.\end{proof}

\begin{figure}
$$\begin{array}{ccccc} \includegraphics[scale=.8]{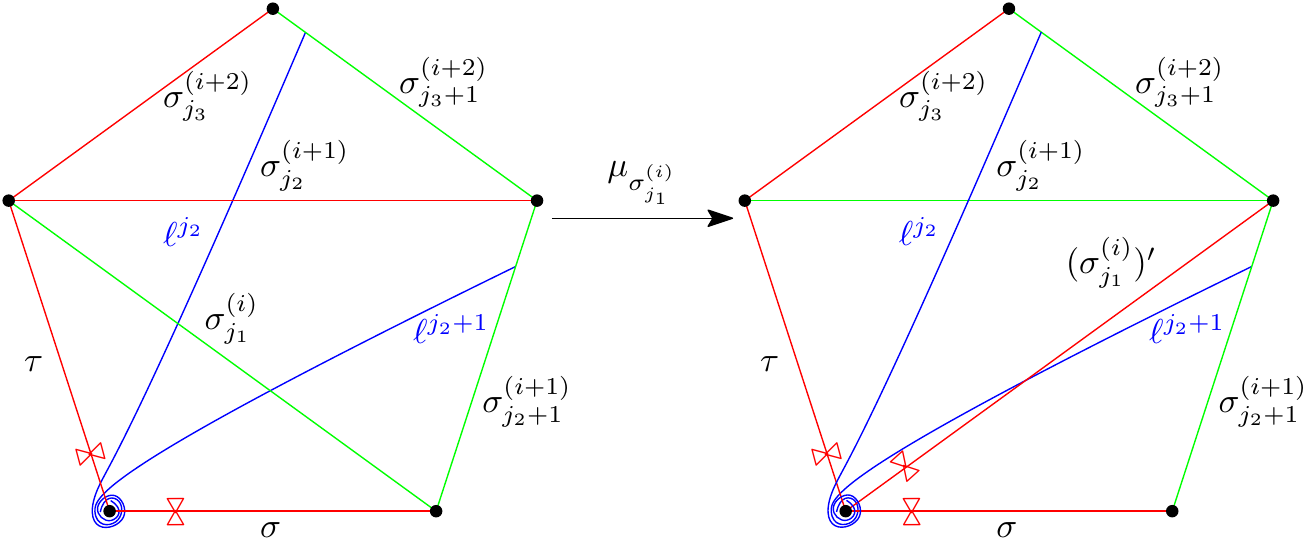} \\ (a) \end{array} $$
$$\begin{array}{ccccc} \includegraphics[scale=.8]{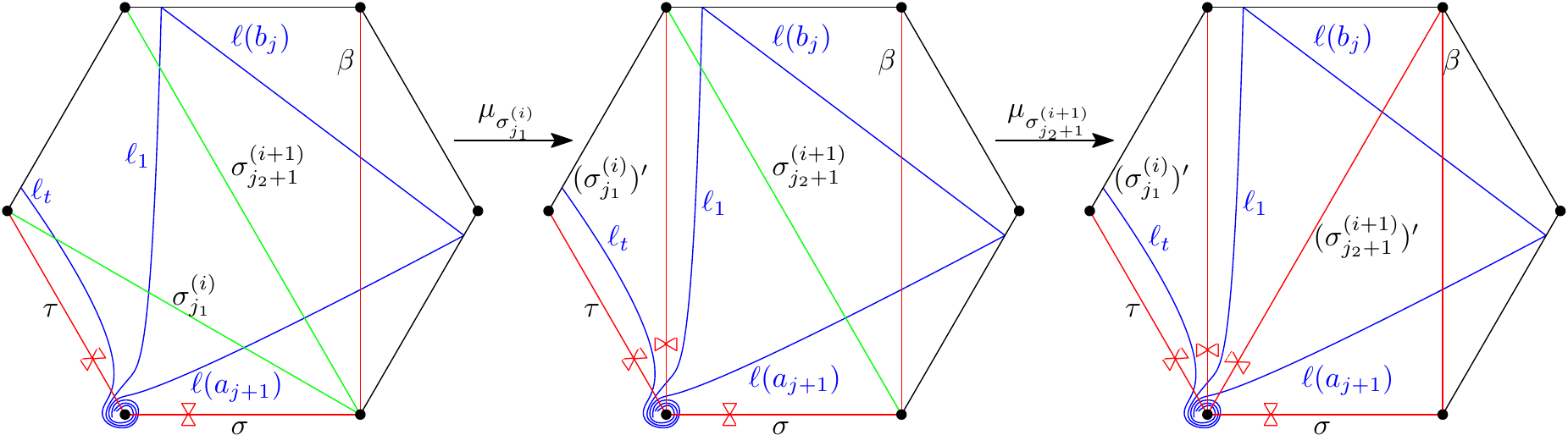} \\ (b) \end{array} $$
$$\begin{array}{ccccc} \includegraphics[scale=.8]{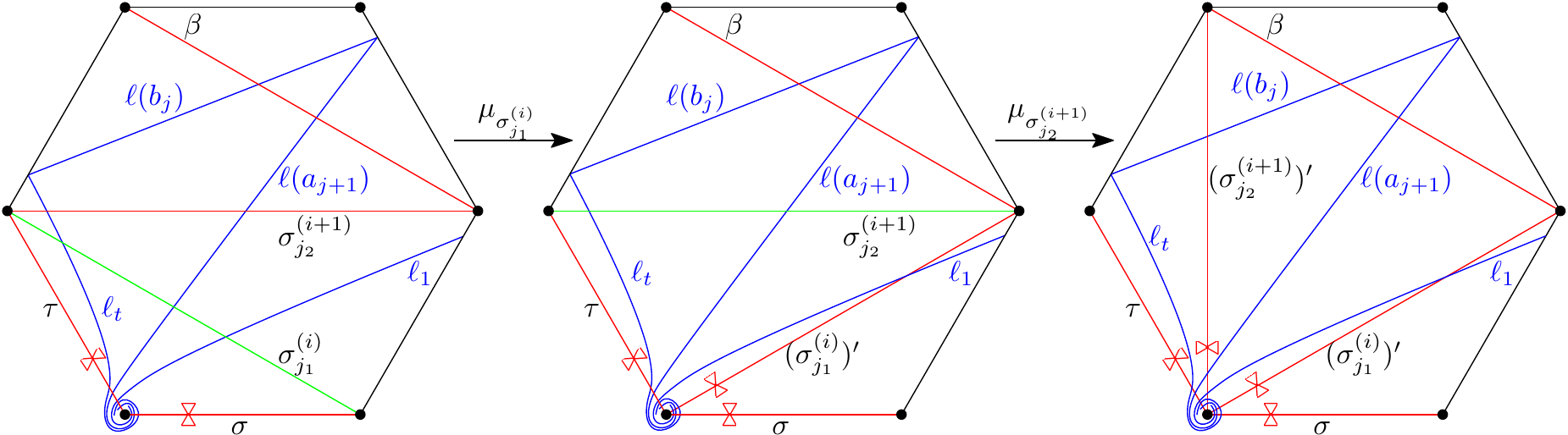} \\ (c) \end{array} $$
$$\begin{array}{ccccc} \includegraphics[scale=.8]{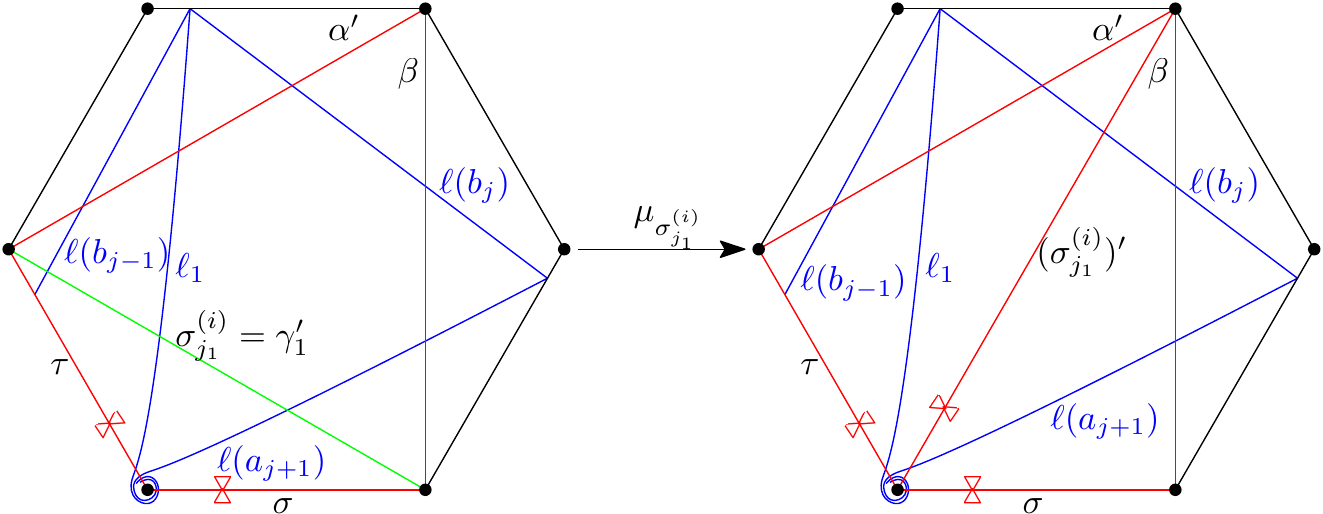} \\ (d) \end{array} $$
\caption{The types of flips performed by applying $\underline{\mu}_{\mathcal{S}_i}$ to $\underline{\mu}_{\mathcal{S}_{i-1}}\cdots \underline{\mu}_{\mathcal{S}_1}\underline{\mu}_4\underline{\mu}_3\underline{\mu}_2\underline{\mu}_1(\textbf{T})$.}
\label{i5_transformation}
\end{figure}


\begin{proof}[Proof of Corollary~\ref{Cor:Dn_short_seq_length}]
As in the proof of Theorem~\ref{Thm:Dn_short_seq}, we prove that $\textbf{i}_1\circ \textbf{i}_2 \circ \textbf{i}_3\circ \textbf{i}_4 \circ \textbf{i}_5$ has the desired length when at least one of the quivers $Q^{(i)}$ is the empty quiver or none of the quivers $Q^{(i)}$ is the empty quiver and $k$ is even. 

Observe that applying the sequence of flips $\underline{\mu}_2\underline{\mu}_1$ to $\textbf{T}$ flips each arc of $\textbf{T}$ exactly once. Thus $\textbf{i}_1\circ \textbf{i}_2$ is a green sequence of $Q_{\textbf{T}}$ with length $\#(Q_\textbf{T})_0 = k_\text{in} + k_\text{out}$. It remains to show that $\ell(\textbf{i}_3\circ \textbf{i}_4 \circ \textbf{i}_5) = k_\text{in} - 2 + m$. 

We first show that $\ell(\textbf{i}_3\circ \textbf{i}_4) = m$. We claim that there is a bijection $f: \text{supp}(\textbf{i}_3\circ \textbf{i}_4) \to \{a_i \in (Q_\textbf{T})_0: \ \text{deg}(a_i) = 4\}$. Let $f(\alpha_i) := v_{\gamma_1} \in (Q_\textbf{T})_0$ where $\alpha_i \in \text{supp}(\textbf{i}_3)$ and where $\gamma_1$ is the arc appearing in a triangle of $\underline{\mu}_1(\textbf{T})$ with $\alpha_i$ from the definition of $\textbf{i}_3.$ From the structure of $Q_\textbf{T}$, we have that $f(\alpha_i) \neq f(\alpha_j)$ for $i \neq j$ and that $\text{deg}(f(\alpha_i)) = 4$.

Let $f(\delta_i) := v_{\overline{\varrho^{-1}(\delta_i)}} \in (Q_{\textbf{T}})_0$ where $\delta_i \in \text{supp}(\textbf{i}_4)$. We need to show that $v_{\overline{\varrho^{-1}(\delta_i)}} \in (Q_\textbf{T})_0$. Observe that for each $\delta_i \in \text{supp}(\textbf{i}_4)$ there exists $b_j \in (Q_\textbf{T})_0$ such that $\delta_i = {b^\prime_j}.$ This implies $\overline{\varrho^{-1}(\delta_i)} = {a_j}$ and $f(\delta_i) \neq f(\delta_j)$ for $i \neq j$. We conclude that $f(\delta_i) = a_j \in (Q_\textbf{T})_0$ and that this vertex has degree 4.

It remains to show that $f$ is surjective. Observe that for any degree 4 vertex $a_i \in (Q_\textbf{T})_0$, either ${a_i}$ appears in the fan $({b_i}, {a_i},{b_{i-1}})$--one of the fans that defines $\textbf{i}_1$--or ${a_i}$ is an arc that gives rise to one of the arcs in $\textbf{i}_3$. As such, the vertex $a_i$ is either the image under $f$ of an arc in $\text{supp}(\textbf{i}_3)$ or of an arc in $\text{supp}(\textbf{i}_4)$.

Lastly, we show that $\ell(\textbf{i}_5) = k_{\text{in}} - 2$. By Theorem~\ref{Thm:Dn_short_seq} and Theorem~\ref{univ_tagged_thm}, we know that $\underline{\mu}_5\underline{\mu}_4\underline{\mu}_3\underline{\mu}_2\underline{\mu}_1(\textbf{T}) = \varrho(\textbf{T}).$ In particular, we know that $\varrho(\textbf{T})$ has exactly $k_\text{in}$ arcs connected to $p$ that are notched at $p$. In the triangulation $\underline{\mu}_4\underline{\mu}_3\underline{\mu}_2\underline{\mu}_1(\textbf{T})$ (see  Figures~\ref{i2_transformation_c_intermed} $(c_3)$ and \ref{i4_transformation}), there are exactly two arcs connected to $p$ that are notched at $p$. Since each flip in $\underline{\mu}_5$ produces a triangulation with exactly one additional arc connected to $p$ that is notched at $p$, we conclude that $\ell(\textbf{i}_5) = k_{\text{in}}-2$. \end{proof}

\section{Quivers of mutation type $\widetilde{\mathbb{A}}(n_1,n_2)$}\label{Sec_affine_A_quivers}

We now focus on the minimal length maximal green sequences of quivers that we refer to as \textbf{mutation type $\widetilde{\mathbb{A}}(n_1,n_2)$} where $n_1$ and $n_2$ are positive integers. Let $\widetilde{\mathbb{A}}_n$ denote an affine type $\mathbb{A}$ Dynkin diagram with $n$ vertices where $n \ge 1$. By definition, the mutation type $\widetilde{\mathbb{A}}(n_1,n_2)$ quivers are those that are mutation-equivalent to an acyclic orientation of $\widetilde{\mathbb{A}}_n$, denoted $Q(n_1,n_2)$, where $Q(n_1,n_2)$ has $n_1$ edges are oriented \textit{clockwise}, $n_2$ edges are oriented \textit{counterclockwise}, and $n = n_1+n_2-1$. Note that the notion of clockwise and counterclockwise arrows is only unique up to changing the roles of the clockwise and counterclockwise arrows. 

\subsection{The mutation class of quivers of mutation type $\widetilde{\mathbb{A}}(n_1,n_2)$ quivers} Here we recall the description of mutation type $\widetilde{\mathbb{A}}(n_1, n_2)$ quivers that was found in \cite{bastian2012mutation}. We first mention the following result, which shows that the mutation class of mutation type $\widetilde{\mathbb{A}}(n_1,n_2)$ quivers depends only on the choice of the parameters $n_1 \in \mathbb{N}$ and $n_2 \in \mathbb{N}$. 

\begin{lemma}\cite[Lemma 6.8]{fomin2008cluster}\label{Lem:2params}
Let $Q(n_1, n_2)$ and $Q(m_1, m_2)$ be two acyclic orientations of $\widetilde{\mathbb{A}}_n$. Then $Q(n_1,n_2)$ and $Q(m_1, m_2)$ are mutation-equivalent if and only if $n_1 = m_1$ and $n_2 = m_2$.
\end{lemma}


\begin{definition}\cite[Definition 3.3]{bastian2012mutation}\label{Qnquiversbastian}
Let $\mathcal{Q}_n$ be the class of connected quivers $Q$ with $n+1$ vertices that satisfy the following conditions:

\begin{itemize}
\item[i)] There exists precisely one full subquiver of $Q$ which is an acyclic orientation of $\widetilde{\mathbb{A}}_m$, denoted $Q(m_1,m_2)$, for some $m$ satisfying $1 \le m \le n$. 
\item[ii)] for each arrow $x \stackrel{\alpha}{\longrightarrow} y$ in $Q(m_1,m_2)$, there may be a vertex $z_\alpha \in Q_0\backslash Q(m_1,m_2)_0$ such that the 3-cycle $z_\alpha \longrightarrow x \stackrel{\alpha}{\longrightarrow} y \longrightarrow z_\alpha$ is a full subquiver of $Q$, and there are no other vertices in $Q_0\backslash Q(m_1,m_2)_0$ that are connected to vertices of $Q(m_1,m_2)$;
\item[iii)] the full subquiver $Q^\prime$ obtained from $Q$ by removing the vertices and arrows of $Q(m_1,m_2)$ consists of the quivers $\{Q^{(\alpha)}\}$, one for each $\alpha \in Q(m_1,m_2)_1$, and each nonempty quiver $Q^{(\alpha)}$ is of mutation type $\mathbb{A}$.
\end{itemize}
Let $Q$ be a quiver in $\mathcal{Q}_n$. We refer to the quiver $Q(m_1,m_2)$ as the \textbf{non-oriented cycle} of $Q$. Now let $x \stackrel{\alpha}{\longrightarrow} y$ be an arrow of its non-oriented cycle $Q(m_1, m_2)$. 
\end{definition}

\begin{example}
In Figure~\ref{affineexample}, we give an example of a quiver $Q$ belonging to $\mathcal{Q}_{16}.$ The full subquiver of $Q$ on the vertices 2, 3, 4, 5, 6, 7, 8, and 9 is the acyclic orientation $Q(3,5)$ of $\widetilde{\mathbb{A}}_7$ described in part i) of Definition~\ref{Qnquiversbastian}. The quiver $Q^\prime$ is the full subquiver of $Q$ on the vertices 1, 10, 11, 13, 12, 14, 15, and 16. It consists of a quiver of mutation type $\mathbb{A}_1$, a quiver of mutation type $\mathbb{A}_3$, and quiver of mutation type $\mathbb{A}_4$. 
\end{example}

\begin{figure}[h]
$$\scalebox{.75}{$\begin{xy} 0;<1pt,0pt>:<0pt,-1pt>:: 
(0,45) *+{1} ="0",
(30,90) *+{2} ="1",
(60,45) *+{3} ="2",
(120,45) *+{4} ="3",
(180,45) *+{5} ="4",
(210,90) *+{6} ="5",
(60,135) *+{7} ="6",
(120,135) *+{8} ="7",
(180,135) *+{9} ="8",
(150,0) *+{10} ="9",
(210,0) *+{11} ="10",
(240,45) *+{12} ="11",
(270,0) *+{13} ="12",
(240,135) *+{14} ="13",
(300,135) *+{15} ="14",
(270,90) *+{16} ="15",
"1", {\ar"0"},
"0", {\ar"2"},
"2", {\ar"1"},
"6", {\ar"1"},
"2", {\ar"3"},
"4", {\ar"3"},
"3", {\ar"9"},
"4", {\ar"5"},
"9", {\ar"4"},
"8", {\ar"5"},
"5", {\ar"13"},
"6", {\ar"7"},
"7", {\ar"8"},
"13", {\ar"8"},
"9", {\ar"10"},
"10", {\ar"11"},
"12", {\ar"10"},
"11", {\ar"12"},
"13", {\ar"14"},
"15", {\ar"13"},
"14", {\ar"15"},
\end{xy}$}$$
\caption{}
\label{affineexample}
\end{figure}


It is clear that any acyclic orientation of $\widetilde{\mathbb{A}}_n$ belongs to $\mathcal{Q}_n$. By \cite[Lemma 3.5]{bastian2012mutation}, $\mathcal{Q}_n$ is closed under mutation. Thus each quiver in $\mathcal{Q}_n$ is mutation equivalent to an acyclic orientation of $\widetilde{\mathbb{A}}_n$.


\begin{theorem}\label{Thm_affinequivers_min_length}
Let $\widetilde{Q}$ be a quiver of mutation type $\widetilde{\mathbb{A}}(n_1,n_2)$. The length of a minimal length maximal green sequence of $\widetilde{Q}$ is $|\widetilde{Q}_0| + |\{\text{3-cycles in } \widetilde{Q}\}|.$
\end{theorem}
\begin{proof}
If there exists a clockwise arrow $\alpha \in \widetilde{Q}(n_1, n_2)_1$ and counterclockwise arrow $\beta \in \widetilde{Q}(n_1,n_2)_1$ where $\widetilde{Q}^{(\alpha)}$ and $\widetilde{Q}^{(\beta)}$ are both the empty quivers, then $\widetilde{Q} = Q^1 \oplus_{(s(\alpha), s(\beta))}^{(t(\alpha), t(\beta))} Q^2$ for some quivers $Q^1$ and $Q^2$. It follows from Definition~\ref{Qnquiversbastian} that $Q^1$ and $Q^2$ are both of mutation type $\mathbb{A}$. By Theorem~\ref{thmtypeA}, there exist minimal length maximal green sequences $\textbf{i}_1 \in \text{MGS}(Q^1)$ and $\textbf{i}_2 \in \text{MGS}(Q^2)$ where $\ell(\textbf{i}_1) = |Q^1_0| + |\{\text{3-cycles in } Q^1\}|$ and $\ell(\textbf{i}_2) = |Q^2_0| + |\{\text{3-cycles in } Q^2\}|$. By Theorem~\ref{tcolordirsum} and Proposition~\ref{minlengthdirectsum}, $\textbf{i}_1\circ \textbf{i}_2 \in \text{MGS}(\widetilde{Q})$ and this is a minimal length maximal green sequence of the desired length.

Suppose there does not exist a clockwise arrow $\alpha \in \widetilde{Q}(n_1, n_2)_1$ and counterclockwise arrow $\beta \in \widetilde{Q}(n_1,n_2)_1$ where $\widetilde{Q}^{(\alpha)}$ and $\widetilde{Q}^{(\beta)}$ are both the empty quivers. By Definition~\ref{Qnquiversbastian}, the quiver $\widetilde{Q}$ is an example of a quiver described by Definition~\ref{quiver_qtilde} where the quivers $Q^1, Q^2,\ldots$ are the nonempty quivers in $\{Q^{(\alpha)}\}$ from Definition~\ref{Qnquiversbastian} and $Q$ is the full subquiver of $\widetilde{Q}$ on the vertex set $\widetilde{Q}(n_1,n_2)_0\cup \{z_\alpha: \alpha \in \widetilde{Q}(n_1,n_2)_1\}$ where the vertices $z_\alpha$ are those appearing in Definition~\ref{Qnquiversbastian}. 

The quivers $Q^1, Q^2, \ldots$ are all of mutation type $\mathbb{A}$ so by Theorem~\ref{thmtypeA} there exist minimal length maximal green sequences $\textbf{i}_j \in \text{MGS}(Q^j)$ where $\ell(\textbf{i}_j) = |Q^j_0| + |\{\text{3-cycles in } Q^j\}|$. By Theorem~\ref{affine_mgs}, there exist a minimal length maximal green sequence $\textbf{i}_Q \in \text{MGS}(Q)$ where $\ell(\textbf{i}_Q) = |Q_0| + |\{\text{3-cycles in } Q\}|$. Now by Theorem~\ref{qtilde}, the quiver $\widetilde{Q}$ has a minimal length maximal green sequence $\textbf{i}$ with $\ell(\textbf{i}) = |\widetilde{Q}_0| + |\{\text{3-cycles in } \widetilde{Q}\}|.$
\end{proof}

\subsection{Construction of minimal length maximal green sequences} In this section, we construct a particular mutation sequence of a quiver of mutation type $\widetilde{\mathbb{A}}(n_1,n_2)$ that will turn out to be a minimal length maximal green sequence of the quiver.

Let $Q$ be a quiver of mutation type $\widetilde{\mathbb{A}}(n_1,n_2)$. In this section, for brevity, we let $\eta(Q)$ be the non-oriented cycle of $Q$. Assume that $Q$ cannot be expressed as a nontrivial direct sum of two mutation type $\mathbb{A}$ quivers. This implies that there do not exist distinct arrows $i_1\stackrel{\alpha_1}{\to} j_1, i_2\stackrel{\alpha_2}{\to} j_2 \in \eta(Q)_1$ such that $Q = Q^{(1)} \oplus_{(i_1,i_2)}^{(j_1,j_2)} Q^{(2)}$ for some full subquivers $Q^{(1)}$ and $Q^{(2)}$ of $Q$.  In particular, it follows that all clockwise arrows of $\eta(Q)_1$ belong to a 3-cycle of $Q$ or all counterclockwise arrows of $\eta(Q)_1$ belong to a 3-cycle of $Q$. Without loss of generality, we will assume that all clockwise arrows of $\eta(Q)_1$ belong to a 3-cycle of $Q$.


\begin{lemma}\label{min_length_lower_bound_affine_A}
The length of any maximal green sequence of $Q$ is at least $|Q_0| + |\{\text{3-cycles in } Q\}|$.
\end{lemma}
\begin{proof}
We will assume that $Q$ is not the quiver of mutation type $\widetilde{\mathbb{A}}_4$ where $\eta(Q)$ is the Kronecker quiver with arrows $\alpha_1, \alpha_2 \in \eta(Q)_1$ and both of its quivers $Q^{(\alpha_1)}$ and $Q^{(\alpha_2)}$ are single vertices. We will also assume that $Q$ is not the quiver of mutation type $\widetilde{\mathbb{A}}_3$ where $\eta(Q)$ is the Kronecker quiver with arrows $\alpha_1, \alpha_2 \in \eta(Q)_1$ and where $Q^{(\alpha_1)}$ is a single vertex and $Q^{(\alpha_2)}$ is the empty quiver. If $Q$ is either of these two quivers, one checks that its minimal length maximal green sequences have the desired length. 

Now let $\textbf{i} \in \text{MGS}(Q)$ and let $\alpha \in \eta(Q)_1$ be an arrow in the 3-cycle $z_\alpha \to s(\alpha) \to t(\alpha) \to z_\alpha$. By Theorem~\ref{thmMullerSubquiver}, we let $\textbf{i}^{(\alpha)}$ denote the induced maximal green sequence on the full subquiver of $Q$ whose vertices are $s(\alpha), t(\alpha),$ and $z_\alpha$. One checks that the sequence of \textbf{c}-vectors $\textbf{c}(\textbf{i}^{(\alpha)})$ includes at least one $\textbf{c}$-vector with two nonzero entries. Thus the sequence $\textbf{c}(\textbf{i})$ includes at least one \textbf{c}-vector with two nonzero entries for each 3-cycle of $Q$. It follows that $\ell(\textbf{i}) \ge |Q_0| + |\{\text{3-cycles in } Q\}|$.
\end{proof}

Now assume that $Q$ has the property that the family of subquivers $\{Q^{(\alpha)}: \alpha \in \eta(Q)_1\}$ of $Q$, using the notation of Definition~\ref{Qnquiversbastian}, contains only empty quivers or quivers of mutation type $\mathbb{A}_1$. It immediately follows that $Q$ has the property that each arrow $\alpha$ of $\eta(Q)$ satisfies one of the following: \begin{itemize}
\item $\alpha$ belongs to a (necessarily, unique) 3-cycle of $Q$
\item $\alpha$ is not an arrow in an oriented cycle of $Q$.
\end{itemize} Let $Q=Q_I$ denote a quiver satisfying all of the assumptions mentioned in the previous two paragraphs.

We now define two families $C(Q)$ and $CC(Q)$ of full subquivers of $Q$ that will be important in our construction of minimal length maximal green sequences of $Q$. We note that the elements of $C(Q)$ and $CC(Q)$ will form a set partition of the arrows of $Q$. Define the \textbf{clockwise components} of $Q$, denoted $C(Q)$, to be the set of full, connected subquivers $R$ of $Q$ such that \begin{itemize}\item $R$ contains no counterclockwise arrows of $\eta(Q)$, \item each clockwise arrow of $R_1\cap\eta(Q)_1$ belongs to a 3-cycle, and \item $R$ is an inclusion-maximal full subquiver of $Q$ with respect to these conditions. \end{itemize}

Let $R \in C(Q)$ and order the clockwise arrows in $R_1 \cap \eta(Q)_1$ as $\alpha_1, \ldots, \alpha_d$ so that $\text{deg}(s(\alpha_1)) = \text{deg}(t(\alpha_d)) = 2$ and $t(\alpha_i) = s(\alpha_{i+1})$ for each $i \in [d-1].$ Define $\textbf{i}_{R,initial}$ to be the following sequence of vertices of $R$: $$\begin{array}{rcl} \textbf{i}_{R,initial} & := & (z_{\alpha_1}, z_{\alpha_2}, s(\alpha_2), z_{\alpha_3}, s(\alpha_3), \ldots, z_{\alpha_d}, s({\alpha_d})).\end{array}$$ 

Next, define the \textbf{counterclockwise components} of $Q$, denoted $CC(Q)$, to be the set of full, connected subquivers $S$ of $Q$ such that \begin{itemize}\item $S$ contains no clockwise arrows of $\eta(Q)$ and \item $S$ is an inclusion-maximal full subquiver of $Q$ with respect to these conditions. \end{itemize}

Let $S \in CC(Q)$ and order the counterclockwise arrows in $S_1 \cap \eta(Q)_1$ as $\alpha_1, \ldots, \alpha_d$ so that $s(\alpha_1)$ and $t(\alpha_d)$ are vertices of some quivers belonging to $C(Q)$ and $t(\alpha_i) = s(\alpha_{i+1})$ for each $i \in [d-1]$. Define $\textbf{i}_{S,initial}$ to be the following sequence of vertices of $S$: $$\begin{array}{rcl}\textbf{i}_{S,initial} & := &(s(\alpha_1), s(\alpha_2), \ldots, s(\alpha_d), t(\alpha_d)).\end{array}$$ 

Let $\textbf{i}_{S, final}$ be a sequence of some distinct vertices of $S$, first containing all of the vertices in $S_0\setm\eta(Q)_0$, followed by the vertices $i$ of $S_0\cap\eta(Q)_0$ for which there exists an arrow $j\ra i$ where $j$ is in $S_0\setm\eta(Q)_0$. We exclude $i$ from this sequence if there exists an arrow $i\ra j$ and $j$ is part of some clockwise component of $Q$.

Now let $R\in C(Q)$. Let $i_1,\ldots,i_l$ be the vertices of $R_0\setm\eta(Q)_0$, ordered such that there exists vertices $j_k$ with a 2-path $i_{k+1}\ra j_k\ra i_k$ in $R$ for $k<l$. Let $j_l$ be the unique vertex in $R$ with an arrow $j_l\ra i_l$. Then $j_l$ is part of some counterclockwise component $S$. Let $j$ be the vertex of $\eta(Q)_0\cap S_0$ with an arrow $j\ra j_l$. If this arrow is part of a 3-cycle, then we define $\textbf{i}_{R,final}$ to be the sequence $(i_1,\ldots,i_l,j)$. Otherwise, we let $\textbf{i}_{R,final}$ be the sequence $(i_1,\ldots,i_l)$.

\begin{example}\label{vertex_seq}
Consider the quiver $Q$ shown in Figure~\ref{subquivers}. The set $C(Q) = \{R_1, R_2, R_3\}$ (resp., $CC(Q) = \{S_1, S_2, S_3\}$) consists of the 3 full subquivers of $Q$ that have solid (resp., dashed) arrows. For the quiver $Q$, the sequences of vertices defined above are as follows.

$$\begin{array}{rclccrcl}
\textbf{i}_{R_1, initial} & = & (1,2,3) & & & \textbf{i}_{S_1, final} & = & (23, 24, 25, 12, 13)\\
\textbf{i}_{R_2, initial} & = & (4) & & & \textbf{i}_{S_2, final} & = & (26) \\
\textbf{i}_{R_3, initial} & = & (5, 6, 7, 8, 9) & & & \textbf{i}_{S_3, final} & = & (27, 20) \\
\textbf{i}_{S_1, initial} & = & (10, 11, 12, 13, 14, 15, 16)  & & & \textbf{i}_{R_1, final} & = & (1, 2, 17) \\
\textbf{i}_{S_2, initial} & = & (17, 18) & & & \textbf{i}_{R_2, final} & = & (4) \\
\textbf{i}_{S_3, initial} & = & (19, 20, 21, 22) & & &  \textbf{i}_{R_3, final} & = & (5, 6, 8, 15)\\
\end{array}$$
\end{example}

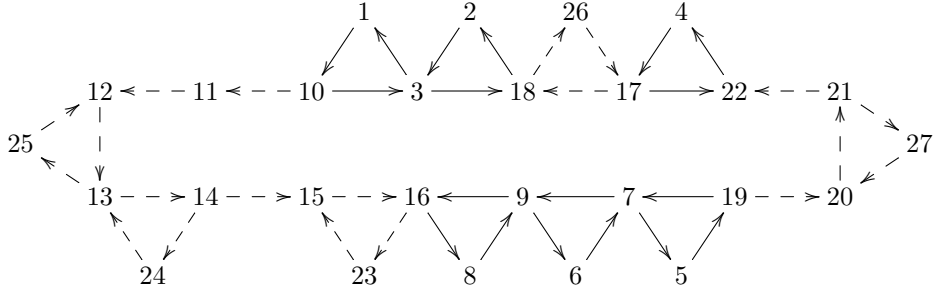
\begin{figure}
$$\begin{xy} 0;<1pt,0pt>:<0pt,-1pt>:: 
(130,0) *+{1} ="1",
(150,30) *+{3} ="2",
(170,0) *+{2} ="3",
(250,0) *+{4} ="4",
(250,100) *+{5} ="5",
(210,100) *+{6} ="6",
(230,70) *+{7} ="7",
(170,100) *+{8} ="8",
(190,70) *+{9} ="9",
(110,30) *+{10} ="10",
(70,30) *+{11} ="11",
(30,30) *+{12} ="12",
(30,70) *+{13} ="13",
(70,70) *+{14} ="14",
(110,70) *+{15} ="15",
(150,70) *+{16} ="16",
(270,70) *+{19} ="17",
(310,70) *+{20} ="18",
(310,30) *+{21} ="19",
(230,30) *+{17} ="20",
(190,30) *+{18} ="21",
(130,100) *+{23} ="22",
(50,100) *+{24} ="23",
(0,50) *+{25} ="24",
(210,0) *+{26} ="25",
(340,50) *+{27} ="26",
(270,30) *+{22} ="27",
"2", {\ar"1"},
"1", {\ar"10"},
"3", {\ar"2"},
"10", {\ar"2"},
"2", {\ar"21"},
"21", {\ar"3"},
"4", {\ar"20"},
"27", {\ar"4"},
"7", {\ar"5"},
"5", {\ar"17"},
"6", {\ar"7"},
"9", {\ar"6"},
"7", {\ar"9"},
"17", {\ar"7"},
"8", {\ar"9"},
"16", {\ar"8"},
"9", {\ar"16"},
"10", {\ar@{-->}"11"},
"11", {\ar@{-->}"12"},
"12", {\ar@{-->}"13"},
"24", {\ar@{-->}"12"},
"13", {\ar@{-->}"14"},
"23", {\ar@{-->}"13"},
"13", {\ar@{-->}"24"},
"14", {\ar@{-->}"15"},
"14", {\ar@{-->}"23"},
"15", {\ar@{-->}"16"},
"22", {\ar@{-->}"15"},
"16", {\ar@{-->}"22"},
"17", {\ar@{-->}"18"},
"18", {\ar@{-->}"19"},
"26", {\ar@{-->}"18"},
"19", {\ar@{-->}"26"},
"19", {\ar@{-->}"27"},
"20", {\ar@{-->}"21"},
"25", {\ar@{-->}"20"},
"20", {\ar"27"},
"21", {\ar@{-->}"25"},
\end{xy}$$
\caption{The subquivers $C(Q)$ and $CC(Q)$.}
\label{subquivers}
\end{figure}


\begin{remark}
There is an obvious bijection between the quivers in $C(Q)$ and those in $CC(Q)$ given by sending $R \in C(Q)$ to $S \in CC(Q)$ so that $s(\alpha^R_1) = s(\alpha^S_1)$ is satisfied. Thus $|C(Q)| = |CC(Q)| = |\{\text{sources of $\eta(Q)$}\}|$. We will henceforth write $C(Q) = \{R_1, \ldots, R_k\}$ and $CC(Q) = \{S_1, \ldots, S_k\}$ so that $s(\alpha^{R_i}_1) = s(\alpha^{S_i}_1)$ for each $i \in [k],$ as we have done in Example~\ref{vertex_seq}.
\end{remark}

We define $\textbf{i} = \textbf{i}_1 \circ \textbf{i}_2 \circ \textbf{i}_3 \circ \textbf{i}_4$ to be the following sequence of vertices.

$$\begin{array}{rcl}
\textbf{i}_{1} & := & \textbf{i}_{R_1,initial}\circ \cdots \circ \textbf{i}_{R_k, initial} \\
\textbf{i}_{2} & := & \textbf{i}_{S_1,initial} \circ \cdots \circ \textbf{i}_{S_k, initial} \\
\textbf{i}_{3} & := & \textbf{i}_{S_1, final} \circ \cdots \circ \textbf{i}_{S_k, final} \\
\textbf{i}_{4} & := & \textbf{i}_{R_1, final} \circ \cdots \circ \textbf{i}_{R_k, final} \\
\end{array}$$

\begin{theorem}\label{affine_mgs}
The sequence $\textbf{i}$ is a minimal length maximal green sequence of $Q$ with $\ell(\textbf{i}) = |Q_0| + |\{\text{3-cycles in } Q\}|$.
\end{theorem}





The remainder of this section is dedicated to proving Theorem~\ref{affine_mgs}. As usual, we let $\ov{Q^\prime} \in \text{Mut}(\widehat{Q})$ denote a quiver with frozen vertices whose mutable part is $Q^\prime \in \text{Mut}(Q)$. Given two full subquivers $Q^1$ and $Q^2$ of some quiver $Q^\prime \in \text{Mut}(Q)$, we define $Q^1\cap Q^2$ (resp., $Q^1\backslash Q^2$) to be the full subquiver of $Q^\prime$ whose vertex set is $Q^1_0\cap Q^2_0$ (resp., $Q^1_0\backslash Q^2_0$). If $R$ is a clockwise or counterclockwise component of $Q$, we refer to the unique source (resp., sink) of $\eta(Q)\cap R$ as the \textbf{first} (resp., \textbf{last}) vertex of $R$. Similarly, we define the $\textbf{first}$ (resp., \textbf{last}) arrow of $R$ to be the unique arrow of $\eta(Q)\cap R$ incident to the first (resp., last) vertex of $R$.

Let us first remark that the sequence is of the appropriate length. Every vertex except the first and last in a clockwise component of $Q$ is in the sequence $\textbf{i}_1$. The sequence $\textbf{i}_2$ contains all of the vertices in the main cycle $\eta(Q)$ that lie in a counterclockwise component of $Q$. Together with $\textbf{i}_1$, this now includes all of the vertices in $\eta(Q)$ and the vertices not in $\eta(Q)$ that are in a clockwise component. The sequence $\textbf{i}_3$ contains the vertices in each counterclockwise component of $Q$ outside the main cycle, as well as one additional vertex for each 3-cycle in a counterclockwise component unless the 3-cycle contains the last vertex of that component. Finally, for each 3-cycle in a clockwise component, there is one corresponding vertex in $\textbf{i}_4$, as well as one vertex in the sequence for each 3-cycle in a counterclockwise component that contains the last vertex of that component. In total, there is a mutation for each vertex of $Q$ and an additional mutation corresponding to each 3-cycle.

To show that this sequence is a maximal green sequence, we characterize the quiver obtained from $Q$ by applying mutations at each of the sequences $\textbf{i}_1\circ\cdots\circ\textbf{i}_l$ for $l\in\{1,2,3,4\}$. We say a quiver is of Type I if it satisfies the defining criteria for $\widehat{Q}$. We say a quiver $\ov{Q}_{II}$ is of Type II if the following holds.

\begin{itemize}
\item For each $R\in C(Q_{II})$:
\begin{itemize}
\item $R$ contains at least three vertices in $\eta(Q_{II})$.
\item The first and last arrows of $R\cap\eta(Q_{II})$ are not part of a 3-cycle, and every other arrow of $R\cap\eta(Q_{II})$ is part of a 3-cycle.
\item If $i$ is a vertex in $R\setm\eta(Q_{II})$, then there is a unique arrow $i\ra j$ with source $i$, and $i$ is a red vertex with arrows $i\la i^{\pr}$ and $i\la j^{\pr}$.
\item If $i$ is the last vertex in $R\cap\eta(Q_{II})$, then there is a unique arrow $j\ra i$ in $R$ with target $i$, and $i$ is a green vertex with arrows $i\ra i^{\pr},\ i\ra j^{\pr}$.
\item If $i$ is the second-to-last vertex in $R\cap\eta(Q_{II})$, then $i$ is red with an arrow $i\la i^{\pr}$.
\end{itemize}

\item Every other vertex is green with an arrow $i\ra i^{\pr}$.

\end{itemize}

\begin{lemma}\label{lem_T2}
Applying the mutation sequence $\mu_{\textbf{i}_1}$ to $\widehat{Q}$ gives a quiver of Type II, up to a permutation of the mutable vertices.
\end{lemma}

\begin{proof}
Let $\ov{Q}_{II}=\mu_{\textbf{i}_1}(\widehat{Q})$. Except for the endpoints, the vertices of any counterclockwise component are not adjacent to any vertex in the sequence $\textbf{i}_1$. Consequently, they will remained unchanged in $\ov{Q}_{II}$.

Let $R$ be a clockwise component of $Q$. After mutating at each of the vertices in $R\setm\eta(Q)$, the resulting quiver $R^{\pr}$ is a directed path with a clockwise orientation. The vertices in this path alternate between $R\cap\eta(Q)$ and $R\setm\eta(Q)$. For each vertex $i\in R\setm\eta(Q)$, there is an arrow $i\la i^{\pr}$ in $\ov{R^{\pr}}$. For a vertex $i\in R\cap\eta(Q)$ and arrow $j\ra i$ in $R^{\pr}$, there are arrows $i\ra i^{\pr}$ and $i\ra j^{\pr}$. Finally, if $i$ is the first vertex of $R\cap\eta(Q)$, then it is green with $i\ra i^{\pr}$.

Now let $R^{\pr\pr}$ be the result of mutating $R^{\pr}$ at each of the vertices in $R\cap\eta(Q)$ except the first and last. Then $R^{\pr\pr}$ is a clockwise component of $Q_{II}$. Each vertex of $R\cap\eta(Q)$ except the first and last is part of a 3-cycle in $R^{\pr\pr}$ but lies in $R^{\pr\pr}\setm\eta(Q_{II})$. Every other vertex of $R$ lies on the main cycle in $Q_{II}$. Since we assumed that every clockwise arrow in $\eta(Q)$ is part of a 3-cycle in $Q$, we deduce that $R^{\pr\pr}\cap\eta(Q_{II})$ contains at least 3 vertices. Furthermore, every arrow of $R^{\pr\pr}\cap\eta(Q_{II})$ except the first and last is part of a 3-cycle.

If $i$ is a vertex in $R\cap\eta(Q)$ besides the first and last vertex, then there is a unique arrow $i\ra j$ in $R^{\pr\pr}$ with source $i$. For each such pair $i,j$ we swap their labels in the mutable part of $\ov{Q}_{II}$. After doing this permutation, it is straight-forward to check that the component $R^{\pr\pr}$ satisfies the remaining conditions to be a clockwise component of a quiver of Type II.
\end{proof}

We say a quiver $\ov{Q}_{III}$ is of Type III if the following holds.

\begin{itemize}
\item For each $R\in C(Q_{III})$:
\begin{itemize}
\item Every arrow in $R$ is part of a 3-cycle.
\item If $i$ is the last vertex in $R\cap\eta(Q_{III})$, then $i$ is red with arrow $i\la i^{\pr}$.
\item If $i$ is any other vertex in $R\cap\eta(Q_{III})$, then $i$ is green with arrow $i\ra i^{\pr}$.
\item If $i$ is in $R\setm\eta(Q_{III})$, and $i\ra j$ is the unique arrow with source $i$, then $i$ is red with arrows $i\la i^{\pr}$ and $i\la j^{\pr}$.
\end{itemize}
\item For each $S\in CC(Q_{III})$:
\begin{itemize}
\item If $i$ is the first vertex in $S\cap\eta(Q_{III})$, then $i$ is green with arrow $i\ra i^{\pr}$.
\item If $i$ is any other vertex in $S\cap\eta(Q_{III})$, then $i$ is red with arrow $i\la i^{\pr}$.
\item If $i$ is in $S\setm\eta(Q_{III})$, and $j\ra i$ is the unique arrow with target $i$, then $i$ is green with arrows $i\ra i^{\pr}$ and $i\ra j^{\pr}$.
\end{itemize}
\end{itemize}

\begin{lemma}\label{lem_T3}
Applying $\mu_{\textbf{i}_2}$ to a quiver of Type II produces a quiver of Type III, up to a permutation of the mutable vertices.
\end{lemma}

\begin{proof}
Let $\ov{Q}_{II}$ be a quiver of Type II. Let $S$ be a counterclockwise component of $Q_{II}$. Let $i_0,i_1,\ldots,i_l$ be the list of vertices in $S\cap\eta(Q_{II})$ such that there exist arrows $i_0\ra i_1\ra\cdots\ra i_l$. Let $i_0\ra i_{-1}$ be the unique arrow outside $S$ whose source is $i_0$, and $i_l\la i_{l+1}$ be the unique arrow outside $S$ whose target is $i_l$. Since $Q_{II}$ is of Type II, both of the arrows $i_0\ra i_{-1}$ and $i_l\la i_{l+1}$ are not part of a 3-cycle. In particular, mutation at the vertices $i_0,\ldots,i_l$ only affects $i_{-1},\ i_{l+1}$, and the vertices in $S$.

Let $\ov{Q}_{III}=\mu_{\textbf{i}_2}(\ov{Q}_{II})$. Let $S^{\pr}$ be the counterclockwise component of $Q_{III}$ containing $i_0$. Then $S^{\pr}\cap\eta(Q_{III})$ is the directed path $i_{-1}\ra i_0\ra\cdots\ra i_{l-1}$. Furthermore, $S^{\pr}$ contains a 3-cycle $i_{l+1}\ra i_{l-1}\ra i_l\ra i_{l+1}$. In particular, any arrow in a clockwise component of $Q_{III}$ is part of a 3-cycle.

There is a 3-cycle $i_k\ra i_{k+1}\ra j\ra i_k$ in $S$ if and only if there is a 3-cycle $i_{k-1}\ra i_k\ra j\ra i_{k-1}$ in $S^{\pr}$.

Swap the labels on the mutable vertices $i_l$ and $i_{l+1}$ in $Q_{III}$. Once again, it is routine to check that $\ov{Q}_{III}$ satisfies the remaining conditions of a quiver of Type III.\end{proof}


Let $\ov{Q}_{III}$ be a Type III quiver. We define a mutation sequence $\mu_{\textbf{j}_3}$ on some of the vertices of the counterclockwise components of $\ov{Q}_{III}$ such that for each $S\in CC(Q_{III})$, the restriction of $\textbf{j}_3$ to $S$ consists of the vertices in $S\setm\eta(Q_{III})$, followed by the vertices $j$ in $S\cap\eta(Q_{III})$ such that $i\ra j$ for some $i\in S\setm\eta(Q_{III})$. If $j$ is the last vertex of $S\cap\eta(Q_{III})$, then it is removed from the list $\textbf{j}_3$.

We say a quiver $\ov{Q}_{IV}$ is of Type IV if the following holds.

\begin{itemize}
\item The clockwise components satisfy most of the same rules as a Type III quiver. The only exception is that if $i$ is the last vertex of some clockwise component, then either $i$ is red with an arrow $i\la i^{\pr}$ or it is green with an arrow $i\ra i^{\pr}$.

\item For $S\in CC(Q_{III})$:
\begin{itemize}
\item The first arrow of $S\cap\eta(Q_{IV})$ is not part of a 3-cycle.
\item If $i$ is the first vertex in $S\cap\eta(Q_{IV})$, then it is green with arrow $i\ra i^{\pr}$.
\item If $i$ is the last vertex in $S\cap\eta(Q_{IV})$, then it is either green with arrow $i\ra i^{\pr}$ or red with arrow $i\la i^{\pr}$. Furthermore, if $i$ is green, then the last arrow $j\ra i$ is not part of a 3-cycle, and $j$ is red with arrows $j\la i^{\pr}$ and $j\la j^{\pr}$.
\item Every other vertex in $S$ is red with arrow $i\la i^{\pr}$.
\end{itemize}
\end{itemize}

\begin{lemma}\label{lem_T4}
Applying the mutation sequence $\mu_{\textbf{j}_3}$ to a quiver of Type III produces a quiver of Type IV, up to a permutation of the mutable vertices.
\end{lemma}

\begin{proof}
Let $\ov{Q}_{III}$ be a quiver of Type III. It is clear that the mutation sequence does not affect any clockwise component with the possible exception of the last vertex.

Let $S$ be a counterclockwise component of $Q_{III}$. Let $S^{\pr}$ be the result of mutating at each of the vertices in $S\setm\eta(Q_{III})$. Then $S^{\pr}$ is a directed path. If $i$ is a vertex of $S\setm\eta(Q_{III})$ and $j$ is the unique vertex in $S$ with $j\ra i$, then $i$ is red in $S^{\pr}$ with arrows $i\la i^{\pr}$ and $i\la j^{\pr}$. Furthermore, $j$ is green in $S^{\pr}$ with an arrow $j\ra i^{\pr}$. All other vertices of $S^{\pr}$ are still red.

Let $S^{\pr\pr}$ be the result of mutating at the remaining vertices in the mutation sequence $\mu_{\textbf{j}_3}$. Let $i$ and $j$ be the same vertices as before, and suppose $j$ is not the last vertex in $S$. Then $j$ is red in $S^{\pr\pr}$ with an arrow $j\la i^{\pr}$ and $i$ is red in $S^{\pr\pr}$ with an arrow $i\la j^{\pr}$. If $j$ is the last vertex, then $i$ and $j$ remain as they were in $S^{\pr}$.

Up to swapping some labels $i,j$ in the mutable part of $\ov{Q}_{III}$, the result of applying $\mu_{\textbf{j}_3}$ to $\ov{Q}_{III}$ is a quiver of Type IV.
\end{proof}

Given a quiver $\ov{Q}_{IV}$ of Type IV, we define a mutation sequence $\mu_{\textbf{j}_4}$ as follows. For each clockwise component $R$, if $\eta(Q_{IV})\cap R$ is $i_1\ra\cdots\ra i_l$ then the restriction of $\mu_{\textbf{j}_4}$ to $R$ is $\mu_{i_1}\cdots\mu_{i_{l-1}}$ if $i_l$ is red and $\mu_{i_1}\cdots\mu_{i_l}$ if $i_l$ is green.

With this setup, we have the following result whose proof we omit.

\begin{lemma}
Applying the mutation sequence $\mu_{\textbf{j}_4}$ to a quiver of Type IV produces a quiver whose vertices are all red.
\end{lemma}

Since the mutation sequences $\mu_{\textbf{j}_3}$ and $\mu_{\textbf{j}_4}$ are the same as $\mu_{\textbf{i}_3}$ and $\mu_{\textbf{i}_4}$ up to relabeling mutable vertices as done in the proofs of Lemmas~\ref{lem_T2}, \ref{lem_T3}, and \ref{lem_T4}, we have now proved Theorem~\ref{affine_mgs}.

\section{Additional Remarks}\label{Sec_conc}

Any quiver $Q$ in the mutation class of an acyclic quiver gives rise to a finite dimensional $\Bbbk$-algebra, denoted $\Bbbk Q/I$, known as a \textbf{cluster-tilted algebra} (as defined in \cite{bmr1}). The quivers of mutation types $\mathbb{A}_n$, $\mathbb{D}_n$, and $\widetilde{\mathbb{A}}_{n}$, which we have considered in this paper, all define cluster-tilted algebras of the corresponding type. By referring to the known results on derived equivalence of cluster-tilted algebras of these types, we suspect that the length of minimal length maximal green sequences is constant on derived equivalence classes. We now briefly summarize these known results.

Let $\Bbbk Q^1/I^1$ and $\Bbbk Q^2/I^2$ be two derived equivalent cluster-tilted algebras. It follows from \cite[Theorem 5.1]{bv08} that $\Bbbk Q^1/I^1$ and $\Bbbk Q^2/I^2$ are of type $\mathbb{A}_n$ if and only if the minimal length maximal green sequences of $Q^1$ have the same length as those of $Q^2$. It follows from \cite[Theorem 1.1]{bastian2012mutation} that if $\Bbbk Q^1/I^1$ and $\Bbbk Q^2/I^2$ are of type $\widetilde{\mathbb{A}}_{n}$, then they have the same number of 3-cycles. Thus the minimal length maximal green sequences of $Q^1$ will have the same length as those of $Q^2$.

In the case of cluster-tilted algebras of type $\mathbb{D}_n$, there is no complete derived equivalence classification. From \cite[Theorem 2.3]{bastian2014towards}, any cluster-tilted algebra $\Bbbk Q^1/I^1$ of type $\mathbb{D}_n$ is derived equivalent to an algebra $\Bbbk Q/I$ whose quiver $Q$ is in one of six classes of mutation type $\mathbb{D}_n$ quivers. Bastian, Holm, and Ladkani call these six classes \textbf{standard forms} and denote them by $(a), (b), (c), (d_1), (d_2), (d_3)$. Any two distinct standard forms which are not of the class $(d_3)$ are not derived equivalent.

If $\Bbbk Q^1/I^1$ is not self-injective, it may be put into its standard form by mutations that the authors call \textbf{good mutations} and \textbf{good double mutations}. Using Theorem~\ref{Thm:typeD}, one checks that the length of minimal length maximal green sequences is invariant under good mutations and good double mutations. If $\Bbbk Q^1/I^1$ is self-injective, then \cite[Lemma 4.5]{bastian2014towards} and Theorem~\ref{Thm:typeD}, shows that the quiver of the standard form of $\Bbbk Q^1/I^1$ has minimal length maximal green sequences with the same length as those of $Q^1.$ These observations suggest the following question.

\begin{question}
Suppose $\Bbbk Q^1/I^1$ and $\Bbbk Q^2/I^2$ are two derived equivalent cluster-tilted algebras where $\textbf{i}_1 \in \text{MGS}(Q^1)$ and $\textbf{i}_2 \in \text{MGS}(Q^2)$ are minimal length maximal green sequences. Is it true that $\ell(\textbf{i}_1) = \ell(\textbf{i}_2)$?
\end{question}





\bibliographystyle{plain}
\bibliography{bib_minlength}

\end{document}